\documentclass[12pt]{amsart}  

\usepackage{custom_preamble} 

\usepackage{xcolor}
\usepackage[normalem]{ulem}

\newenvironment{oldversion}
  {\color{red}}
  {}

\newenvironment{newversion}
  {\color{blue}}
  {}

\begin{document}

\title[Quantitative version of non-abelian idempotent theorem]{A quantitative version of the non-abelian idempotent theorem}

\author{\tsname}
\address{\tsaddress}
\email{\tsemail}

\begin{abstract}
Suppose that $G$ is a finite group and $f$ is a complex-valued function on $G$.  $f$ induces a (left) convolution operator from $L^2(G)$ to $L^2(G)$ by $g \mapsto f \ast g$ where
\begin{equation*}
 f \ast g(z):=\E_{xy=z}{f(x)g(y)} \textrm{ for all } z \in G.
\end{equation*}
This operator is a linear map $L^2(G) \rightarrow L^2(G)$ between two finite dimensional Hilbert spaces, and so it has well-defined singular values; we write $\|f\|_{A(G)}$ for their sum.

The quantity $\|\cdot\|_{A(G)}$ is of particular interest because in the abelian setting it coincides with the $\ell^1$-norm of the Fourier transform of $f$.  Thus, in the abelian setting it is an algebra norm, and it turns out that this extends to the non-abelian setting as well when $\|\cdot\|_{A(G)}$ is defined as above.

It is relatively easy to see that if $A:=xH$ where $H \leq G$ and $x \in G$, then $\|1_{A}\|_{A(G)} = 1$, so that indicator functions of cosets of subgroups have algebra norm $1$.  Since $\|\cdot\|_{A(G)}$ is a norm we can easily construct other sets whose indicator functions have small algebra norm by taking small integer-valued sums of indicator functions of cosets (when these sums are themselves indicator functions of sets); the object of this paper is to show the following converse.

Suppose that $A \subset G$ has $\|1_A\|_{A(G)} \leq M$. Then there is an integer $L=L(M)$, subgroups $H_1,\dots,H_L \leq G$, elements $x_1,\dots,x_L \in G$ and signs $\sigma_1,\dots,\sigma_L \in \{-1,0,1\}$ such that
\begin{equation*}
1_A=\sum_{i=1}^L{\sigma_i1_{x_iH_i}},
\end{equation*}
where $L$ may be taken to be at most triply tower in $O(M^2)$.  This may be seen as a quantitative version of the non-abelian idempotent theorem.
\end{abstract}

\maketitle

\section{Introduction}

Suppose that $G$ is a finite group and $f,g \in L^1(\mu_G)$, where $\mu_G$ denotes the unique Haar probability measure on $G$.  The convolution $f \ast g$ of $f$ and $g$ is then defined point-wise by
\begin{equation*}
f \ast g(x):=\int{f(y)g(y^{-1}x)d\mu_G(y)}.
\end{equation*}
This can be used to introduce the family of convolution operators: given $f \in L^1(\mu_G)$ we define $L_f \in \End(L^2(\mu_G))$ via
\begin{equation*}
L_f:L^2(\mu_G) \rightarrow L^2(\mu_G); v \mapsto f \ast v.
\end{equation*}
In this paper we are interested in the algebra norm which is defined by
\begin{equation*}
\|f\|_{A(G)} := \sup\{|\langle f,g\rangle_{L^2(\mu_G)}|: \|L_g\| \leq 1\},
\end{equation*}
where $\|\cdot\|$ denotes the operator norm.

It is easy to check that this is, indeed, a norm, but to help understand it, we consider the case when $G$ is abelian.  Here we write $\wh{G}$ for the dual group, that is the finite abelian group of homomorphisms $G \rightarrow S^1$. The Fourier transform is then the map taking $f \in L^1(\mu_G)$ to $\wh{f} \in \ell^\infty(\wh{G})$ defined by
\begin{equation*}
\wh{f}(\gamma):=\langle f,\gamma\rangle_{L^2(\mu_G)} = \int{f(x)\overline{\gamma(x)}d\mu_G(x)}.
\end{equation*}
The elements of $\wh{G}$ form an orthonormal basis for $L^2(\mu_G)$ so the map $L_f$ takes $v$ to 
\begin{equation*}
\sum_{\gamma \in \wh{G}}{\wh{f}(\gamma)\langle v,\gamma\rangle_{L^2(\mu_G)}\gamma}.
\end{equation*}
It is then easy to see that the operator norm of this operator is $\|\wh{f}\|_{\ell^\infty(\wh{G})}$. Now, by Parseval's theorem we have
\begin{equation*}
\|f\|_{A(G)} = \sup\{|\langle \wh{f},\wh{v}\rangle_{\ell^2(\wh{G})}|:v \in L^1(\mu_G) \textrm{ and }\|\wh{v}\|_{\ell^\infty(\wh{G})}\leq 1\} ,
\end{equation*}
and so by the Fourier inversion formula and the dual characterisation of the $\ell^1$-norm we get that
\begin{equation*}
\|f\|_{A(G)} =\sum_{\gamma \in \wh{G}}{|\wh{f}(\gamma)|}.
\end{equation*}
Thus our definition of the $A(G)$-norm coincides with the usual one when $G$ is abelian.  There are many basic properties of the $A(G)$-norm which can be arrived at from the above expression in the case when $G$ is abelian, but require a little more work in the non-abelian case.  These are developed in detail in \S\ref{sec.algnormprop}.

Remaining with $G$ abelian we consider some examples.  If $H \leq G$ then it is easy to compute its Fourier transform:
\begin{equation*}
\wh{1_H}(\gamma)=\begin{cases} \mu_G(H) & \textrm{ if } \gamma(x)=1 \textrm{ for all }x \in H\\ 0 & \textrm{ otherwise.} \end{cases}
\end{equation*}
It follows that $\|1_H\|_{A(G)} =1$.  Moreover, the $A(G)$-norm is easily seen to be translation invariant so we conclude that $1_{A}$ has algebra norm $1$ whenever $A$ is a coset.  It turns out that the same is true for general $G$ (the details may be found in Corollary \ref{cor.indsmall}), but it is particularly easy to see when $G$ is abelian.

At the other end of the spectrum we have highly unstructured -- random -- sets: suppose that $A$ is a set of $k$ independent elements of $G$.  Then by Kinchine's inequality we have that $\|1_A\|_{A(G)} = \Omega(\sqrt{|A|})$ which is to say that the algebra norm is very large.  The optimist might feel inclined to guess that `small algebra norm implies structure', and they would be right.

By taking sums and differences of indicator functions of cosets we can produce other indicator functions of sets where the algebra norm is small by the triangle inequality, but this is essentially the only way in which this can happen.  In particular, the following is the quantitative idempotent theorem in the abelian case.
\begin{theorem}[{\cite[Theorem 1.3]{gresan::0}}]\label{thm.green}
Suppose that $G$ is a finite abelian group and $f\in A(G)$ is integer-valued and has $\|f\|_{A(G)} \leq M$. Then there is some $L= \exp(\exp(O(M^4)))$, subgroups $H_1,\dots,H_{L} \leq G$, elements $x_1,\dots,x_{L} \in G$ and signs $\sigma_1,\dots,\sigma_L\in \{-1,0,1\}$ such that
\begin{equation*}
f = \sum_{i=1}^L{\sigma_i1_{x_i+H_i}}.
\end{equation*}
\end{theorem}
The objective of this paper is to extend this to non-abelian groups.
\begin{theorem}\label{thm.main}
Suppose that $G$ is a finite group and $f\in A(G)$ is integer-valued and has $\|f\|_{A(G)} \leq M$. Then there is some $L=L(M)$, subgroups $H_1,\dots,H_{L} \leq G$, elements $x_1,\dots,x_{L} \in G$ and signs $\sigma_1,\dots,\sigma_L\in \{-1,0,1\}$ such that
\begin{equation*}
f = \sum_{i=1}^L{\sigma_i1_{x_iH_i}}.
\end{equation*}
\end{theorem}
There is nothing special about the choice of left cosets: $xH=(xHx^{-1})x$ for all $x \in G$ and $H \leq G$, hence we can easily pass between the left and right versions of the above result.

More than the above, our proof gives an effective, albeit weak, bound: $L$ may be taken to be triply tower in $O(M^2)$ -- $A(6,O(M^2))$ where $A$ is the Ackermann function -- although clearly the precise nature of the bound is not important.  A more detailed discussion of these matters may be found in the concluding remarks of \S\ref{sec.con}.

Theorem \ref{thm.green} was motivated by the celebrated idempotent theorem of Cohen \cite{coh::} which characterises idempotent (with respect to convolution) measures on locally compact abelian groups.  Our result is essentially a quantitative version of the non-abelian idempotent theorem of Lefranc \cite{lef::}.

Recall that if $G$ is a locally compact group then $B(G)$ denotes the Fourier-Stieltjes algebra, that is the linear span of the set of continuous positive-definite functions on $G$ endowed with point-wise multiplication.  We say that $f \in B(G)$ is \emph{idempotent} if $f$ is $\{0,1\}$-valued.  Finally we write $R(G)$ for the ring of subsets of $G$ generated by the left cosets of open subgroups of $G$.
\begin{theorem}[The idempotent theorem]
Suppose that $G$ is a locally compact group. Then the idempotents of $B(G)$ are exactly the indicator functions of the sets in $R(G)$.
\end{theorem}
Cohen proved this when $G$ is abelian by making heavy use of the dual group.  Unfortunately a suitable dual object is not available in general and so in the non-abelian setting the proofs had to proceed along rather different operator-theoretic lines.

The idempotent theorem tends to be established as an immediate corollary of the following theorem; it is this theorem that we have made quantitative.
\begin{theorem}\label{thm.lefranc}
Suppose that $G$ is a locally compact group and $f \in B(G)$ is integer-valued. Then there is some $L<\infty$, open subgroups $H_1,\dots,H_{L} \leq G$, elements $x_1,\dots,x_{L} \in G$ and signs $\sigma_1,\dots,\sigma_L\in \{-1,0,1\}$ such that
\begin{equation*}
f = \sum_{i=1}^L{\sigma_i1_{x_iH_i}}.
\end{equation*}
\end{theorem}
There are a number of proofs of this result including a very elegant one by Host \cite{hos::}, however they are all very soft.  Of course the theorem has nevertheless received many applications in the classification of various structures, for example closed unital ideals in \cite{ulg::} and homomorphisms between Fourier algebras in \cite{ilispr::} to name two.

The above applications are, however, not our primary interest: our agenda is two-fold.  In the first instance the objective has been to extend the quantitative idempotent theorem to the non-abelian setting.  However, on the way to doing this we have developed our second objective of finding a useful version of the non-abelian Fourier transform relative to a suitable notion of `approximate group', and tools to simulate Fre{\u\i}man's theorem in the non-abelian setting.  Both Fourier and Fre{\u\i}man-type tools are used extensively in additive combinatorics and an overview of what they become in the non-abelian setting may be found in the next section.

\section{An overview of the paper}\label{sec.abover}

We shall now give an overview of the proof of the abelian quantitative idempotent theorem from \cite{gresan::0}, and then explain how this needs to be adapted to the non-abelian setting which should serve as motivation for the remainder of the paper. 

The proof is iterative in nature and at each stage it takes an almost integer-valued function with small algebra norm and gives out a decomposition of that function into two non-trivial almost integer-valued functions with disjoint spectral support.
\begin{enumerate}
\item \emph{(Arithmetic connectivity)} If $f:G \rightarrow \R$ has $\|f\|_{A(G)} \leq M$ then it follows from the $\log$-convexity of the $L^p(\mu_{\wh{G}})$-norms that
\begin{equation*}
\|f \ast f\|_{L^2(\mu_G)}^2 \geq \|f\|_{L^2(\mu_G)}^6/M^2.
\end{equation*}
Given this we should like to use the Balog-Szemer{\'e}di-Fre{\u\i}man theorem to say that $f$ correlates with a coset progression.  Unfortunately if $f$ is almost integer-valued rather than integer-valued then it may be that most of its mass is supported at points where $f$ is very small.  

It turns out that the set of points where $f$ is large has a property called arithmetic connectivity and a consequence of this is that it itself has large additive energy which recovers the situation.
\item \emph{(Correlation with an approximate subgroup)} Given a set with large additive energy one can apply the Balog-Szemer{\'e}di-Fre{\u\i}man theorem; we now do so as planned to get that the set of points where $f$ is large correlates with a coset progression.
\item \emph{(Quantitative continuity)} It is classically well known (and easy to check) that if $f \in A(G)$ then $f=g$ \emph{a.e.} for some continuous function $g$.  When $G$ is finite this statement is not useful, but it can nevertheless be made quantitative in a certain sense.  Indeed, this is the main aim of the paper \cite{grekon::} of Green and Konyagin.  For our purposes we require a relative version of the Green-Konyagin result which says that given a coset progression we can find a large subset over which $f$ does not vary very much.
\item \emph{(Generating a group)}  Since $f$ is almost integer-valued and by the above will turn out to not vary very much over translation by some set $B$, it follows that it does not vary very much over the group generated by $B$.  Moreover, since $f$ also correlates with $B$ we find that we have more or less generated a subgroup $H \leq G$ such that
\begin{equation*}
\|f \ast \mu_H\|_{L^\infty(\mu_G)} \geq 1/2 \textrm{ and } \|f - f\ast \mu_H\|_{L^\infty(\mu_G)} \leq \epsilon.
\end{equation*}
\end{enumerate}
Having done this we find that $H^\perp$ supports a spectral mass of size $\Omega(1)$, $f-f\ast \mu_H$ has Fourier transform that is orthogonal to $H^\perp$ and $f-f\ast \mu_H$ is almost integer-valued.  We may now repeat the process applied to this new function $f-f\ast \mu_H$.  The iteration eventually terminates since $\|f\|_{A(G)} \leq M$ (and therefore we cannot go on collecting spectral mass indefinitely).  When it terminates we can unravel the output to find that we have subgroups $H_1,\dots,H_l$ such that
\begin{equation*}
f\ast (\delta - \mu_{H_1}) \ast \dots (\delta-\mu_{H_l}) \equiv 0.
\end{equation*}
This leads to a representation of $f$ as a weighted sum of subgroups and it is easy to see that the weights are integers.  The result then follows.

The proof has an additive combinatorial flavour and some motivation for why this might be is the following which may be checked from the definition of convolution.  If $A,B \subset G$ for some finite group $G$, then
\begin{equation*}
\supp 1_A \ast 1_B = AB \textrm{ and }1_A \ast 1_B(x) = \mu_G(A \cap xB^{-1}).
\end{equation*}
Here $AB$ denotes the product set $\{ab: a \in A, b \in B\}$, and we make the obvious convention for powers $A^n$.

The above proof strategy is fairly straightforward, and the real work of \cite{gresan::0} is in pushing it through in a general abelian group; in this paper we take the next step and try to push it through in a general group. This sort of thing has varying degrees of success and in general it is combinatorial arguments which `remain entirely in physical space' that stand a better chance of transferring directly, which is very definitely not the situation we are in.  An example of this phenomenon is with Roth's theorem, which has many proofs; the one which transfers directly to the non-abelian setting is the regularity proof (see \cite{kraserven::}) which is the most purely combinatorial. This is one reason to be interested in finding more combinatorial arguments for results in additive combinatorics, a philosophy which has been independently pursued by other authors such as Shkredov in \cite{shk::4,shk::3}.

We turn now to the process of transferring the high level proof above. First we address the notion of arithmetic connectivity.  This does not seem to have a useful non-abelian analogue.  Indeed, part of the problem is that it is naturally exploited by Riesz products which are also abelian in nature, and it is tricky to replicate their properties in general groups.  In view of this we are unable to induct on the wider class of almost integer-valued functions.  Instead we note that it is relatively easy to show that if $f$ is integer-valued, $f \ast \mu_H$ is almost integer-valued, and $\mu_G(H)$ is comparable to $\|f\|_{L^1(\mu_G)}$ then the algebra norm of $g$, the function which takes $x$ to the integer closest to $f \ast (\delta-\mu_H)(x)$, is bounded in terms of $M$.  We then apply the next two steps of the argument to $g$, only recovering $f$ in the last step because $f \ast (\delta-\mu_H)$  is close to $g$ in $L^\infty(\mu_G)$.

This modification may be done in the abelian setting.  However, doing so results in a tower-type loss in the consequent bounds, and is one of the quantitative weaknesses of our argument here.  In any case the details of this are wrapped up in the proof of the main theorem in \S\ref{sec.pmt}.

The next difficulty in generalising the argument is in the correlation with an approximate subgroup.  While the Balog-Szemer{\'e}di theorem extends immediately to the non-abelian setting, Fre{\u\i}man's theorem does not.  One of the main contributions of \cite{gresan::0} was to give a useful definition of approximate subgroups called Bourgain systems which combined both coset progressions and Bohr sets.  The definition of a Bourgain systems transfers directly to the non-abelian setting, while the correct version of coset progressions is conjecturally coset nilprogressions\footnote{See Tao's blog for a discussion of this and related matters.}.  In any case we are unfortunately not able to establish correlation with any structure as strong as a Bourgain system in our replacement of Fre{\u\i}man's theorem; we have to make do with something which we shall call a multiplicative pair.  That is roughly a pair of sets $(B,B')$ such that $BB' \approx B$.  These structures are formally introduced in \S\ref{sec.ag}, and the general theory in \S\S\ref{sec.locpsa}--\ref{sec.anallarge}.  In \S\S\ref{sec.frei}--\ref{sec.containment}, a suitable Fre{\u\i}man-type theorem is proved showing that any set with small doubling correlates with a multiplicative pair.  These Fre{\u\i}man-type results are brought together in the algebra norm context in \S\ref{sec.ank}.

The main disadvantage with these multiplicative pairs is that their notion of dimension is not additive.  In the abelian setting, all the notions of approximate subgroup have an additive notion of dimension: the intersection of two Bohr sets, multi-dimensional progressions or Bourgain systems of dimension $d_1$ and $d_2$ is $O(d_1+d_2)$.  By contrast, the intersection of multiplicative pairs has a dimension roughly exponential in $d_1$ and $d_2$,  and this is where one of the `tower contributions' comes into the bound in Theorem \ref{thm.main}.

The final difficulty we face is in running the quantitative continuity argument relative to a multiplicative pair. There is a well established non-abelian Fourier transform and in the case when $f$ is dense in the ambient group -- meaning that it has density $\Omega(1)$ -- these arguments work to prove an analogue of the result in \cite{grekon::}.   The non-abelian Fourier transform does not work relative to a multiplicative pair because there is no way to get a handle on the dimension of the representations the argument outputs. 

It turns out that when we use the Fourier transform on abelian groups all we really use is the magnitude of the Fourier coefficients and that the map diagonalizes the convolution square $v \mapsto\tilde{f} \ast f\ast v$. Now, if $G$ is non-abelian then $v \mapsto \tilde{f} \ast f \ast v$ is still diagonalizable by the spectral theorem and the magnitudes of the Fourier coefficients correspond to the singular values of $v \mapsto f \ast v$.  It is this diagonalizing basis and these singular values, therefore, which we use in place of the traditional Fourier transform.  Extending this idea, when we are working relative to a multiplicative pair we consider the basis diagonalizing the convolution square relative to said multiplicative pair.  The basic details of the non-relative analysis are contained in \S\ref{sec.ft} and are extended to multiplicative pairs in \S\S\ref{sec.locft}--\ref{sec.anallarge}.  This work is then used in \S\S\ref{sec.bog},\ref{sec.qc}\verb!&!\ref{sec.dis} to prove the quantitative continuity result we desire.

The group generation of the last stage of the abelian argument generalises straightforwardly and is contained in the main proof in \S\ref{sec.pmt}.

It remains now to say that the basic facts about linear operators which we need for our non-abelian Fourier analysis are contained in \S\ref{sec.lon}.  The transform relative to multiplicative pairs is then developed in \S\ref{sec.ft}; it may be worth reading this before the relative Fourier transform is introduced in \S\ref{sec.locft} as it provides a simple introduction to that framework.  Finally we have \S\S\ref{sec.algnormprop}--\ref{sec.sag} where the basic properties of the algebra norm are developed and the case of very small algebra norm is studied.  These sections provide a gentle introduction to the sort of arguments we shall be using throughout the paper, but may be skimmed over by the experts as much of the material is standard.

Some final remarks on notation are due.  Throughout the paper we shall write things such as $f \in A(G)$ or $g \in L^2(\mu_G)$ when, of course, all complex-valued functions are in all spaces.  We write them in this way to give an idea of which space $f$ `morally' belongs in, and would belong in in the infinitary setting.  Indeed, the passage from finite to compact groups is not difficult, and in that case the functional restrictions would be necessary.  We have not presented the work in this way because we feel that the necessary addition of various `almost everywhere' qualifiers and continuity requirements only serve to obscure our work.

Since the bounds in all our results are so poor, we make heavy use of $O$ and $\Omega$-notation, although we usually give some indication in words about the shape of the bounds after the various lemmas and theorems. 

\section{Linear operators: notation and basic properties}\label{sec.lon}

In this section we briefly record the notation and properties of linear operators which we shall need.  Since we are only interested in operators on finite dimensional spaces all the material is standard and may be found in any basic book on linear analysis.

Suppose that $H$ is a $d$-dimensional complex Hilbert space.  We write $\End(H)$ for the algebra of endomorphisms of $H$, that is linear maps $H \rightarrow H$.  The adjoint of some operator $M \in \End(H)$ is defined in the usual way as the unique $M^*\in \End(H)$ such that
\begin{equation*}
\langle Mx,y\rangle_H = \langle x,M^*y\rangle_H \textrm{ for all } x,y \in H.
\end{equation*}
The operation of taking adjoints is an involution and when coupled with this $\End(H)$ becomes a $*$-algebra; if the $*$-algebra is additionally normed with the operator norm (denoted $\|\cdot\|$) it becomes a $C^*$-algebra and so, in particular, $\|M^*\|=\|M\|$.

There is also an inner product on $\End(H)$ that is of particular interest to us.  Let $v_1,\dots,v_d$ be an orthonormal basis of $H$, then the trace of an operator $M \in \End(H)$ is defined to be
\begin{equation*}
\tr M:=\sum_{i=1}^d{\langle Mv_i,v_i\rangle_H},
\end{equation*}
and is independent of the particular choice of orthonormal basis.  We then define an inner product on $\End(H)$ by
\begin{equation*}
\langle M,M' \rangle_{\End(H)}:=\tr M'^*M \textrm{ for all } M,M' \in \End(H).
\end{equation*}
This inner product is sometimes called the Hilbert-Schmidt inner product.

In fact $\End(H)$ can be normed in a number of ways, and all those of interest to us can be expressed in terms of the singular values of an operator.  Recall that if $M \in \End(H)$ then the singular values of $M$ are the non-negative square roots of the eigenvalues of $M^*M$; we denote them $s_1(M) \geq s_2(M)\geq \dots\geq s_d(M)\geq 0$.

Singular values are most conveniently analysed through an orthonormal basis diagonalizing $M^*M$ and to find such we record a spectral theorem.

Recall that if an orthonormal basis simultaneously diagonalizes two operators then they must commute, since scalar multiplication commutes.  Furthermore, if an operator is unitarily diagonalizable, then the same basis diagonalizes its adjoint.  Thus any maximal collection of simultaneously unitarily diagonalizable operators is commuting and adjoint closed.  The spectral theorem encodes the remarkable fact that the converse holds.
\begin{theorem}[Spectral theorem for adjoint closed families of commuting operators]
Suppose that $H$ is a $d$-dimensional Hilbert space and $\mathcal{M}$ is an adjoint closed family of commuting operators on $H$. Then there is an orthonormal basis $v_1,\dots,v_d$ of $H$ such that each $v_i$ is an eigenvector of $M$ for all $M \in \mathcal{M}$.
\end{theorem}
With this result in hand we record the relevant corollary for singular values.
\begin{corollary}\label{cor.speccor}
Suppose that $H$ is a $d$-dimensional Hilbert space and $M$ is an operator on $H$.  Then there is an orthonormal basis $v_1,\dots,v_d$ of $H$ such that $M^*Mv_i = |s_i(M)|^2v_i$ for all $1 \leq i \leq d$.
\end{corollary}
\begin{proof}
Since $(M^*M)^*=M^*M$ we have that $M^*M$ is self-adjoint and therefore is an adjoint closed family of commuting operators.  A suitable basis then follows from the spectral theorem.
\end{proof}
This basis immediately gives a characterisation of the operator norm in terms of singular values.
\begin{lemma}[Singular value characterisation of the operator norm]
Suppose that $H$ is a $d$-dimensional Hilbert space and $M \in \End(H)$. Then
\begin{equation*}
\|M\|=|s_1(M)|.
\end{equation*}
\end{lemma}
\begin{proof}
By definition of the operator norm there is some unit vector $v \in H$ such that $\|M\|^2 = \langle Mv,Mv\rangle_H$.  We may decompose $v$ in terms of the basis $v_1,\dots,v_d$ given by Corollary \ref{cor.speccor} so that
\begin{equation*}
v=\sum_{i=1}^d{\mu_iv_i} \textrm{ and } \sum_{i=1}^d{|\mu_i|^2} = 1.
\end{equation*}
Then
\begin{equation*}
\|M\|^2=\langle Mv,Mv\rangle_H = \langle M^*Mv,v\rangle_H =\sum_{i=1}^d{|s_i(M)|^2|\mu_i|^2} \leq |s_1(M)|^2,
\end{equation*}
with equality when $v=v_1$.  The lemma follows.
\end{proof}
Similarly, but more easily, we have the following characterisation of the Hilbert-Schmidt norm.
\begin{lemma}[Singular value characterisation of the Hilbert-Schmidt norm]
Suppose $H$ is a $d$-dimensional Hilbert space and $M \in \End(H)$. Then
\begin{equation*}
\|M\|_{\End(H)} = \sqrt{\sum_{i=1}^d{|s_i(M)|^2}}.
\end{equation*}
\end{lemma}
\begin{proof}
Since trace is basis invariant we have a short calculation using Corollary \ref{cor.speccor} to see that
\begin{equation*}
\langle M,M\rangle_{\End(H)} = \sum_{i=1}^d{\langle Mv_i,Mv_i\rangle_H} = \sum_{i=1}^d{\langle M^*Mv_i,v_i\rangle_H} = \sum_{i=1}^d{|s_i(M)|^2},
\end{equation*}
from which the lemma follows.
\end{proof}
More generally, one can define the Schatten $p$-norm of an operator $M$ to be the $p$-norm of the singular values of $M$, so that the Hilbert-Schmidt norm $\|\cdot\|_{\End(H)}$ is the Schatten $2$-norm and the operator norm $\|\cdot\|$ is the Schatten $\infty$-norm.  In view of the definition of $A(G)$ it will be of little surprise that we are essentially interested in the Schatten $1$-norm.

\section{Convolution and the Fourier transform}\label{sec.ft}

In this section we shall develop a lot of the ideas of Fourier analysis on non-abelian groups.  This is all standard material and most books on the subject go far beyond what we need, although we found the notes \cite{tao::2} of Tao to be an uncluttered introduction.

Suppose that $G$ is a finite group.  It is easy to check that convolution is associative so that $L^1(\mu_G)$ becomes an algebra with this multiplication.  If $f \in L^1(\mu_G)$ then we write $\tilde{f}$ for the \emph{adjoint} of $f$, that is the function $x \mapsto \overline{f(x^{-1})}$, and, again, it is easy to check that $L^1(\mu_G)$ is now a $*$-algebra.

Now, the map $f \mapsto L_f$, the convolution operator defined by $f \in L^1(\mu_G)$ functions a bit like the Fourier transform and encodes the aspects of the Fourier transform which are most easily generalised to the approximate setting of interest to us later.  Linearity and associativity of convolution give us that this map is an algebra homomorphism and, in fact, a short calculation shows that it is a $*$-algebra homomorphism.
\begin{lemma}[Adjoints of convolution operators]\label{lem.triv}
Suppose that $G$ is a finite group and $f,g,h \in L^2(\mu_G)$. Then
\begin{equation*}
 \langle g,\tilde{f} \ast h\rangle_{L^2(\mu_G)} = \langle f \ast g,h\rangle_{L^2(\mu_G)} =  \langle f,h \ast \tilde{g}\rangle_{L^2(\mu_G)}.
\end{equation*}
In particular $L_f^*=L_{\tilde{f}}$.
\end{lemma}
The image of the map $f\mapsto L_f$, is the sub-$*$-algebra of convolution operators in the space $\End(L^2(\mu_G))$.  It turns out that this map is not just a $*$-homomorphism, but it is also isometric.  Write $\delta_y$ for the usual Dirac delta measure supported at $y\in G$, that is the unique measure such that
\begin{equation*}
\int{kd\delta_y}=k(y) \textrm{ for all } k \in L^\infty(\mu_G).
\end{equation*}
\begin{theorem}[Parseval's theorem]
Suppose that $G$ is a finite group and $f,g \in L^2(\mu_G)$. Then
\begin{equation*}
\langle L_f,L_g\rangle_{\End(L^2(\mu_G))} = \langle f,g\rangle_{L^2(\mu_G)}.
\end{equation*}
\end{theorem}
\begin{proof}
The collection $(\delta_y/\sqrt{|G|})_{y \in G}$ is an orthonormal basis of $L^2(\mu_G)$, and since trace is independent of basis we get that
\begin{equation}\label{eqn.park}
\langle L_f,L_g\rangle_{\End(L^2(\mu_G))} = \frac{1}{|G|}\sum_{y \in G}{\langle L_f\delta_y,L_g\delta_y\rangle_{L^2(\mu_G)}}.
\end{equation}
However 
\begin{equation*}
L_f\delta_y(x) = \int{f(z)\delta_y(z^{-1}x)d\mu_G(z)} = f(xy^{-1}),
\end{equation*}
and similarly $L_g\delta_y(x)=g(xy^{-1})$, so it follows that
\begin{equation*}
\langle L_f\delta_y,L_g\delta_y\rangle_{L^2(\mu_G)} = \int{f(xy^{-1})\overline{g(xy^{-1})}d\mu_G(x)} = \langle f,g\rangle_{L^2(\mu_G)}.
\end{equation*}
Inserting this into (\ref{eqn.park}) we arrive at the result.
\end{proof}
A key property of the algebra of convolution operators is that all the elements commute with right translation. To be clear the right regular representation $\rho$ is defined by
\begin{equation*}
\rho_y:L^2(\mu_G) \rightarrow L^2(\mu_G);v \mapsto v \ast \delta_{y^{-1}}. 
\end{equation*}
It is easy to check that $\rho_y(f)(x)=f(xy)$ and that $x \mapsto \rho_x$ is a group homomorphisms of $G$ into $U(L^2(\mu_G))$.  Now, since convolution is associative
\begin{equation*}
\rho_yL_f = L_f \rho_y \textrm{ for all } y \in G, f \in L^1(\mu_G).
\end{equation*}
In fact it turns out that the algebra of convolution operators is precisely the subalgebra of operators in $\End(L^2(\mu_G))$ that commute with all right translation operators as the following result shows.
\begin{theorem}[Inversion theorem]
Suppose that $G$ is a finite group and $M \in \End(L^2(\mu_G))$ is such that $M\rho_y = \rho_yM$ for all $y \in G$.  Then there is some $f \in L^1(\mu_G)$ such that $M=L_f$.
\end{theorem}
\begin{proof}
We define $f$ in the obvious way: put $f=M\delta_{1_G}$.  Now $(\delta_y)_{y \in G}$ is a basis for $L^2(\mu_G)$ so we just need to check that $L_f=M$ on this basis and we shall be done.  First,
\begin{equation*}
L_f\delta_y(z) = \int{f(x)\delta_y(x^{-1}z)d\mu_G(x)} = f(zy^{-1}) = M\delta_{1_G}(zy^{-1}) = \rho_{y^{-1}}(M\delta_{1_G})(z).
\end{equation*}
Since $M$ commutes with $\rho_{y^{-1}}$ we conclude that
\begin{equation*}
L_f\delta_y = M\rho_{y^{-1}}(\delta_{1_G})=M\delta_y \textrm{ for all }y\in G.
\end{equation*}
The result follows.
\end{proof}
This result is basically the Fourier inversion theorem, and will not have an analogue when we generalise from groups to approximate groups, so we need to be careful about using it.

It should be said that the combinatorial importance of the fact that the algebra of convolution operators commutes with all right translation operators is well demonstrated by Lubotzky, Phillips and Sarnak in \cite{lubphisar::} (see also \cite{davsarval::}), and was put to work in additive combinatorics by Gowers \cite{gow::2}.

The utility of the Fourier transform in abelian groups is derived from the fact that it is the unique (up to scalar) unitary change of basis that simultaneously diagonalizes all convolution operators.  Of course simultaneously diagonalizable operators commute with each other so such a basis cannot exist in the non-abelian setting.

Examining many Fourier arguments in additive combinatorics reveals that in fact the universaily of the Fourier basis is unnecessary and, furthermore, we tend only to consider the absolute values of Fourier coefficients.  In light of this we make the following definitions. 

Given $f \in L^1(\mu_G)$ we write $s_i(f)$ for $s_i(L_f)$, the $i$th singular value of the operator $L_f$.  We then call an orthonormal basis $v_1,\dots,v_N$ a \emph{Fourier basis of $L^2(\mu_G)$ for $f$} if
\begin{equation*}
L_f^*L_fv_i = |s_i(f)|^2v_i \textrm{ for all } i \in \{1,\dots, N\}.
\end{equation*}
The existence of Fourier bases follows immediately from Corollary \ref{cor.speccor}.
\begin{theorem}[Fourier bases]
Suppose that $G$ is a finite group and $f \in L^1(\mu_G)$. Then there is an orthonormal Fourier basis for $f$.
\end{theorem}
Note that if $v_1,\dots,v_N$ is a Fourier basis of $L^2(\mu_G)$ for $f$ then so is the orthonormal basis $\rho_y(v_1),\dots,\rho_y(v_N)$ for all $y \in G$ since $\rho_y$ is unitary and the operators $L_f$ and $L_f^*$ commute with it.

Finally for this section we record the Hausdorff-Young inequality as it is realised in our framework.
\begin{lemma}[Hausdorff-Young inequality]
Suppose that $G$ is a finite group and $f \in L^1(\mu_G)$.  Then
\begin{equation*}
s_1(f) \leq \|f\|_{L^1(\mu_{G})}.
\end{equation*}
\end{lemma}
We call this the Hausdorff-Young inequality because when $G$ is abelian, the singular values $s_1(f) \geq s_2(f) \geq \dots \geq s_N(f)\geq 0$ are just the absolute values of the Fourier transform arranged in order, whence $s_1(f)=\|\wh{f}\|_{\ell^\infty(\wh{G})}$ and the inequality reduces to the usual Hausdorff-Young inequality.  Since this is one of the few facts that relativises without loss we shall not prove it here; it is a special case of Lemma \ref{lem.hdy}.

\section{The $A(G)$-norm: basic properties}\label{sec.algnormprop}

For an arbitrary locally compact group the space $A(G)$ was defined by Eymard in \cite{eym::}, and many of the basic properties are developed in that paper.  For completeness and because of their simplicity we shall include some the tools we require here.  First recall that the two main norms of interest:
\begin{equation*}
\|f\|_{PM(G)}:=\|L_f\| \textrm{ and } \|f\|_{A(G)}=\sup\{|\langle f,g\rangle_{L^2(\mu_G)}|: \|g\|_{PM(G)} \leq 1\}.
\end{equation*}
To begin with we note that the $A(G)$-norm is involution and translation invariant in the following sense.
\begin{lemma}[Invariance of the $A(G)$-norm]\label{lem.invariance}
Suppose that $G$ is a  finite group and $f \in A(G)$.  Then
\begin{equation*}
\|\tilde{f}\|_{A(G)} = \|f\|_{A(G)} \textrm{ and } \|\rho_y(f)\|_{A(G)} = \|f\|_{A(G)}.
\end{equation*}
\end{lemma}
\begin{proof} 
Since $f=\tilde{\tilde{f}}$ and $\rho_y(\rho_{y^{-1}}(f)) = f$ it suffices to prove that $\|f\|_{A(G)} \leq \|\tilde{f}\|_{A(G)}$ and $\|f\|_{A(G)} \leq \|\rho_y(f)\|_{A(G)}$ for all $f \in A(G)$ and $y \in G$.

The space $A(G)$ is defined by duality so there is essentially only one way to begin.  Suppose $f \in A(G)$ and let $g \in PM(G)$ be such that 
\begin{equation*}
\|f\|_{A(G)} = \langle f,g\rangle_{L^2(\mu_G)} \textrm{ and } \|L_g\| \leq 1.
\end{equation*}
First we show involution invariance. By change of variables we see that
\begin{equation*}
\langle f,g\rangle_{L^2(\mu_G)} = \overline{\langle \tilde{f},\tilde{g}\rangle_{L^2(\mu_G)}} \leq \|\tilde{f}\|_{A(G)}\|L_{\tilde{g}}\|.
\end{equation*}
However, by Lemma \ref{lem.triv} we have that $L_{\tilde{g}}=L_g^*$, and so
\begin{equation*}
\|L_{\tilde{g}}\| = \|L_g^*\|=\|L_g\|\leq 1.
\end{equation*}
The first inequality follows.

Translation invariance is proved in much the same way.  By a (different) change of variables we get that
\begin{equation*}
\langle f,g\rangle_{L^2(\mu_G)} = \langle \rho_y(f),\rho_y(g)\rangle_{L^2(\mu_G)} \leq \|\rho_y(f)\|_{A(G)}\|L_{\rho_y(g)}\|.
\end{equation*}
On the other hand by definition of $\rho_y$ we have that
\begin{equation*}
\|L_{\rho_y(g)}\| = \|L_{g \ast \delta_{y^{-1}}}\| = \|L_g L_{\delta_{y^{-1}}}\| \leq \|L_g\| \|L_{\delta_{y^{-1}}}\| =\|L_{\delta_{y^{-1}}}\|. 
\end{equation*}
But
\begin{equation*}
L_{\delta_{y^{-1}}}v(x) = v(yx),
\end{equation*}
whence $\|L_{\delta_{y^{-1}}}\| = 1$ and we are done.
\end{proof}
Inspired by the abelian setting where the singular values of $L_f$ are just the absolute values of the Fourier coefficients of $f$ we have the following useful explicit formula for the $A(G)$-norm.
\begin{lemma}[Explicit formula]\label{lem.algnormexp}
Suppose that $G$ is a finite group and $f \in A(G)$. Then
\begin{equation*}
\|f\|_{A(G)} = \sum_{i=1}^N{|s_i(f)|}.
\end{equation*}
\end{lemma}
\begin{proof}
Let $v_1,\dots,v_N$ be a Fourier basis of $L^2(\mu_G)$ for $f$. We shall show that $\|f\|_{A(G)}$ is both less than or equal and greater than or equal to the right hand side.  The first of these is easy: let $g$ be such that
\begin{equation*}
\|f\|_{A(G)}= \langle f,g\rangle_{L^2(\mu_G)} \textrm{ and } \|L_g\| \leq 1.
\end{equation*}
By Parseval's theorem and the definition of trace we have that
\begin{equation*}
\langle f,g\rangle_{L^2(\mu_G)} = \tr L_g^*L_{f}=\sum_{i=1}^N{\langle L_fv_i,L_gv_i\rangle_{L^2(\mu_G)}}.
\end{equation*}
It follows from the Cauchy-Schwarz inequality that
\begin{equation*}
\|f\|_{A(G)} \leq \sum_{i=1}^N{\|L_fv_i\|_{L^2(\mu_G)}\|L_gv_i\|_{L^2(\mu_G)}}.
\end{equation*}
Of course,
\begin{equation*}
\|L_fv_i\|_{L^2(\mu_G)}^2 = \langle L_f v_i,L_fv_i\rangle_{L^2(\mu_G)} = \langle L_f^*L_fv_i,v_i\rangle_{L^2(\mu_G)} = |s_i(f)|^2
\end{equation*}
and $\|L_gv_i\|_{L^2(\mu_G)} \leq 1$, whence
\begin{equation*}
\|f\|_{A(G)} \leq \sum_{i=1}^N{|s_i(f)|}.
\end{equation*}
For the other direction we define an operator $U$, extending by linearity from the basis $v_1,\dots,v_N$ as follows
\begin{equation*}
Uv_i:=\begin{cases}L_fv_i/|s_i(f)| & \textrm{ if } s_i(f)\neq 0\\ 0 & \textrm{ otherwise.}\end{cases}
\end{equation*}
We have two claims about $U$.
\begin{claim*}
$\|U\| \leq 1$.
\end{claim*}
\begin{proof}
As usual it suffices to check that $\|Uv_i\|_{L^2(\mu_G)} \leq 1$ since the basis $v_1,\dots,v_N$ is orthonormal.  If $s_i(f)=0$ then $Uv_i=0$ whence $\|Uv_i\|_{L^2(\mu_G)} = 0$; if $s_i(f)\neq 0$ then
\begin{eqnarray*}
\|Uv_i\|_{L^2(\mu_G)}^2 = \langle Uv_i,Uv_i\rangle_{L^2(\mu_G)} &= &\frac{1}{|s_i(f)|^2}\langle L_fv_i,L_fv_i\rangle_{L^2(\mu_G)}\\ & = & \frac{1}{|s_i(f)|^2}\langle L_f^*L_fv_i,v_i\rangle_{L^2(\mu_G)}=1.
\end{eqnarray*}
The claim follows.
\end{proof}
\begin{claim*}
$U\rho_y = \rho_yU \textrm{ for all } y \in G$.
\end{claim*}
\begin{proof}
By linearity it suffices to verify this on the basis $v_1,\dots,v_N$.   Since $\rho_y$ is unitary we see that $\rho_y(v_i)$ is a unit vector and so there are complex numbers $\mu_1,\dots,\mu_N$ such that
\begin{equation*}
\rho_y(v_i)=\sum_{j=1}^N{\mu_jv_j} \textrm{ and } \sum_{j=1}^N{|\mu_j|^2}=1.
\end{equation*}
Now, $\rho_y$ commutes with $L_f^*L_f$ for all $y \in G$ whence
\begin{equation*}
\sum_{j=1}^N{\mu_j|s_j(f)|^2v_j} = \sum_{j=1}^N{\mu_jL_f^*L_fv_j}=L_f^*L_f\rho_yv_i = \rho_yL_f^*L_fv_i = \sum_{j=1}^N{\mu_j|s_i(f)|^2v_j}.
\end{equation*}
Since $v_1,\dots,v_N$ is a basis it follows that $|s_j(f)|=|s_i(f)|$ whenever $\mu_j\neq 0$. Now, if $s_i(f)\neq 0$ then it follows that
\begin{eqnarray*}
U\rho_yv_i = \sum_{j=1}^N{\frac{\mu_j}{|s_j(f)|}L_fv_j}&=&\frac{1}{|s_i(f)|}L_f\sum_{j=1}^N{\mu_jv_j}\\ &=&L_f\rho_yv_i/|s_i(f)|=\rho_yL_fv_i/|s_i(f)| = \rho_yUv_i.
\end{eqnarray*}
Similarly if $s_i(f)=0$, both $U\rho_yv_i$ and $\rho_yUv_i$ is $0$.
\end{proof}
Now, by the inversion formula there is some $g \in  L^1(\mu_G)$ such that $U=L_g$ and hence by Parseval's theorem and the definition of trace we have
\begin{eqnarray*}
\langle f,g\rangle_{L^2(\mu_G)} & =& \langle L_f,L_g\rangle_{\End(L^2(\mu_G))}\\& =& \sum_{i=1}^N{\langle L_fv_i,Uv_i\rangle_{L^2(\mu_G)}}\\ & =& \sum_{i: s_i(f)\neq 0}{\frac{1}{|s_i(f)|}\langle L_fv_i,L_fv_i\rangle_{L^2(\mu_G)}}\\ & =& \sum_{i: s_i(f)\neq 0}{\frac{1}{|s_i(f)|}\langle L_f^* L_fv_i,v_i\rangle_{L^2(\mu_G)}} =  \sum_{i=1}^N{|s_i(f)|}.
\end{eqnarray*}
On the other hand
\begin{equation*}
|\langle f,g\rangle_{L^2(\mu_G)}| \leq \|f\|_{A(G)}\|g\|_{PM(G)} = \|f\|_{A(G)} \|L_g\|  \leq \|f\|_{A(G)},
\end{equation*}
since $\|L_g\|=\|U\| \leq 1$ by the claim and construction of $U$.  We conclude that
\begin{equation*}
\sum_{i=1}^N{|s_i(f)|}\leq \|f\|_{A(G)},
\end{equation*}
and hence the result is proved.
\end{proof}
Qualitatively if a function is in $A(G)$ then it is continuous.  Of course this has little utility in the finite setting, but a key part of this paper is concerned with developing a quantitative analogue of this statement. To begin this process we record the following trivial nesting of norms.
\begin{lemma}[$A(G)$ dominates $L^\infty(\mu_G)$]\label{lem.linfag}
Suppose that $G$ is a finite group and $f \in A(G)$. Then
\begin{equation*}
\|f\|_{L^\infty(\mu_G)} \leq \|f\|_{A(G)}.
\end{equation*}
\end{lemma}
\begin{proof}
Let $v_1,\dots,v_N$ be a Fourier basis of $L^2(\mu_G)$ for $f$.  Since for every $y \in G$, the sequence $\rho_yv_1,\dots,\rho_yv_N$ is also a Fourier basis of $L^2(\mu_G)$ for $f$ we see that
\begin{equation*}
f(x)=L_f\delta_{1_G}(x)=\sum_{i=1}^N{\langle \delta_{1_G},\rho_yv_i\rangle_{L^2(\mu_G)}L_f\rho_yv_i(x)}
\end{equation*}
for all $x,y \in G$.  On the other hand $L_f\rho_yv_i(x)=\rho_yL_fv_i(x) = L_fv_i(xy)$ by definition of $\rho_y$ and the fact that it commutes with $L_f$, whence
\begin{equation*}
f(x)=\sum_{i=1}^N{\overline{v_i(y)}L_fv_i(xy)}.
\end{equation*}
Integrating against $y$ and applying the triangle inequality and then Cauchy-Schwarz inequality term-wise to the summands we get that
\begin{eqnarray*}
|f(x)| & \leq & \sum_{i=1}^N{|\int{\overline{v_i(y)}L_fv_i(xy)d\mu_G(y)}|}\\ & \leq & \sum_{i=1}^N{\|v_i\|_{L^2(\mu_G)}\|L_fv_i\|_{L^2(\mu_G)}} = \sum_{i=1}^N{|s_i(f)|}.
\end{eqnarray*}
The last equality is since
\begin{equation*}
\|L_fv_i\|_{L^2(\mu_G)}^2 = \langle L_f v_i,L_fv_i\rangle_{L^2(\mu_G)} = \langle L_f^*L_fv_i,v_i\rangle_{L^2(\mu_G)} = |s_i(f)|^2.
\end{equation*}
The lemma now follows from the explicit formula for the $A(G)$-norm.
\end{proof}
Although the above results are useful, the main result of this section and the principal reason that the $A(G)$-norm is so important is that it is an algebra norm.  Finally we are in a position to prove this fact.
\begin{proposition}[The $A(G)$-norm is an algebra norm]
Suppose that $G$ is a finite group and $f,g \in A(G)$.  Then
\begin{equation*}
\|fg\|_{A(G)}\leq \|f\|_{A(G)}\|g\|_{A(G)}.
\end{equation*}
\end{proposition}
\begin{proof}
As usual we proceed by duality.  Let $h$ be such that
\begin{equation}\label{eqn.inne}
\|fg\|_{A(G)}=\langle fg,h\rangle_{L^2(\mu_G)} \textrm{ and } \|L_h\| \leq 1.
\end{equation}
Let $(v_i)_{i=1}^N$ be a Fourier basis of $L^2(\mu_G)$ for $f$ and $(w_i)_{i=1}^N$ be a Fourier basis of $L^2(\mu_G)$ for $g$. As in the previous lemma, for all $y \in G$ we have that $\rho_yv_1,\dots,\rho_yv_N$ is a Fourier basis of $L^2(\mu_G)$ for $f$ and so
\begin{equation*}
f(x)= L_f\delta_{1_G}(x)=\sum_{i=1}^N{\langle \delta_{1_G},\rho_yv_i\rangle_{L^2(\mu_G)}L_f\rho_yv_i(x)}=\sum_{i=1}^N{\overline{v_i(y)}L_fv_i(xy)},
\end{equation*}
and similarly
\begin{equation*}
g(x) = \sum_{i=1}^N{\overline{w_i(yz)}L_gw_i(xyz)}.
\end{equation*}
Inserting these expressions for $f$ and $g$ into the inner product in (\ref{eqn.inne}) we get that
\begin{eqnarray*}
\|fg\|_{A(G)} & = & \sum_{i=1}^N{\sum_{j=1}^N{\int{\overline{v_i(y)}L_fv_i(xy)\overline{w_i(yz)}L_gw_i(xyz)\overline{h(x)}d\mu_G(x)}}}\\ & = & \sum_{i=1}^N{\sum_{j=1}^N{\overline{v_i(y)w_j(yz)}\int{\tilde{h}(x)L_fv_i(x^{-1}y)L_g\rho_{z}(w_j)(x^{-1}y)d\mu_G(x)}}}\\& = & \sum_{i=1}^N{\sum_{j=1}^N{\overline{v_i(y)w_j(yz)}L_{\tilde{h}}(L_fv_iL_g\rho_{z}w_j)(y)}}.
\end{eqnarray*}
Since the above expression is valid for all $y,z \in G$ we may apply the triangle inequality and integrate to get that
\begin{equation*}
\|fg\|_{A(G)} \leq \sum_{i=1}^N{\sum_{j=1}^N{\int{|v_i(y)w_j(yz)L_{\tilde{h}}(L_fv_iL_g\rho_{z}w_j)(y)|d\mu_G(y)d\mu_G(z)}}}.
\end{equation*}
By the Cauchy-Schwarz inequality 
\begin{equation*}
\int{|v_i(y)w_j(yz)L_{\tilde{h}}(L_fv_iL_g\rho_{z}w_j)(y)|d\mu_G(y)\mu_G(z)}
\end{equation*}
is at most
\begin{equation*}
\left(\int{|v_i(y)w_j(yz)|^2d\mu_G(y)d\mu_G(z)}\right)^{1/2}\left(\int{|L_{\tilde{h}}(L_fv_iL_g\rho_{z}w_j)(y)|^2d\mu_G(y)\mu_G(z)}\right)^{1/2}.
\end{equation*}
The first integral is $1$ by the change of variables $u=yz$; the second is at most
\begin{equation*}
\int{\|L_{\tilde{h}}\|^2 \int{|L_fv_i(y)L_g\rho_zw_j(y)|^2d\mu_G(y)}d\mu_G(z)}.
\end{equation*}
Since $L_g$ and $\rho_z$ commute we see that $L_g\rho_zw_j(y)=L_gw_j(yz)$ whence, by change of variables, the previous expression is equal to
\begin{equation*}
\|L_{\tilde{h}}\|^2\|L_fv_i\|_{L^2(\mu_G)}^2\|L_gw_j\|_{L^2(\mu_G)}^2 = |s_i(f)|^2|s_j(g)|^2.
\end{equation*}
The inequality follows since $\|L_{\tilde{h}}\| = \|L_h^*\| =\|L_h\|=1$.  It follows that
\begin{equation*}
\|fg\|_{A(G)} \leq \sum_{i=1}^N{\sum_{j=1}^N{|s_i(f)||s_j(g)|}} = \|f\|_{A(G)}\|g\|_{A(G)},
\end{equation*}
where the last equality is by the explicit formula for the $A(G)$-norm.  The result is proved.
\end{proof}
Related to the above is what happens when we convolve two functions instead of multiplying them.  In this regard we have the following lemma.
 \begin{lemma}\label{lem.a-ap}
 Suppose that $G$ is a finite group and $f\in A(G), g \in PM(G)$. Then
 \begin{equation*}
 \|f \ast g\|_{A(G)} \leq \|f\|_{A(G)}\|g\|_{PM(G)}.
 \end{equation*}
 \end{lemma}
 \begin{proof}
 By the definition of the algebra norm and Lemma \ref{lem.triv} we have
 \begin{eqnarray*}
 \|f \ast g\|_{A(G)} & = & \sup\{|\langle f \ast g,h\rangle_{L^2(\mu_G)}|: \|h\|_{PM(G)}\leq 1\}\\ & = &  \sup\{|\langle f,h\ast \tilde{g} \rangle_{L^2(\mu_G)}|: \|h\|_{PM(G)}\leq 1\}\\ & \leq &   \sup\{|\langle f,k\rangle_{L^2(\mu_G)}|: \|k\|_{PM(G)}\leq \|g\|_{PM(G)}\}\\ & =& \|f\|_{A(G)}\|g\|_{PM(G)}.
 \end{eqnarray*}
The result is proved.
 \end{proof}
 
 \section{Basic computations with the algebra norm: some functions with small algebra norm}\label{sec.algfunc}
 
In this section we shall compute the algebra norm of functions of various shapes which will be used later in our work.  It may also be useful to read the short lemmas that follow to get more of a hand on how the norm behaves.

As indicated in the overview in \S\ref{sec.abover} and as should be clear from the definition of a Fourier basis we shall make heavy use of convolution squares.  No small part of that reason is the following lemma.
\begin{lemma}\label{lem.convag}
Suppose that $G$ is a finite group and $A$ is a non-empty subset of $G$. Then
\begin{equation*}
\|\widetilde{1_A} \ast \mu_A\|_{A(G)} = 1.
\end{equation*}
\end{lemma}
\begin{proof}
This is an easy calculation. Let $v_1,\dots,v_N$ be a Fourier basis of $L^2(\mu_G)$ for $1_A$.  By Parseval's theorem we then have
\begin{equation*}
\sum_{i=1}^N{|s_i(1_A)|^2} = \sum_{i=1}^N{\langle L_{1_A}v_i,L_{1_A}v_i\rangle_{L^2(\mu_G)}} = \langle 1_A, 1_A \rangle_{L^2(\mu_G)} = \mu_G(A).
\end{equation*}
On the other hand, by the definition of the algebra norm there is some $U$ with $\|U\| \leq 1$ such that
\begin{equation*}
\|\widetilde{1_A} \ast \mu_A\|_{A(G)} =\langle L_{\widetilde{1_A} \ast \mu_A},U\rangle_{\End(L^2(\mu_G))}.
\end{equation*}
By the definition of trace we expand this in the basis $v_1,\dots,v_N$ to get that
\begin{equation*}
\|\widetilde{1_A} \ast 1_A\|_{A(G)} =\sum_{i=1}^N{\langle L_{1_A}^*L_{1_A}v_i,Uv_i\rangle_{L^2(\mu_G)}} = \sum_{i=1}^N{|s_i(1_A)|^2\langle v_i,Uv_i\rangle_{L^2(\mu_G)}}.
\end{equation*}
Since $|\langle v_i,Uv_i\rangle_{L^2(\mu_G)}| \leq 1$ we conclude that
\begin{equation*}
\|\widetilde{1_A} \ast 1_A\|_{A(G)} \leq \sum_{i=1}^N{|s_i(1_A)|^2} = \mu_G(A),
\end{equation*}
and hence that $\|\widetilde{1_A} \ast \mu_A\|_{A(G)} \leq 1$. In the other direction we note that $\widetilde{1_A} \ast \mu_A(1_G) = 1$ and the result follows from Lemma \ref{lem.linfag}.
 \end{proof}
 An immediate corollary of the above is that indicator functions of cosets have algebra norm $1$, a fact we claimed in the introduction.
 \begin{corollary}\label{cor.indsmall}
Suppose that $G$ is a finite group, $H \leq G$ and $x \in G$. Then
\begin{equation*}
\|1_{xH}\|_{A(G)} =\|1_{Hx}\|_{A(G)}=1.
\end{equation*}
\end{corollary}
\begin{proof}
First note that by Lemma \ref{lem.invariance} we have
\begin{equation*}
\|1_{xH}\|_{A(G)} = \|\widetilde{1_{xH}}\|_{A(G)} = \|1_{Hx^{-1}}\|_{A(G)} = \|\rho_x(1_H)\|_{A(G)} = \|1_H\|_{A(G)},
\end{equation*}
and similarly (but more easily) $\|1_{Hx}\|_{A(G)} = \|1_H\|_{A(G)}$.  It follows that without loss of generality we may assume that $x=1_G$.  Of course then the lemma is a simple consequence of Lemma \ref{lem.convag} since $1_H \ast \mu_H = 1_H$.
\end{proof}
Convolution squares are most useful for their `positivity in the dual' property in the abelian case which is captured for our setting by the following lemma.
 \begin{lemma}\label{lem.bohrfourier}
 Suppose that $G$ is a finite group and $A \subset G$ is non-empty. Then
 \begin{equation*}
 \|\mu_{\{1_G\}} - \widetilde{\mu_A} \ast \mu_{A}\|_{PM(G)} \leq 1
 \end{equation*}
 \end{lemma}
 \begin{proof}
 Let $v_1,\dots,v_N$ be a Fourier basis of $L^2(\mu_G)$ for $\mu_A$.  It follows that
 \begin{equation*}
 (\mu_{\{1_G\}} - \widetilde{\mu_A} \ast \mu_{A}) \ast v_i = v_i(1-|s_i(\mu_A)|^2)
 \end{equation*}
 for all $i$.  Integrating we conclude that
  \begin{equation*}
 \| (\mu_{\{1_G\}} - \widetilde{\mu_A} \ast \mu_{A}) \ast v_i \|_{L^2(\mu_G)}^2 = |1-|s_i(\mu_A)|^2|^2.
 \end{equation*}
 By the Hausdorff-Young inequality we see that $|s_i(\mu_A)|\leq \|\mu_A\| = 1$, and hence that $1-|s_i(\mu_A)|^2 \geq 0$; thus
 \begin{equation*}
 \|\mu_{\{1_G\}} - \widetilde{\mu_A} \ast \mu_{A}\|_{PM(G)} = \sup_{1 \leq i \leq N}{1-|s_i(\mu_A)|^2} \leq 1.
 \end{equation*}
The result is proved.
 \end{proof}
 Our main argument is an induction on spectral mass so we need a way to hive off portions of spectral mass without destroying the physical space properties of our functions.  The following is the crucial decomposition lemma which is, as usual, much easier to see in the abelian setting.
 \begin{oldversion}
 \begin{lemma}\label{lem.decompmass}
Suppose that $G$ is a finite group, $H \leq G$ and $f \in A(G)$.  Then
\begin{equation*}
\|f\|_{A(G)} = \|f-f\ast \mu_H\|_{A(G)} + \|f \ast \mu_H\|_{A(G)}.
\end{equation*}
\end{lemma}
\begin{proof}
We write $M_1:=L_{f - f\ast \mu_H}$ and $M_2=L_{f\ast\mu_H}$.  Now, 
\begin{equation*}
M_1M_2^* = L_{(f - f\ast \mu_H)\ast \widetilde{\mu_H}\ast \widetilde{f}} = L_{f \ast \widetilde{\mu_H} \ast \widetilde{f} - f \ast \widetilde{\mu_H} \ast \widetilde{f}} = 0
\end{equation*}
since $\mu_H \ast \widetilde{\mu_H} = \widetilde{\mu_H}$.  Similarly $M_2M_1^*=0$ since $\mu_H\ast \widetilde{\mu_H} = \mu_H$, whence $M_1^*M_1$ and $M_2^*M_2$ are commuting self-adjoint operators, indeed
\begin{equation*}
M_1^*M_1M_2^*M_2 = M_1^*.0.M_2 = 0 = M_2^*.0.M_1 = M_2^*M_2M_1^*M_1.
\end{equation*}
It follows that there is a basis $v_1,\dots,v_N$ of $L^2(\mu_G)$ which simultaneously diagonalizes both of them, whence there are permutations $\pi_1$ and $\pi_2$ of $\{1,\dots,N\}$ such that
\begin{equation*}
M_1^*M_1v_i = |s_{\pi_1(i)}(f-f \ast \mu_H)|^2v_i \textrm{ and } M_2^*M_2v_i = |s_{\pi_2(i)}(f \ast \mu_H)|^2v_i.
\end{equation*}
By the explicit formula for the algebra norm we then have
\begin{equation*}
\|f-f \ast \mu_H\|_{A(G)} = \sum_{i=1}^N{\langle M_1^*M_1v_i,v_i\rangle_{L^2(\mu_G)}^{1/2}}
\end{equation*}
and
\begin{equation*}
\|f \ast \mu_H\|_{A(G)} =  \sum_{i=1}^N{\langle M_2^*M_2v_i,v_i\rangle_{L^2(\mu_G)}^{1/2}}.
\end{equation*}
Now, since $M_1^*M_1M_2^*M_2=0$ and the $v_i$s are eigenvectors of both $M_1^*M_1$ and $M_2^*M_2$ we also know that for each $i$ at most one of the two summands $\langle M_1^*M_1v_i,v_i\rangle_{L^2(\mu_G)}^{1/2}$ and $\langle M_2^*M_2v_i,v_i\rangle_{L^2(\mu_G)}^{1/2}$ can be non-zero whence
\begin{equation*}
\langle M_1^*M_1v_i,v_i\rangle_{L^2(\mu_G)}^{1/2}+\langle M_2^*M_2v_i,v_i\rangle_{L^2(\mu_G)}^{1/2} = \langle (M_1^*M_1+M_2^*M_2)v_i,v_i\rangle_{L^2(\mu_G)}^{1/2}.
\end{equation*}
Of course, since $M_1^*M_2=0$ and $M_2^*M_1=0$ we conclude that
\begin{equation*}
\langle M_1^*M_1v_i,v_i\rangle_{L^2(\mu_G)}^{1/2}+\langle M_2^*M_2v_i,v_i\rangle_{L^2(\mu_G)}^{1/2} = \langle (M_1+M_2)^*(M_1+M_2)v_i,v_i\rangle_{L^2(\mu_G)}^{1/2}.
\end{equation*}
However, $L_f=M_1+M_2$ and, furthermore,
\begin{equation*}
L_f^*L_fv_i = (M_1+M_2)^*(M_1+M_2)v_i = M_1^*M_1v_i + M_2^*M_2v_i,
\end{equation*}
which is a scalar multiple of $v_i$ since $v_i$ is an eigenvector of $M_1^*M_1$ and $M_2^*M_2$. It follows that $v_1,\dots,v_N$ also diagonalizes $L_f^*L_f$, and hence
\begin{equation*}
\sum_{i=1}^N{\langle (M_1+M_2)^*(M_1+M_2)v_i,v_i\rangle_{L^2(\mu_G)}^{1/2}}=\|f\|_{A(G)}
\end{equation*}
by the explicit formula for the algebra norm.  The lemma follows combining this with the previous.
\end{proof}\end{oldversion}
\begin{newversion}
This lemma is false. My thanks to Daniel Naylor for alerting me to this error which also occurs in the sequel \cite{san::24}. It can be fixed by replacing it with the following lemma. More details of this are given in \cite{san::24}.
 \begin{lemma*}[Lemma \ref{lem.decompmass}, fixed]
Suppose that $G$ is a finite group, $H \leq G$ and $f \in A(G)$.  Then
\begin{equation*}
\|f\|_{A(G)}^2 \geq \|f-f\ast \mu_H\|_{A(G)}^2 + \|f \ast \mu_H\|_{A(G)}^2.
\end{equation*}
\end{lemma*}
\end{newversion}

\section{Indicator functions with very small $A(G)$-norm}\label{sec.sag}

Suppose that $G$ is a finite group and $A\subset G$ is not empty.  It follows from Lemma \ref{lem.linfag} that
\begin{equation*}
\|1_A\|_{A(G)} \geq \|1_A\|_{L^\infty(\mu_G)} \geq 1.
\end{equation*}
Our main theorem is to be thought of as describing the structure of $1_A$ when $\|1_A\|_{A(G)}$ tends to infinity very slowly in the size of the group.  If $\|1_A\|_{A(G)}$ is, in fact, close to $1$ then even more can be said.

First recall Corollary \ref{cor.indsmall} where we showed that $\|1_A\|_{A(G)}$ may, in fact, be as small as the above trivial lower bound.  Curiously, it turns out that there is a jump in the possible values of the algebra norm after $1$.  The following proposition is the content of this section and confirms this fact.
\begin{proposition}\label{prop.small}
Suppose that $G$ is a finite group and $A \subset G$ is non-empty and has $\|1_A\|_{A(G)} < 1+1/750$. Then there is a subgroup $H \leq G$ and an element $x \in G$ such that $1_A=1_{Hx}$. 
\end{proposition}
The proof itself can be seen as a sort of very simplified model for the wider work of the paper.  To begin with we note the following lemma which is essentially due to Fournier \cite{fou::} and is a sort of Balog-Szemer{\'e}di-Fre{\u\i}man theorem for very large energy sets.

Before beginning the proof it will be useful to recall the symmetry set notation of Tao and Vu \cite{taovu::}.  Suppose that $G$ is a finite group, $A \subset G$ and $\eta \in (0,1]$ is a parameter. Then the \emph{symmetry set of $A$ at threshold $\eta$} is 
\begin{equation*}
\Sym_{\eta}(A):=\{x \in G: 1_A \ast 1_{A^{-1}}(x) \geq \eta\mu_G(A)\}.
\end{equation*}
\begin{lemma}\label{lem.fourn}
Suppose that $G$ is a finite group, $A \subset G$ is non-empty with $\|1_A\ast 1_{A^{-1}}\|_{L^2(\mu_G)}^2 \geq (1-c)\mu_G(A)^3$ and $\eta \in [12c,1/12)$ is a parameter.  Then there is a subgroup $H \leq G$ and some $x \in G$ such that
\begin{equation*}
\mu_G(H) \geq (1-c\eta^{-1})\mu_G(A) \textrm{ and } \mu_G(A \cap Hx) \geq (1-2\eta)\mu_G(H).
\end{equation*}
\end{lemma}
\begin{proof}
Write $\alpha$ for the density of $A$ in $G$ and put $K:=\Sym_{1-\eta}(A)$.  Now, if $x,y \in H$ then
\begin{equation*}
\mu_G(A \cap xA) > (1-\eta)\alpha \textrm{ and } \mu_G(A \cap yA) > (1-\eta)\alpha.
\end{equation*}
It follows that $1_A \ast 1_{A^{-1}}(xy) > (1-2\eta)\alpha$ by the pigeonhole principle (or, more formally, Lemma \ref{lem.sub-mult}) and so we put $K':=\Sym_{1-2\eta}(A)$, and note that $K^2 \subset K'$.  

Now we shall estimate the size of $K$:
\begin{eqnarray*}
\int{(1_A \ast 1_{A^{-1}})^2d\mu_G} & = & \int_{G \setminus K}{(1_A \ast 1_{A^{-1}})^2d\mu_G}+\int_{K}{(1_A \ast 1_{A^{-1}})^2d\mu_G}\\ & \leq & (1-\eta)\alpha. \int_{G \setminus K}{1_A \ast 1_{A^{-1}}d\mu_G}+\alpha.\int_{K}{1_A \ast 1_{A^{-1}}d\mu_G}\\ & = &(1-\eta)\alpha^3 +\eta\alpha\int_{K}{1_A\ast 1_{A^{-1}}d\mu_G},
\end{eqnarray*}
whence
\begin{equation*}
\int_K{1_A \ast 1_{A^{-1}}d\mu_G} \geq (1-\eta^{-1}c)\alpha^2,
\end{equation*}
and it follows that $\mu_G(K) \geq (1-\eta^{-1}c)\alpha$ form the trivial upper bound on the integrand.  On the other hand
\begin{equation*}
\mu_G(K').(1-2\eta)\alpha \leq \int{1_A \ast 1_{A^{-1}}d\mu_G} \leq \alpha^2,
\end{equation*}
so that $\mu_G(K') \leq (1-2\eta)^{-1}\alpha$.  It follows that
\begin{equation*}
\mu_G(K^2) \leq (1-2\eta)^{-1}(1-\eta^{-1}c)^{-1}\mu_G(K)< 3/2\mu_G(K)
\end{equation*}
since $12c< \eta < 1/12$. It follows from \cite[Exercise 2.6.5]{taovu::} (which is, in turn, from \cite{aba::}) that $H:=K^2$ is a subgroup of $G$, and hence that
\begin{equation*}
\alpha \|\mu_{K^2}\ast 1_A\|_{L^\infty(\mu_G)} \geq \langle \mu_{K^2},1_A \ast 1_{A^{-1}}\rangle_{L^2(\mu_G)} \geq (1-2\eta)\alpha.
\end{equation*}
We conclude that there is some $x$ such that $\mu_G(H\cap xA^{-1}) > (1-2\eta)\mu_G(H)$, and since
\begin{equation*}
\mu_G(H\cap xA^{-1}) =\mu_G(H^{-1}\cap Ax^{-1}) = \mu_G(H \cap Ax^{-1})
\end{equation*}
the result follows. 
\end{proof}
\begin{proof}[Proof of Proposition \ref{prop.small}]
Let $v_1,\dots,v_N$ be a Fourier basis for $\widetilde{1_A}=1_{A^{-1}}$.  It follows from Parseval's theorem that
\begin{eqnarray*}
\|1_A \ast 1_{A^{-1}}\|_{L^2(\mu_G)}^2 & = & \sum_{i=1}^N{\langle L_{1_A\ast\widetilde{1_A}}v_i,L_{1_A\ast\widetilde{1_A}}v_i\rangle_{L^2(\mu_G)}}\\ & = &  \sum_{i=1}^N{\langle L_{\widetilde{1_A}}^*L_{\widetilde{1_A}}v_i,L_{\widetilde{1_A}}^*L_{\widetilde{1_A}}v_i\rangle_{L^2(\mu_G)}}\\ &= & \sum_{i=1}^N{|s_i(\widetilde{1_A})|^4}.
\end{eqnarray*}
It follows from H{\"o}lder's inequality that
\begin{equation*}
\left(\sum_{i=1}^N{|s_i(\widetilde{1_A})|^2}\right)^3 \leq \left(\sum_{i=1}^N{|s_i(\widetilde{1_A})|^4}\right)\left(\sum_{i=1}^N{|s_i(\widetilde{1_A})|}\right)^2.
\end{equation*}
On the other hand by the explicit formula for the $A(G)$-norm and Lemma \ref{lem.invariance} that
\begin{equation*}
\|1_A\|_{A(G)} =\|\widetilde{1_A}\|_{A(G)}= \sum_{i=1}^N{|s_i(\widetilde{1_A})|},
\end{equation*}
and by Parseval's theorem that
\begin{equation*}
\|\widetilde{1_A}\|_{L^2(\mu_G)}^2 = \sum_{i=1}^N{\langle L_{\widetilde{1_A}}v_i,L_{\widetilde{1_A}}v_i\rangle_{L^2(\mu_G)}}=\sum_{i=1}^N{|s_i(\widetilde{1_A})|^2}.
\end{equation*}
Combining all these we see that
\begin{equation*}
\|1_A \ast 1_{A^{-1}}\|_{L^2(\mu_G)}^2 \geq \mu_G(A)^3/\|1_A\|_{A(G)}^2.
\end{equation*}
Since $\|1_A\|_{A(G)} \leq 1+1/750$ we can apply Lemma \ref{lem.fourn} with $\eta=1/20$ to get a group $H \leq G$ such that
\begin{equation*}
\mu_G(H) \geq 19\mu_G(A)/20 \textrm{ and } \mu_G(A \cap Hx) \geq 9\mu_G(H)/10
\end{equation*}
 for some $x \in G$.  By Lemma \ref{lem.invariance} we may translate $A$ without changing the hypotheses of the proposition so without loss of generality we assume that $Hx=H$. It turns out that $A=H$ as we shall now show.

Suppose that $x' \in A \setminus H$.  Then
\begin{eqnarray*}
1_{A} \ast 1_{H}(x')& =& \mu_G(A\cap x'^{-1}H)\\& \leq& \mu_G(A) - \mu_G(A \cap H)\\ &  \leq& \frac{20}{19}\mu_G(H) - \frac{9}{10}\mu_G(H)= 29\mu_G(H)/190.
\end{eqnarray*}
It follows that
\begin{equation*}
\|1_{A} - 1_{A} \ast \mu_{H}\|_{L^\infty(\mu_G)}\geq 161/190.
\end{equation*}
On the other hand $1_{A}\ast \mu_{H}(x^{-1}) \geq 9/10$, whence
\begin{equation*}
\|1_{A} - 1_{A} \ast \mu_{H}\|_{A(G)} \geq 161/190 \textrm{ and } \|1_{A} \ast \mu_{H}\|_{A(G)} \geq 9/10,
\end{equation*}
by Lemma \ref{lem.linfag}. This leads to a contradiction by Lemma \ref{lem.decompmass} and we conclude that $A \subset H$. 

In the other direction, if $x' \in H \setminus A$ then
\begin{equation*}
|(1_A-1_A \ast \mu_H)(x')|\geq \mu_H(A \cap x'^{-1}H)=\mu_H(A \cap H) \geq 9/10.
\end{equation*}
Similarly
\begin{equation*}
|1_A \ast \mu_H(x)| \geq \mu_H(A \cap x'^{-1}H)=\mu_H(A \cap H) \geq 9/10,
\end{equation*}
whence
\begin{equation*}
\|1_{A} - 1_{A} \ast \mu_{H}\|_{A(G)} \geq 9/10 \textrm{ and } \|1_{A} \ast \mu_{H}\|_{A(G)} \geq 9/10,
\end{equation*}
by Lemma \ref{lem.linfag}.  Again this leads to a contradiction by Lemma \ref{lem.decompmass} and we conclude that $A=H$ completing the result.
\end{proof}
Of course with care one can considerably improve the constant $1/750$ in the above, but even then the conclusion is not strong.  In the abelian setting this sort of problem has been considered by Saeki in \cite{sae::} and \cite{sae::0} who has given a much stronger answer through the construction of cleverly chosen dual functions.  It does not seem impossible that such an approach would also work here although we have not tried it.

It should also be remarked that there is a parallel in an area of additive combinatorics called the structure theory of set addition (see \cite{fre::0}).  There one finds theorems describing the structure of sets with small but slowly increasing doubling, and then much stronger theorems describing sets with doubling at most $3$, say.  See, for example, \cite{fre::,hampla::} and \cite{desfre::}.

\section{Approximate groups: an introduction to multiplicative pairs}\label{sec.ag}

In this section we introduce the notion of `approximate group' which we shall be using in this paper.  There are a number of candidates for such structures in the literature and for a survey the reader may wish to consult \cite{gre::05}.  Our candidate is motivated by some ideas of Bourgain from \cite{bou::5} and we now turn to its introduction.

One begins by observing that many sets are symmetric neighbourhoods of the identity; a group is such a set which is also closed.  This additional closure requirement can be very restrictive if, for example, $G$ is a cyclic group of prime order.  Bourgain noted that it may be relaxed to an approximate closure condition which may be summarised by saying that if you take a small ball and add it to a large ball, then most of the time you remain in the large ball.

The prototypical examples of the above idea are $\delta$-balls in $\R^d$: let $B_\delta$ be the ball (cube) centred at the origin of side length $\delta$ in the $\ell^\infty$-norm.  It is easy to see that $0 \in B_\delta$ and $-B_\delta = B_\delta$.  Unfortunately these balls are not closed as $B_\delta +B_\delta = B_{2\delta}$ which is, in general, much larger than $B_\delta$.

This problem is solved by introducing an asymmetry in the group operation: instead of perturbing $B_\delta$ by itself, we perturb it by $B_{\delta'}$ for some $\delta'$ much smaller than $\delta$.  In this case we have
\begin{equation*}
B_\delta+B_{\delta'} \subset B_{\delta+\delta'} \textrm{ and } B_{\delta-\delta'} + B_{\delta'} \subset B_\delta
\end{equation*}
and recover a sort of approximate closure property in the sense that
\begin{equation*}
\frac{\mu(B_{\delta+\delta'}\setminus B_{\delta-\delta'})}{\mu(B_\delta)}= O(d\delta'\delta^{-1}),
\end{equation*}
where $\mu$ denotes Lebesgue measure on $\R^d$.  Fortunately this notion makes sense not just in abelian groups but also in non-abelian groups.

Suppose that $G$ is a finite group and $r \in \N$. We say that $\mathcal{B}=(B,B')$ is an \emph{$r$-multiplicative pair} with \emph{ground set $B$} and \emph{perturbation set $B'$} if
\begin{enumerate}
\item $B$ and $B'$ are symmetric neighbourhoods of the identity;
\item there are symmetric neighbourhoods of the identity $B^+$ and $B^-$ such that
\begin{equation*}
B'^rB^-B'^r \subset B \textrm{ and } B'^rBB'^r \subset B^+.
\end{equation*}
\end{enumerate}
We say that $\mathcal{B}$ is \emph{$\epsilon$-closed} if
\begin{equation*}
\mu_G(B^+\setminus B^-) \leq \epsilon\mu_G(B),
\end{equation*}
and \emph{$c$-thick} if
\begin{equation*}
\mu_G(B') \geq c\mu_G(B).
\end{equation*}
Given an $\epsilon$-closed $r$-multiplicative pair $\mathcal{B}$ with ground set $B$, the sets $B^+$ and $B^-$ are not unambiguously defined.  Of course our arguments only ever use the above properties of these sets so this ambiguity does not present a problem.

The parameter $r$ essentially tells us how many times we are `allowed to' multiply elements of $B$ by elements of $B'$, and the level of closure determines the extent to which we remain in $B$ when doing this.  Ideally we should like to be able to scale up $r$ by a factor $k$ at the cost of replacing $\epsilon$ with $O(k\epsilon)$.  We cannot quite do this but in practice it is a good heuristic to keep in mind.

Typically $r$ will be $O(1)$, $\epsilon\rightarrow 0$ very slowly and $c\rightarrow 0$ as $\epsilon \rightarrow 0$ or $r \rightarrow \infty$.   It is instructive to consider a few examples.
\begin{example}[Subgroups]  Suppose that $H \leq G$.  Then $\mathcal{B}=(B,B'):=(H,H)$ is easily seen to be a $0$-closed $1$-thick $\infty$-multiplicative pair on setting $B^+:=B^-:=H$.
\end{example}
\begin{example}[Unions of cosets]  Suppose that $H \leq G$ and $A$ is a symmetric neighbourhood of the identity of size $k$ in the normaliser of $H$, so that $aH=Ha$ for all $a \in A$. Then $\mathcal{B}=(B,B'):=(AH,H)$ is easily seen to be a $0$-closed $k^{-1}$-thick $\infty$-multiplicative pair on setting $B^+:=B^-:=AH$.
\end{example}
\begin{example}[Subpairs] Suppose that $\mathcal{B}=(B,B')$ is an $\epsilon$-closed $c$-thick $r$-multiplicative pair, $B''^k \subset B'$ is a symmetric neighbourhood of the identity with $\mu_G(B'') \geq c'\mu_G(B)$, $\epsilon' \geq \epsilon$, and $r' \leq r$ is a natural. Then the pair $\mathcal{B}':=(B,B''^k)$ is an $\epsilon'$-closed $c'$-thick $\lfloor r'/k\rfloor$-multiplicative pair.
\end{example}
\begin{example}[Conjugate pairs]
Suppose that $\mathcal{B}=(B,B')$ is an $\epsilon$-closed $c$-thick $r$-multiplicative pair.  Then so is $\mathcal{B}^y:=(yBy^{-1},yB'y^{-1})$ for all $y \in G$.
\end{example}
\begin{example}[Product sets]  Suppose that $A \subset G$ is a symmetric neighbourhood of the identity and $\mu_G(A^3) \leq K\mu_G(A)$.  Then $\mathcal{B}=(B,B'):=(A^{2r},A)$ is an $r$-multiplicative pair on setting $B^+:=A^{4r}$ and $B^-:=\{1_G\}$: all the sets are symmetric neighbourhoods of the identity and 
\begin{equation*}
B'^rBB'^r=A^r.A^{2r}.A^r \subset A^{4r}=B^+
\end{equation*}
and
\begin{equation*}
B'^rB^{-}B'^r = A^r.\{1_G\}.A^r \subset A^{2r}=B.
\end{equation*}
Of course, here the closure parameter will be at least $1/2$ unless $B$ is a subgroup (by \cite[Exercise 2.6.5]{taovu::}); the thickness is $K^{-O_r(1)}$ by Lemma \ref{lem.cover}.
\end{example}
\begin{example}[Sets with polynomial growth]\label{eg.pg}  Suppose that $A \subset G$ is a symmetric neighbourhood of the identity such that
\begin{equation*}
\mu_G(A^n) \leq Cn^d\mu_G(A) \textrm{ for all } n \geq 1;
\end{equation*}
$A$ is a set of polynomial growth.  By the pigeonhole principle and the polynomial growth condition there is some $2r \leq n \leq O_{\epsilon,r,d,C}(1)$ such that
\begin{equation*}
\mu_G(A^{n+2r}) \leq (1+\epsilon)\mu_G(A^{n-2r}).
\end{equation*}
Then $\mathcal{B}=(B,B'):=(A^{n},A)$ is clearly seen to be an $r$-multiplicative pair by taking $B^+:=A^{n+2r}$ and $B^-:=A^{n-2r}$.  Moreover,
\begin{eqnarray*}
\mu_G(A^{n+2r} \setminus A^{n-2r}) &\leq &\mu_G(A^{n+2r}) - \mu_G(A^{n-2r})\\ & \leq & \epsilon \mu_G(A^{n-2r}) \leq \epsilon \mu_G(A^{n}),
\end{eqnarray*}
 whence $\mathcal{B}$ is $\epsilon$-closed, and
\begin{equation*}
\mu_{A^n}(A) = \frac{\mu_G(A)}{\mu_G(A^n)} \geq 1/Cn^d = \Omega_{\epsilon,r,d,C}(1)
\end{equation*}
so $\mathcal{B}$ is $\Omega_{\epsilon,r,d,C}(1)$-thick.
\end{example}
This last example behaves like a discrete version of a Bourgain system.  While we shall not work with Bourgain systems explicitly in this paper, we shall consider symmetry sets which are a type of Bourgain system and naturally give rise to multiplicative pairs in the same way.

When $G$ is abelian all multi-dimensional coset progressions and Bohr sets give rise to additive (multiplicative) pairs in a fairly simple and natural way, and it turns out that in that setting the converse is essentially true by the Green-Ruzsa-Fre{\u\i}man theorem from \cite{greruz::0}. Unfortunately there is no known generalisation of this theorem to arbitrary finite groups, although many attempts have been made: see \cite{bregre::,bregre::0,fiskatpen::,hru::} and \cite{tao::9} for details.  

Although the Green-Ruzsa-Fre{\u\i}man theorem can be used as above to classify additive (multiplicative) pairs its utility comes not from this, but rather from the fact that it shows that \emph{any} set of small doubling actually correlates with a multiplicative pair.  In fact it shows the much stronger statement that any set with small doubling correlates with a multi-dimensional coset progression but, as mentioned, no such result is known in the non-abelian setting. 

Our programme now is two-fold: we shall prove a weak Fre{\u\i}man-type result which will show that any set with small doubling correlates with a multiplicative pair in general finite groups, and we shall develop some analysis relative to the rather weak structure of a multiplicative pair.  More specifically we have the following sections developing these two goals.
\begin{enumerate}
\item In \S\S\ref{sec.frei},\ref{sec.freimult}\verb!&!\ref{sec.containment} we prove our Fre{\u\i}man-type results.  The first of these sections is the basic result, the second contains a multi-scale generalisation and the third effects the passage between containment (of a multiplicative pair) and correlation (with a multiplicaive pair).
\item \S\ref{sec.locpsa}\verb!&!\ref{sec.norm} contain the basic lemmas for physical space analysis on multiplicative pairs and how to normalise them so that they behave more like normal subgroups.
\item \S\S\ref{sec.locft},\ref{sec.locspec}\verb!&!\ref{sec.anallarge} introduce the techniques for spectral analysis on multiplicative pairs and establish the basic results regarding the large spectrum.  (They all feed into \S\ref{sec.bog} where a Bogolio{\'u}boff-type result is proved, which provides a good example of the application of the ideas from these sections.)
\item \S\ref{sec.specmp} governs the spectral behaviour of the multiplicative pair itself and is arguably the last section on the `general theory of multiplicative pairs'.
\end{enumerate}

\section{Symmetry sets and a Fre{\u\i}man-type theorem}\label{sec.frei}

It is the objective of this section to show how symmetry sets give rise to multiplicative pairs. The importance of symmetry sets has been clear for a while and a good introduction in the abelian setting may be found in \S2.6 of the book \cite{taovu::} of Tao and Vu.  

We begin by explaining how they give rise to multiplicative pairs: it is immediate that $\Sym_\eta(A)$ is a symmetric neighbourhood of the identity contained in $AA^{-1}$, and that we have the nesting property
\begin{equation*}
\Sym_\eta(A) \subset \Sym_{\eta'}(A) \textrm{ whenever } \eta \geq \eta'.
\end{equation*}
At this point we can declare our candidate for a multiplicative pair: we shall take a certain set $A$ and define
\begin{equation*}
\mathcal{B}=(B,B'):=(\Sym_\delta(A),\Sym_{1-\eta'}(A))
\end{equation*}
and
\begin{equation*}
B^+:= \Sym_{\delta-2\eta}(A) \textrm{ and } B^-:=\Sym_{\delta + 2\eta}(A)
\end{equation*}
for some suitably chosen $\delta$, much smaller $\eta$, and still smaller $\eta'$.  We now have four things we wish to show:
\begin{enumerate}
\item that $B$ is a large part of $AA^{-1}$ for suitably chosen $\delta$ if $A$ has small doubling, which we show in Lemma \ref{lem.symsize};
\item that $\mathcal{B}$ is an $r$-multiplicative pair for suitably chosen $\eta'$ in terms of $\eta$, which we show in Lemma \ref{lem.sub-mult};
\item that $\mathcal{B}$ is $\epsilon$-closed for suitably chosen $\delta$ and $\eta$, which we show in Lemma \ref{lem.ubr};
\item and that $\mathcal{B}$ is $c$-thick for suitably chosen $\eta'$ and $A$, which will follow from Proposition \ref{prop.intit}.
\end{enumerate}
Of course we should also like the `suitable choices' to be compatible!  We then combine all this in the main result of the section: Proposition \ref{prop.frcon}.

First we show that if $A$ has large multiplicative energy then $\Sym_\delta(A)$ is large for $\delta$ sufficiently small in terms of the energy constant.
\begin{lemma}[Largeness of symmetry sets]\label{lem.symsize}
Suppose that $G$ is a finite group, $A \subset G$ has $\|1_A \ast 1_{A^{-1}}\|_{L^2(\mu_G)}^2 \geq c\mu_G(A)^3$ and $\delta \in (0,1]$.  Then
\begin{equation*}
\min\{\mu_G(AA^{-1}),\delta^{-1}\mu_G(A)\} \geq \mu_G(\Sym_{\delta}(A)) \geq (c-\delta)\mu_G(A).
\end{equation*}
\end{lemma}
\begin{proof}
The proof is an immediate calculation. First the upper bound: since the set $\Sym_\delta(A)$ is a subset of $AA^{-1}$ the first upper bound is trivial and by its definition we have
\begin{equation*}
\mu_G(\Sym_\delta(A)).\delta \mu_G(A) \leq \int{1_A \ast 1_{A^{-1}}d\mu_G} = \mu_G(A)\mu_G(A^{-1}),
\end{equation*}
from which the second follows immediately.  Now, the lower bound: again by definition of $\Sym_\delta(A)$ we have
\begin{equation*}
\mu_G(\Sym_\delta(A))\mu_G(A)^2 + \delta\mu_G(A).\int{1_A \ast 1_{A^{-1}}d\mu_G} \geq  \int{1_A \ast 1_{A^{-1}}^2d\mu_G}.
\end{equation*}
However, the right hand side is at least $c\mu_G(A)^3$ we get the desired bound since
\begin{equation*}
\int{1_A \ast 1_{A^{-1}}d\mu_G} =\mu_G(A)\mu_G(A^{-1})=\mu_G(A)^2.
\end{equation*}
The lemma is proved.
\end{proof}
Generically $\Sym_{\delta}(A)$ may just contain the element $1_G$ for $\delta>c$: consider, for example, the situation when $A$ is a random subset of $G$ with density $c$.  In this case $1_A \ast 1_{A^{-1}}$ will almost always take the value $\mu_G(A)^2=c\mu_G(A)$. 

The next lemma establishes an iterated containment property for symmetry sets. 
\begin{lemma}[Sub-multiplicativity of symmetry sets]\label{lem.sub-mult}
Suppose that $G$ is a finite group, $A \subset G$ and $\delta,\epsilon\in (0,1]$. Then
\begin{equation*}
\Sym_{\delta}(A)\Sym_{1-\epsilon}(A) \subset \Sym_{\delta-\epsilon}(A).
\end{equation*}
\end{lemma}
\begin{proof}
Suppose that $s \in \Sym_\delta(A)$ and $t \in \Sym_{1-\epsilon}(A)$ so that
\begin{equation*}
\mu_G(A \cap sA)=1_A \ast 1_{A^{-1}}(s) \geq \delta \mu_G(A)
\end{equation*}
and
\begin{equation*}
\mu_G(A\cap tA) =1_A \ast 1_{A^{-1}}(t) \geq (1-\epsilon)\mu_G(A).
\end{equation*}
Now $B \cap C \supset (B \cap D) \setminus (D \setminus C)$ for all sets $B,C,D$, so 
\begin{equation*}
\mu_G(A\cap stA) \geq \mu_G(A \cap sA)- \mu_G(sA \setminus stA),
\end{equation*}
whence
\begin{equation*}
\mu_G(A\cap stA) \geq \mu_G(A\cap sA) - \mu_G(A \setminus tA) \geq (\delta-\epsilon)\mu_G(A).
\end{equation*}
It follows that $st \in \Sym_{\delta - \epsilon}(A)$ as required.
\end{proof}
Note the symmetry in the statement of the lemma if we write $\delta=1-\epsilon'$; then it is exactly like the first half of \cite[Lemma 2.33]{taovu::}.

We now go on to prove that if $A$ has large multiplicative energy then there must be two symmetry sets with close threshold of similar size -- these are the candidates for $B^+$ and $B^-$ in our multiplicative pair.
\begin{lemma}\label{lem.ubr}
Suppose that $G$ is a finite group and $A \subset G$ has $\|1_A \ast 1_{A^{-1}}\|_{L^2(\mu_G)}^2 \geq c\mu_G(A)^3$. Then there is some $c' \in (c/4,c/2]$ such that
\begin{equation*}
\left|\frac{\mu_G(\Sym_{c'(1+\eta)}(A))}{\mu_G(\Sym_{c'}(A))} -1\right| \leq C_{\mathcal{R}}|\eta| (1+\log c^{-1})
\end{equation*}
whenever $|\eta| \leq c_{\mathcal{R}}/(1+\log c^{-1})$ for some absolute constants $C_{\mathcal{R}},c_{\mathcal{R}}>0$.
\end{lemma}
\begin{proof}
Let $f:[0,1]\rightarrow \R$ be the function defined by
\begin{equation*}
f(\delta):=\frac{1}{\log 8c^{-2}}\log \mu_G(\Sym_{c/2^{1+\delta}}(A)).
\end{equation*}
By nesting we have that $f$ is an increasing function of $\delta$, and by Lemma \ref{lem.symsize} we have that
\begin{equation*}
f(0) \geq \frac{1}{\log 8c^{-2}}\log (c\mu_G(A)/2) \textrm{ and } f(1) \leq \frac{1}{\log 8c^{-2}}\log (4c^{-1}\mu_G(A)),
\end{equation*}
so $f(1)-f(0)\leq 1$.  We claim that there is some $\delta \in [1/6,5/6]$ such that
\begin{equation*}
|f(\delta+\delta') - f(\delta)| \leq 3|\delta'| \textrm{ whenever } |\delta'| \leq 1/6.
\end{equation*}
In not, then for every $\delta \in [1/6,5/6]$ there is some interval $I_\delta$ of length at most $1/6$ having one endpoint equal to $\delta$ and
\begin{equation*}
\int_{I_\delta}{df} > \int_{I_\delta}{3dx}.
\end{equation*}
These intervals cover $[1/6,5/6]$ which has length $2/3$, whence by a covering lemma (e.g. \cite[Lemma 3.4]{grekon::}) lets us pass to a disjoint collection of intervals $I_{\delta_1},\dots,I_{\delta_n}$ of total length at least $1/3$.  However,
\begin{equation*}
1 \geq \int_0^1{df} \geq \sum_{i=1}^n{\int_{I_{\delta_i}}{df}} > \sum_{i=1}^n{\int_{I_{\delta_i}}{3dx}} \geq 1.
\end{equation*}
This contradiction proves the claim and there is such a $\delta \in [1/6,5/6]$.  Put $c'=c/2^{1+\delta}$ and note that
\begin{equation*}
|\log \mu_G(\Sym_{c'/2^{\delta'}(A))} - \log \mu_G(\Sym_{c'}(A))| \leq 3\delta'\log 8c^{-2}
\end{equation*}
whenever $|\delta'| \leq 1/6$.  It follows that
\begin{equation*}
(8c^{-2})^{-3\delta'}\leq \frac{\mu_G(\Sym_{c'/2^{\delta'}}(A))}{\mu_G(\Sym_{c'}(A)) } \leq (8c^{-2})^{3\delta'},
\end{equation*}
from which we get the result.
\end{proof}
The above proof is the same as the now ubiquitous Bourgain regularity argument from \cite{bou::5}, and we could have made do with a straightforward pigeonhole argument (see, Example \ref{eg.pg} for an idea of how) as our later results will not be able to preserve the fact that $c'$ does not depend on which particular $\eta$ we choose in the allowed range.

We now turn to the fourth objective of finding a supply of large symmetry sets with threshold close to $1$, so as to ensure that our multiplicative pair is thick.  This is provided by the central result of \cite{san::03} which we now recall.
\begin{proposition}[{\cite[Proposition 1.3]{san::03}}]\label{prop.intit}
Suppose that $G$ is a finite group, $A$ is a non-empty subset of $G$ with $\mu_G(A^2) \leq K\mu_G(A)$, and $\epsilon \in (0,1]$ is a parameter. Then there is a non-empty set $A'\subset A$ such that
\begin{equation*}
\mu_G(\Sym_{1-\epsilon}(A'A)) \geq \exp(-K^{O(\epsilon^{-1})})\mu_G(A).
\end{equation*}
\end{proposition}
Finally we are in position to prove our weak Fre{\u\i}man-type theorem.
\begin{proposition}\label{prop.frcon}
Suppose that $G$ is a finite group, $A \subset G$ is symmetric with $\mu_G(A^2) \leq K\mu_G(A)$ and $r \in \N$ and $\epsilon \in (0,1]$ are parameters. Then there is a positive real $c=\Omega_{K,r,\epsilon}(1)$ and an $\epsilon$-closed, $c$-thick $r$-multiplicative pair $\mathcal{B}$ with ground set $B$ such that
\begin{equation*}
B \subset A^4 \textrm{ and }\mu_G(B)\geq\Omega(K^{-O(1)}\mu_G(A)).
\end{equation*}
\end{proposition}
\begin{proof}
Recall that $C_{\mathcal{R}}$ and $c_{\mathcal{R}}$ are the constants from Lemma \ref{lem.ubr} and put
\begin{equation*}
\eta= \min \{1,C_{\mathcal{R}},c_{\mathcal{R}}\}\epsilon/(1+\log 4K^4) = \Omega(\epsilon/(1+\log K)).
\end{equation*}
Apply Proposition \ref{prop.intit} with parameter $\eta/16rK^4$ to get a non-empty set $A' \subset A$ such that
\begin{equation*}
\mu_G(\Sym_{1-\eta/16rK^4}(A'A)) \geq \Omega_{K,r,\epsilon}(\mu_G(A)).
\end{equation*}
We set $B':=\Sym_{1-\eta/16rK^4}(A'A)$ and by Lemma \ref{lem.sub-mult} we have that
\begin{equation*}
B'^r\subset \Sym_{1-\eta/16K^4}(A'A).
\end{equation*}
It is easy to check that $A'A$ has large energy.  In particular let $a \in A'$ (such exists since $A'$ is non-empty) we have
\begin{equation*}
\|1_{A'A} \ast 1_{(A'A)^{-1}}\|_{L^2(\mu_G)}^2 \geq \|1_{aA} \ast 1_{(aA)^{-1}}\|_{L^2(\mu_G)}^2 =\|1_{A} \ast 1_{A^{-1}}\|_{L^2(\mu_G)}^2,
\end{equation*}
by change of variables $x =aza^{-1}$ in the second integral.  On the other hand
\begin{equation*}
\|1_{A} \ast 1_{A^{-1}}\|_{L^2(\mu_G)}^2\geq \frac{1}{\mu_G(AA^{-1})}\left(\int{1_A \ast 1_{A^{-1}}d\mu_G}\right)^2\geq \mu_G(A'A)^3/K^4
\end{equation*}
since $A$ is symmetric, $\mu_G(A^2) \leq K\mu_G(A)$ and $A'\subset A$.  We conclude that
\begin{equation*}
\|1_{A'A} \ast 1_{(A'A)^{-1}}\|_{L^2(\mu_G)}^2 \geq \mu_G(A'A)^3/K^4.
\end{equation*}
Finally we apply Lemma \ref{lem.ubr} to get $c' \in (1/4K^4,1/2K^4]$ such that
\begin{equation*}
 \mu_G(\Sym_{c'(1+\eta')}(A'A)) \leq (1+C_{\mathcal{R}}\eta' (1+\log 4K^4)) \mu_G(\Sym_{c'}(A'A))
\end{equation*}
for all $\eta' \leq c_{\mathcal{R}}/(1+\log 4K^4)$.  In particular, given our choice of $\eta$ 
\begin{equation}\label{eqn.epsc}
 \mu_G(\Sym_{c'(1+\eta)}(A'A)) \leq (1+\epsilon)\mu_G(\Sym_{c'}(A'A)).
\end{equation}
We set
\begin{equation*}
B^+:=\Sym_{c'}(A'A), B^-:=\Sym_{c'(1+\eta)}(A'A) \textrm{ and }B:=\Sym_{c'(1+\eta/2)}(A'A),
\end{equation*}
and now verify that $\mathcal{B}$ has the desired properties.  By nesting and the lower bound on $c'$ we have that
\begin{equation*}
B'^r\subset \Sym_{1-\eta/16K^4}(A'A)\subset \Sym_{1-c'\eta/4}(A'A),
\end{equation*}
and so it follows from Lemma \ref{lem.sub-mult} that $\mathcal{B}$ is $r$-wide.  By (\ref{eqn.epsc}) we see that $\mathcal{B}$ is $\epsilon$-closed.  By Lemma \ref{lem.symsize} we see that
\begin{equation*}
(c'(1+\eta/2))^{-1}\mu_G(A'A) \geq \mu_G(B) \geq (1/K^4 - c'(1+\eta/2))\mu_G(A'A),
\end{equation*}
whence
\begin{equation*}
\mu_G(B) = O_K(\mu_G(A)) \textrm{ and } \mu_G(B) = \Omega(K^{-4}\mu_G(A)).
\end{equation*}
The first of these coupled with the lower bound on the size of $B'$ shows that $\mathcal{B}$ is $\Omega_{K,r,\epsilon}(1)$-thick.  The second of these establishes the lower bound on the size of $B$ and, finally,
\begin{equation*}
B=\Sym_{c'(1+\eta/2)}(A'A) \subset A'A(A'A)^{-1} \leq A^4
\end{equation*}
since $A' \subset A$ and $A$ is symmetric.  The result has been proved.
\end{proof}
It may be worth recalling Ruzsa's proof of Fre{\u\i}man's theorem \cite{ruz::9} at this  point, where he shows that the four-fold sumset of a set with small doubling contains a large multi-dimensional arithmetic progression.  It is a short covering argument to pass from this to Fre{\u\i}man's theorem.  We shall not take this second step, instead proceeding along the lines of \S\ref{sec.containment} to show that $A$ correlates with this multiplicative pair.

It is also worth recording the bounds in this theorem, which follow directly from the application of Proposition \ref{prop.intit}: we may take $c^{-1}$ to be doubly exponential in $O(r\epsilon^{-1}K^{O(1)})$.

\section{A Fre{\u\i}man-type theorem with multiple scale multiplicative pairs}\label{sec.freimult}

As it stands Proposition \ref{prop.frcon} outputs one $\epsilon$-closed $r$-multiplicative pair $\mathcal{B}$ with ground set $B$ and perturbation set $B'$.  However, sometimes it is useful to have another perturbation set $B''$ such that $\mathcal{B}':=(B,B'')$ is an $\epsilon'$-closed $r'$-multiplicative pair where $\epsilon'$ and $r'$ may depend on the thickness of $\mathcal{B}$.  

In fact we shall need a whole system of pairs which behaves somewhat like a restricted range Bourgain system.  Such a result does not follow from repeated applications of Proposition \ref{prop.frcon}, instead we have to use a large pigeonhole argument coupled with the ingredients that went into Proposition \ref{prop.frcon}. 
\begin{proposition}\label{prop.regapp}
Suppose that $G$ is a finite group, $A$ is symmetric and $\mu_G(A^2) \leq K\mu_G(A)$, $r:(0,1]\rightarrow \N$ is a monotone decreasing function, $\epsilon:(0,1]\rightarrow (0,1]$ is a monotone increasing function and $J \in \N$ is a parameter.  Then there are positive reals $(c_i)_{i=1}^J$ with $c_j=\Omega_{K,r,\epsilon,j}(1)$ and a nested sequence of sets $(B_i)_{i=0}^J$ such that $\mathcal{B}_{i,j}=(B_i,B_j)$ is an $\epsilon(c_{j-1})$-closed, $c_j$-thick $r(c_{j-1})$-multiplicative pair whenever $i<j$ and
\begin{equation*}
B_0 \subset A^4\textrm{ and } \mu_G(B_0)=\Omega_K(\mu_G(A)).
\end{equation*}
\end{proposition}
\begin{proof}
We begin by defining auxiliary sequences of non-empty sets $(D_i)_{i=0}^J$, reals $(c_i)_{i=0}^J$ and $(K_i)_{i=0}^J$, and naturals $(k_i)_{i=1}^J$. The reals are defined directly in terms of these sets by
\begin{equation*}
c_{i}:=\mu_G(D_{i}^4)/\mu_G(D_{0}^{12}) \textrm{ and } K_i:=\mu_G(D_i^{12})/\mu_G(D_i),
\end{equation*}
which then lets us define the naturals by
\begin{equation*}
k_{i+1}:=\lceil(1+ \log K_i)/\epsilon(c_i)\rceil (2r(c_{i})+1).
\end{equation*}
To begin the definition of the sets (which will be inductive) apply Proposition \ref{prop.intit} to the set $A$ to get a non-empty set $A' \subset A$ such that
\begin{equation*}
\mu_G(\Sym_{1-1/13}(A'A))= \Omega_{K}(\mu_G(A)).
\end{equation*}
Set $D_0:=\Sym_{1-1/13}(A'A)$ and note that since $A' \subset A$ and $A$ is symmetric we have $D_0 \subset A^4$, and by Lemma \ref{lem.symsize} and Lemma \ref{lem.sub-mult}
\begin{equation*}
\mu_G(D_0^{12}) \leq 13\mu_G(A'.A) = O(K\mu_G(A)).
\end{equation*}
It follows from the lower bound on $\mu_G(D_0)$ that $K_0 = O_K(1)$ and $c_0=\Omega_K(1)$.  We shall arrange the sets so that they have the following properties:
\begin{enumerate}
\item \label{pty.81} $D_i$ is a symmetric neighbourhood of the identity for all $0 \leq i \leq J$;
\item \label{pty.86} $K_{i} = O_{K,r,\epsilon,i}(1)$ for all $0 \leq i \leq J$;
\item \label{pty.87} $c_i = O_{K,r,\epsilon,i}(1)$ for all $1 \leq i \leq J$;
\item \label{pty.82} $D_{i+1}^{12(k_{i+1}+1)} \subset D_i^4$ for all $0 \leq i \leq J-1$.
\end{enumerate}
It is immediate that $D_0$ satisfies the above ((\ref{pty.81}), (\ref{pty.86}) and (\ref{pty.87}) by design and (\ref{pty.82}) vacuously). Suppose that we have defined $D_i$ satisfying the above.  $D_i$ is a symmetric neighbourhood so $D_i^2 \subset D_i^{12}$, whence $\mu_G(D_i^2) \leq K_i\mu_G(D_i)$ and we may apply Proposition \ref{prop.intit} to get a non-empty set $D_i' \subset D_i$ such that
\begin{equation*}
\mu_G(\Sym_{1-1/12(k_{i+1}+1)}(D_i'D_i)) =\Omega_{K_i,k_{i+1}}(\mu_G(D_i)).
\end{equation*}
Put $D_{i+1}:=\Sym_{1-1/12(k_{i+1}+1)}(D_i'D_i)$ and note that we have property (\ref{pty.81}).  Property (\ref{pty.82}) follows from Lemma \ref{lem.sub-mult}, and that lemma and the lower bound on $\mu_G(D_{i+1})$ tell us that
\begin{equation*}
K_{i+1}=O_{K_i,k_{i+1}}(1)=O_{K,r,\epsilon,i+1}(1).
\end{equation*}
Finally
\begin{eqnarray*}
c_{i+1} = \frac{\mu_G(D_{i+1}^4)}{\mu_G(D_0^{12})} \geq \frac{\mu_G(D_{i+1})}{\mu_G(D_0^{12})}& \geq & \Omega_{K_i,k_{i+1}}\left(\frac{\mu_G(D_{i})}{\mu_G(D_0^{12})}\right)\\&=&\Omega_{K_i,k_{i+1}}\left(\frac{\mu_G(D_{i}^4)}{\mu_G(D_0^{12})}\right)=\Omega_{K_i,k_{i+1}}(c_{i})
\end{eqnarray*}
by the lower bound for $\mu_G(D_{i+1})$ and the fact that
\begin{equation*}
\mu_G(D_i) \geq \mu_G(D_i^{12})/K_i \geq \mu_G(D_i^4)/K_i.
\end{equation*}
The construction is complete.

We shall now define the sets $B_i$, $B_{i,i'}^+$ and $B_{i,i'}^-$ backwards in terms of the $D_i$s: at stage $j \leq J$ we shall have defined $B_i$ for all $j \leq i \leq J$, and $B_{i,i'}^+$ and $B_{i,i'}^-$ for all  $j \leq i < i' \leq J$ such that
\begin{enumerate}
\item \label{pty.1} $B_i$ is a symmetric neighbourhood of the identity;
\item \label{pty.2} $D_i^4 \subset B_i \subset D_i^{12}$ for all $j \leq i \leq J$;
\item \label{pty.5} $\mathcal{B}_{i,i'}:=(B_i,B_{i'})$ is an $\epsilon(c_{i'-1})$-closed, $c_{i'}$-thick $r(c_{i'-1})$-multiplicative pair for all $j \leq i < i' \leq J$.
\end{enumerate}
Notice that setting $B_J:=D_J^4$ certainly satisfies the requirements for $j=J$.  Properties (\ref{pty.1}) and (\ref{pty.2}) are trivially true, and (\ref{pty.5}) vacuously true.

Now, suppose that we are at stage $1 \leq j \leq J$ of the iteration. The following claim is pivotal.
\begin{claim*}
Suppose that $j \leq j_0\leq j_1 \leq J$ and $l_{j_0},\dots,l_{j_1}$ are integers such that $0 \leq l_i \leq k_i$ for all $j_0 \leq i \leq j_1 -1$ and $0 \leq l_{j_1} \leq k_{j_1}+1$. Then
\begin{equation*}
B_{j_0}^{l_{j_0}} \dots B_{j_1}^{l_{j_1}} \subset D_{j_0}^{12(k_{j_0}+1)}.
\end{equation*}
\end{claim*}
\begin{proof} 
We prove this by downward induction on $j_0$; if $j_0=j_1$ then the result is trivial since $B_j$ is a neighbourhood of the identity and $B_j \subset D_j^{12}$.  Now, suppose that we have proved the claim for some $j \leq j_0 \leq j_1$.  If $(l_i)_{i=j_0-1}^{j_1}$ satisfy the hypotheses of the claim, then by the inductive hypothesis
\begin{equation*}
B_{j_0}^{l_{j_0}} \dots B_{j_1}^{l_{j_1}} \subset D_{j_0}^{12(k_{j_0}+1)}.
\end{equation*}
But then
\begin{equation*}
B_{j_0-1}^{l_{j_0-1}} \dots B_{j_1}^{l_{j_1}} \subset D_{j_0-1}^{12l_{j_0-1}}D_{j_0}^{12(k_{j_0}+1)}.
\end{equation*}
However, $D_{j_0}^{12(k_{j_0}+1)} \subset D_{j_0-1}^4$ by construction of the $D_i$s and we are done since $12l_{j_0-1} + 4 \leq 12(k_{j_0-1}+1)$.  The claim is proved.
\end{proof}
It follows from the claim that if $j \leq i \leq J$ and $(l_{j'})_{j'=j}^i$ is a sequence of integers such that $0 \leq l_{j'} \leq k_{j'}$ for all $j \leq j' \leq i-1$ and $l_{i} \leq k_{i}+1$, then
\begin{equation}\label{eqn.ort}
B_{i}^{l_{i}}\dots B_{j}^{l_{j}} D_{j-1}^4B_{j}^{l_{j}} \dots B_{i}^{l_{i}}  \subset D_{j}^{12(k_{j}+1)}.D_{j-1}^4. D_{j}^{12(k_{j}+1)} \subset D_{j-1}^{12},
\end{equation}
by construction of the $D_i$s. 

We shall now define a sequence of naturals $(l_i)_{i=j}^J$ with
\begin{equation*}
r(c_{i-1}) \leq l_i \leq k_i - (r(c_{i-1})+1) \textrm{ for all }j \leq i \leq J
\end{equation*}
and sequences of sets $(B_{j-1,i}^+)_{i=j}^J$, $(B_{j-1,i})_{i=j}^J$ and $(B_{j-1,i}^-)_{i=j}^J$.  We define the sets in terms of the naturals:
\begin{equation*}
B_{j-1,i}^+:=B_{i+1}^{l_{i+1}+r(c_{i})+1}B_{i}^{l_{i}}\dots B_{j}^{l_{j}} D_{j-1}^4B_{j}^{l_{j}} \dots B_{i}^{l_{i}}B_{i+1}^{l_{i+1}+r(c_{i})+1},
\end{equation*}
\begin{equation*}
B_{j-1,i}:=B_{i}^{l_{i}}\dots B_{j}^{l_{j}} D_{j-1}^4B_{j}^{l_{j}} \dots B_{i}^{l_{i}},
\end{equation*}
and
\begin{equation*}
B_{j-1,i}^-:=B_{i+1}^{l_{i+1}-r(c_{i})}B_{i}^{l_{i}}\dots B_{j}^{l_{j}} D_{j-1}^4B_{j}^{l_{j}} \dots B_{i}^{l_{i}}B_{i+1}^{l_{i+1}-r(c_{i})}.
\end{equation*}
Suppose that we have picked $(l_{j'})_{j'=j}^i$.  In view of (\ref{eqn.ort}) and the definition of $B_{j-1,i}$ we have that
\begin{equation*}
B_{i+1}^{k_{i+1}+1}B_{j-1,i}B_{i+1}^{k_{i+1}+1} \subset D_{j-1}^{12},
\end{equation*}
whence
\begin{equation*}
\mu_G(B_{i+1}^{k_{i+1}+1}B_{j-1,i}B_{i+1}^{k_{i+1}+1} ) \leq \frac{\mu_G(D_{j-1}^{12})}{\mu_G(D_{j-1}^4)}\mu_G(B_{j-1,i}).
\end{equation*}
By construction of the $D_i$s we have that
\begin{equation*}
\mu_G(B_{i+1}^{k_{i+1}+1}B_{j-1,i}B_{i+1}^{k_{i+1}+1} ) \leq K_{j-1}\mu_G(B_{j-1,i}).
\end{equation*}
Thus we may apply the pigeonhole principle to pick $l_{i+1}$ with
\begin{equation*}
r(c_{i}) \leq l_{i+1} \leq k_{i+1} - (r(c_{i})+1)
\end{equation*}
such that
\begin{equation*}
\mu_G(B_{i+1}^{l_{i+1}+r(c_i)+1}B_{j-1,i}B_{i+1}^{l_{i+1}+r(c_i)+1} ) 
\end{equation*}
is at most
\begin{equation*}
 K_{j-1}^{-(2r(c_i)+1)/k_{i+1}}\mu_G(B_{i+1}^{l_{i+1}-r(c_i)}B_{j-1,i}B_{i+1}^{l_{i+1}-r(c_i)} ),
\end{equation*}
which when decoded tells us that
\begin{equation}\label{eqn.doff}
\mu_G(B_{j-1,i+1}^+) \leq (1+\epsilon(c_i))\mu_G(B_{j-1,i+1}^-).
\end{equation}
Having completed the above construction put $B_{j-1}:=B_{j-1,J}$, and note that by design $B_{j-1}$ is a symmetric neighbourhood of the identity containing $D_{j-1}^4$. Furthermore, by (\ref{eqn.ort}) we have
\begin{equation*}
 B_{j-1} \subset D_{j-1}^{12},
\end{equation*}
and to complete the induction it remains to check that $\mathcal{B}_{j-1,i}$ has the desired properties for all $j \leq i \leq J$.
\begin{enumerate}
\item (Symmetry) The sets $B_{j-1}$ and $B_i$ are symmetric neighbourhoods of the identity.
\item (Multiplicative pair) The sets $B_{j-1,i}^+$ and $B_{j-1,i}^-$ are symmetric neighbourhoods of the identity.  To check the necessary inclusions we note that in one direction it is immediate:
\begin{equation*}
 B_{i}^{r(\delta'_{i-1})}B_{j-1,i}^-B_{i}^{r(\delta'_{i-1})} =B_{j-1,i} \subset B_{j-1}.
\end{equation*}
In the other we have
\begin{equation*}
B_{i+1}^{l_{i+1}} \dots B_{J}^{l_{J}} \subset D_{i+1}^{12(k_{i+1}+1)} \subset D_{i}^4 \subset B_i,
\end{equation*}
and by symmetry
\begin{equation*}
B_{J}^{l_{J}}\dots B_{i+1}^{l_{i+1}} \subset D_{i+1}^{12(k_{i+1}+1)} \subset D_{i}^4 \subset B_i.
\end{equation*}
from the claim. Thus
\begin{equation*}
B_{j-1} \subset B_i^{l_i+1}B_{i-1}^{l_{i-1}} \dots B_j^{l_j}D_{j-1}^4 B_j^{l_j}\dots B_{i-1}^{l_{i-1}} B_i^{l_i+1},
\end{equation*}
and so it follows that
\begin{equation*}
B_{i}^{r(c_{i-1})}B_{j-1}B_{i}^{r(c_{i-1})} 
\end{equation*}
is contained in
\begin{equation*}
B_{i}^{l_{i}+r(c_{i-1})+1}B_{i-1}^{l_{i-1}}\dots B_{j}^{l_j}D_{j-1}^4B_{j}^{l_j}\dots B_{i-1}^{l_{i-1}} B_{i}^{l_{i}+r(c_{i-1})+1},
\end{equation*}
which is $B_{j-1,i}^+$ as required.  It follows that $\mathcal{B}_{j,i}$ is $r(c_{i-1})$-wide.
\item (Closure) The closure parameter of $\mathcal{B}_{j-1,i}$ is $\epsilon(c_{i-1})$ by (\ref{eqn.doff}).
\item (Thickness) The thickness of $\mathcal{B}_{j-1,i}$ is
\begin{eqnarray*}
\mu_G(B_i)/\mu_G(B_{j-1})& \geq& \mu_G(D_i^4)/\mu_G(D_{j-1}^{12}) \\ &\geq & \mu_G(D_i^4)/\mu_G(D_0^{12}) = c_i
\end{eqnarray*}
by the properties of the sets $D_{i'}$.
\end{enumerate}
The induction is closed and terminates when $j=0$, when we have a system of multiplicative pairs with the desired properties and it remains to note that $B_0 \subset D_0^{12}\subset A^4$, and 
\begin{equation*}
\mu_G(B_0) \geq \mu_G(D_0^4) = \Omega_K(\mu_G(A)).
\end{equation*}
The result is proved.
\end{proof}
Since the above argument is essentially a regularity construction (\emph{c.f.} \cite{tao::4}) it will come as little surprise that the bound are tower type in $J$.  Indeed, we shall have $r(x) = O(1)$ and $\epsilon(x)\leq Cx^{O(1)}$ in applications in which case it is easy to read out a lower bound on $c_i$ from the above: it is a tower of height $O(i)$ in $C+O(1)$.  Moreover, $\mu_G(A)/\mu_G(B_0)$ may be taken to be exponential in $K^{O(1)}$.

\section{From containment to correlation with multiplicative pairs}\label{sec.containment}

In this section we show how to pass from the situation of containing a multiplicative pair to correlation with a multiplicative pair.  This shift in perspective with Fre{\u\i}man-type results was introduced by Green and Tao in \cite{gretao::3} (although it is heavily foreshadowed in \cite{gow::4}) and has since been used fruitfully in many situations.  

We shall prove the following as a consequence of Proposition \ref{prop.frcon}.
\begin{proposition}\label{prop.wkfr}
Suppose that $G$ is a finite group, $A \subset G$ has $\mu_G(A^2) \leq K\mu_G(A)$ and $r \in \N$ and $\epsilon \in (0,1]$ are parameters.  Then there is a positive real $c=\Omega_{K,r,\epsilon}(1)$ and an $\epsilon$-closed, $c$-thick $r$-multiplicative pair $\mathcal{B}$ with ground set $B$ such that
\begin{equation*}
\|1_A \ast \mu_B\|_{L^\infty(\mu_G)} = \Omega_K(1) \textrm{ and } \mu_G(B)= \Omega_K(\mu_G(A)).
\end{equation*}
\end{proposition}
It should also be remarked that in recent work of Croot and Sisask \cite{crosis::} some combinatorial arguments have been developed for showing that if $A$ is dense then $1_A \ast 1_A$ is almost invariant over a large set -- repeated addition of this set can be used to give a multiplicative pair.  If their results extend to give large sets when $A$ merely has small doubling then it seems that it might be used to give another version of the above.

We require the following trivial projection fact for symmetry sets, which shows that if the threshold of a symmetry set of $A$ is very close to $1$ then $\mu_A$ is approximately invariant under convolution by probability measures supported on that set.
\begin{lemma}\label{lem.approxproj}
Suppose that $G$ is a finite group, $A \subset G$ and $\epsilon \in (0,1]$. Then
\begin{equation*}
\int{|1-\mu \ast 1_A|d\mu_A} \leq \epsilon
\end{equation*}
for all probability measures $\mu$ with $\supp \mu \subset \Sym_{1-\epsilon}(A)$.
\end{lemma}
\begin{proof}
Suppose that $\mu$ is a probaility measure with $\supp \mu \subset \Sym_{1-\epsilon}(A)$. Then
\begin{equation*}
\langle 1_A \ast 1_{A^{-1}},\mu\rangle_{L^2(\mu_G)} \geq (1-\epsilon) \mu_G(A)
\end{equation*}
by definition of the symmetry set.  However, it follows from Lemma \ref{lem.triv} that
\begin{equation*}
\langle 1_A , \mu \ast 1_{A} \rangle_{L^2(\mu_G)}\geq (1-\epsilon)\mu_G(A),
\end{equation*}
whence
\begin{equation*}
\langle 1_A,1-\mu \ast 1_A\rangle_{L^2(\mu_G)} \leq \epsilon \mu_G(A).
\end{equation*}
However, $0 \leq \mu\ast 1_A \leq 1$ and so the lemma is proved.
\end{proof}
The proof of Proposition \ref{prop.wkfr} will also use a couple of results from \cite{tao::6}.  The first is a sort of non-abelian Pl{\"u}nnecke theorem (\emph{c.f.} \cite{plu::}). 
\begin{lemma}[{\cite[Lemma 3.4]{tao::6}}] \label{lem.cover} Suppose that $G$ is a finite group and $A \subset G$ has $\mu_G(A^3) \leq K\mu_G(A)$. Then
\begin{equation*}
\mu_G(A^{\sigma_1}\dots A^{\sigma_n})\leq K^{O_n(1)}\mu_G(A)
\end{equation*}
for any signs $\sigma_1,\dots,\sigma_n \in \{-1,1\}$.
\end{lemma}
The proof is not difficult -- it is a covering argument of a type popularised by Ruzsa \cite{ruz::01} -- although it was a key insight of \cite{tao::6} that these arguments go through directly in the non-abelian setting.

We shall also require a result which lets us pass from small doubling to a large subset with small tripling.  Again this is from \cite{tao::6}, but it turns out that this is also a trivial corollary of Proposition \ref{prop.intit}.
\begin{corollary}\label{cor.doub2trip}
Suppose that $G$ is a finite group and $A \subset G$ has $\mu_G(A^2) \leq K\mu_G(A)$. Then there is a set $A'$ and element $x \in G$ such that $xA' \subset A$
\begin{equation*}
\mu_G(A') = \Omega_K(\mu_G(A)) \textrm{ and } \mu_G(A'^3) =O_K(\mu_G(A')).
\end{equation*}
\end{corollary}
\begin{proof}
Apply Proposition \ref{prop.intit} to get a non-empty set $A'' \subset A$ such that
\begin{equation*}
\mu_G(\Sym_{1-1/6}(A''A)) \geq \exp(K^{O(1)})\mu_G(A).
\end{equation*}
Put $A_0 :=\Sym_{1-1/6}(A''A)$, and then note by Lemma \ref{lem.approxproj} that
\begin{equation*}
\int{|1-\mu_{A_0} \ast 1_A|d\mu_A} \leq 1/3,
\end{equation*}
so that there is some $x \in G$ such that
\begin{equation*}
\mu_G( A_0 \cap (Ax^{-1}))=\mu_G(A_0 \cap xA^{-1})\geq 2\mu_G(A_0)/3=\Omega_K(\mu_G(A)) 
\end{equation*}
since $A_0$ is symmetric.  Put $A_1:=A_0 \cap (Ax^{-1})$ so that $\mu_G(A_1) = \Omega_K(\mu_G(A))$.  Furthermore, $A_1^3 \subset \Sym_{1/2}(A''A)$ by Lemma \ref{lem.sub-mult}, and hence by Lemma \ref{lem.symsize} we have that
\begin{equation*}
\mu_G(A_1^3) \leq 2\mu_G(A''A) \leq 2K\mu_G(A).
\end{equation*}
The result follows on putting $A'=x^{-1}A_1x$.
\end{proof}
Of course the bounds in this corollary are immediately seen to be exponentially dependent on $O(K^{O(1)})$. 
\begin{proof}[Proof of Proposition \ref{prop.wkfr}]
First we apply Corollary \ref{cor.doub2trip} to get a symmetric set $A'$ such that $x'A' \subset A$, $\mu_G(A') = \Omega_K(\mu_G(A))$ and $\mu_G(A'^3)=O_K(\mu_G(A'))$. Apply Proposition \ref{prop.intit} to get a set $A'' \subset A'$ such that
\begin{equation*}
\mu_G(\Sym_{1-2^{-4}}(A''A')) =\Omega_K(\mu_G(A)).
\end{equation*}
By Lemma \ref{lem.sub-mult} we have that
\begin{equation*}
\mu_G(\Sym_{1-2^{-4}}(A''A')^2) \leq \mu_G(\Sym_{1-1/8}(A''A')) \leq 2\mu_G(A''.A') \leq 2\mu_G(A).
\end{equation*}
Putting $A''':=\Sym_{1-2^{-4}}(A''A')$ we have that
\begin{equation*}
\mu_G(A''') = \Omega_K(\mu_G(A)) \textrm{ and } \mu_G(A'''^2) = O_K(\mu_G(A''')).
\end{equation*}
We apply Proposition \ref{prop.frcon} to get a positive real $c=\Omega_{K,r,\epsilon}(1)$ and an $\epsilon$-closed $c$-thick $r$-multiplicative pair $\mathcal{B}$ with ground set $B$ with
\begin{equation*}
B \subset A'''^4 \textrm{ and } \mu_G(B) = \Omega_K(\mu_G(A)).
\end{equation*}
On the other hand by Lemma \ref{lem.approxproj}
\begin{equation*}
\int{|1-\mu \ast 1_{A''A'}|d\mu_{A''A'}} \leq 1/2 
\end{equation*}
whenever $\supp \mu \subset \Sym_{1-1/2}(A''A')$.  Since $B \subset A'''^4 \subset \Sym_{1-1/2}(A''A')$ by Lemma \ref{lem.sub-mult} we see that we may put $\mu=\mu_{B}$.  It follows that
\begin{equation*}
\|\mu_B \ast 1_{A''A'}\|_{L^\infty(\mu_G)} \geq 1/2.
\end{equation*}
However, it is easy to see that
\begin{equation*}
1_{A''A'} \leq \mu_{A'} \ast 1_{(A'^{-1}A''A')^{-1}}
\end{equation*}
whence
\begin{equation*}
\|\mu_B \ast \mu_{A'} \ast 1_{(A'^{-1}A''A')^{-1}}\|_{L^\infty(\mu_G)} \geq 1/2
\end{equation*}
by non-negativity. We apply Young's inequality to see that
\begin{equation*}
\mu_G((A'^{-1}A''A')^{-1}).\|\mu_B \ast \mu_{A'}\|_{L^\infty(\mu_G)} \geq 1/2.
\end{equation*}
On the other hand $A'' \subset A'$ and so by Lemma \ref{lem.cover} and the fact that $A'$ has small tripling we get that
\begin{equation*}
\mu_G((A'^{-1}A''A')^{-1})\leq \mu_G(A'^{-2}A') =O_K(\mu_G(A')),
\end{equation*}
whence
\begin{equation*}
\|\mu_{B} \ast 1_{A'}\|_{L^\infty(\mu_G)} = \Omega_K(1).
\end{equation*}
It remains to note that $1_{x'^{-1}A} \geq 1_{A'}$ so that
\begin{equation*}
\|\mu_{B} \ast 1_{A}\|_{L^\infty(\mu_G)}=\|\mu_{B} \ast 1_{x'^{-1}A}\|_{L^\infty(\mu_G)}\geq \|\mu_{B} \ast 1_{A'}\|_{L^\infty(\mu_G)}=\Omega_K(1)
\end{equation*}
by non-negativity and the definition of convolution.  The result is proved.
\end{proof}
Regarding the bounds, $c$ is quadruply exponential in $O(\epsilon^{-1}rK^{O(1)})$ and the correlation bounds are both exponential in $K^{O(1)}$.

\section{Analysis on multiplicative pairs}\label{sec.locpsa}

There is a very general class of problems in combinatorics which involve counting small structures in large structures.  The prototypical example is three-term arithmetic progressions in abelian groups.  Suppose that $G$ is an abelian group and $A \subset G$. A three-term arithmetic progression in $A$ is a triple $x-y,x,x+y \in A$ and there is a natural way to count them:
\begin{equation*}
T(A):=\int{1_A(x-y)1_A(y)1_A(x+y)d\mu_G(x)d\mu_G(y)}.
\end{equation*}
Finding good lower bounds on $T(A)$ in terms of the density of $A$ is essentially the same as finding good bounds in Roth's theorem \cite{rot::0} which has received the attention of numerous authors.

If one now has a $1$-additive (multiplicative) pair $\mathcal{B}=(B,B')$ and set $A \subset B$, the question becomes one of how to meaningfully count progressions in $A$ relative to $B$.  One way to do it is to think of $B'$ as being the set we're `allowed to add' to $B$, and thus count
\begin{equation*}
T_\mathcal{B}(A):=\int{1_A(x-y)1_A(y)1_A(x+y)d\mu_B(x)d\mu_{B'}(y)}.
\end{equation*}
Of course with more complicated structures than just three-term progressions, involving more variables and terms we would need to assume that $\mathcal{B}$ was an $r$-additive (multiplicative) pair for some larger natural $r$, but the basic idea is the same.

The advantage of this definition is that many of the properties enjoyed by $A$ on a genuine group are approximately true on a multiplicative pair. For example, when $A$ is roughly the whole of $B$, $T_\mathcal{B}(A)$ is close to $1$; when $A$ is quasi-random in a certain rather nice sense, $T_{\mathcal{B}}(A)$ is close to $\mu_B(A)^3$; and when $\mu_B(A)$ has density bigger than $2/3+\eta$ (for some $\eta \rightarrow 0$ as $\epsilon \rightarrow 0$), $A$ contains a $3$-term progression by the pigeonhole principle.

It is the purpose of this section to extend the straightforward physical space manipulations that work so well for groups to the setting of multiplicative pairs.  The proofs proceed largely as expected and may be omitted by the experts.

We begin with an approximate substitute for the unimodular Haar measure $\mu_G$.  As was hinted at above, the measure $\mu_B$ is our candidate and the desired property is encoded in the next lemma.\begin{lemma}[Approximate Haar measure]\label{lem.approxhaar}  Suppose that $G$ is a finite group and $\mathcal{B}=(B,B')$ is an $\epsilon$-closed $r$-multiplicative pair.  Then
\begin{equation*}
\|\mu \ast \mu_B - \mu\| = \|\mu_B \ast \mu - \mu_B\| \leq \epsilon
\end{equation*}
for all probability measures $\mu$ with $\supp \mu \subset B'^r$.
\end{lemma}
\begin{proof}
The equality is trivial: $\widetilde{\mu \ast \nu} = \widetilde{\nu}\ast \widetilde{\mu}$ for all measure $\mu,\nu$ on $G$, $\widetilde{\mu_{B}}=\mu_B$ since $B$ is symmetric and $\supp \mu \subset B'$ iff $\supp \widetilde{\mu} \subset B'$ since $B'$ is symmetric.

Now, suppose that $\mu$ is a probability measure with $\supp \mu \subset B'$.  Then
\begin{equation*}
\|\mu_B \ast \mu - \mu_B\| \leq \int{\|\rho_{y^{-1}}(\mu_B)-\mu_B\|d\mu(y)}
\end{equation*}
by the triangle inequality.  However,
\begin{equation*}
\|\rho_{y^{-1}}(\mu_B)-\mu_B\| = \frac{\mu_G(By \triangle B)}{\mu_G(B)} \leq \frac{\mu_G(B^+ \setminus B^-)}{\mu_G(B)} \leq \epsilon
\end{equation*}
since $\mathcal{B}$ is an $\epsilon$-closed $r$-multiplicative pair.  The result follows.
\end{proof}
An immediate consequence of this is  a sort of continuity result on convolution with this approximate Haar measure.
 \begin{lemma}\label{lem.appcon}
 Suppose that $G$ is a finite group, $f \in L^\infty(\mu_G)$ and $\mathcal{B}=(B,B')$ is an $\epsilon$-closed $r$-multiplicative pair. Then
 \begin{equation*}
 \sup_{y \in B'^r}{|f \ast \mu_{B}(xy) - f\ast \mu_B(x)|} \leq \epsilon \|f\|_{L^\infty(\mu_G)}.
 \end{equation*}
 \end{lemma}
 \begin{proof}
First we recall that $f \ast \mu_{B}(xy) =f \ast \rho_y(\mu_B)(x)$ since $\rho_y$ commutes with convolution, so by Young's inequality
\begin{eqnarray*}
|f \ast \mu_{B}(xy) - f\ast \mu_B(x)| & = & |f \ast \rho_y(\mu_{B})(x) - f\ast \mu_B(x)|\\ & \leq & \|f\|_{L^\infty(\mu_G)}\|\rho_y(\mu_B)-\mu_B\|.
\end{eqnarray*}
The lemma then follows from Lemma \ref{lem.approxhaar}.
 \end{proof}
The next argument is a short calculation typical of physical space manipulations with multiplicative pairs.
\begin{lemma}\label{lem.bogcalc}
Suppose that $G$ is a finite group, $\mathcal{B}=(B,B')$ is an $\epsilon$-closed $r$-multiplicative pair, $f  \in L^1(\mu_{B'})$ and $g \in L^\infty(\mu_{BB'^r})$. Then
\begin{equation*}
|\|(fd\mu_{B'}) \ast (g|_B)\|_{L^2(\mu_B)}^2 - \|(fd\mu_{B'}) \ast g\|_{L^2(\mu_B)}^2 | \leq 2\sqrt{\epsilon}\|f\|_{L^1(\mu_{B'})}^2\|g\|_{L^\infty(\mu_{BB'^r})}^2.
\end{equation*}
\end{lemma} 
\begin{proof}
First we note by Young's inequality that we have
\begin{eqnarray*}
\|(fd\mu_{B'}) \ast (g|_B)\|_{L^2(\mu_B)}^2 & =& \frac{1}{\mu_G(B')^2\mu_G(B)}\|f \ast (g1_B)\|_{L^2(\mu_G)}^2\\ & \leq &  \frac{1}{\mu_G(B')^2\mu_G(B)}\|f\|_{L^1(\mu_G)}^2\|g1_B\|_{L^2(\mu_G)}^2\\ & \leq & 
\|f\|_{L^1(\mu_{B'})}^2\|g\|_{L^\infty(\mu_{BB'^r})}^2.
\end{eqnarray*}
Similarly we have
\begin{equation*}
\|(fd\mu_{B'}) \ast g\|_{L^2(\mu_B)}^2 \leq
\|f\|_{L^1(\mu_{B'})}^2\|g\|_{L^\infty(\mu_{BB'^r})}^2,
\end{equation*}
whence
\begin{equation*}
|\|(fd\mu_{B'}) \ast (g|_B)\|_{L^2(\mu_B)}^2 - \|(fd\mu_{B'}) \ast g\|_{L^2(\mu_B)}^2 | 
\end{equation*}
is at most
\begin{equation*}
2\|(fd\mu_{B'}) \ast (g - g|_B)\|_{L^2(\mu_B)}\|f\|_{L^1(\mu_{B'})}\|g\|_{L^\infty(\mu_{BB'^r})}.
\end{equation*}
Of course by Young's inequality again
\begin{eqnarray*}
\|(fd\mu_{B'}) \ast (g - g|_B)\|_{L^2(\mu_B)}^2 & =& \frac{1}{\mu_G(B')^2\mu_G(B)}\|f \ast (g-g1_B)\|_{L^2(\mu_G)}^2\\ & \leq &  \frac{1}{\mu_G(B')^2\mu_G(B)}\|f\|_{L^1(\mu_G)}^2\|g-g1_B\|_{L^2(\mu_G)}^2\\ & \leq & 
\|f\|_{L^1(\mu_{B'})}^2\|g\|_{L^\infty(\mu_{BB'^r})}^2\mu_G(BB'^r\setminus B).
\end{eqnarray*}
The result follows on combining all this.
\end{proof}
 The final result of the section will be used in \S\ref{sec.dis} and while it is a calculation of the type presented here, its utility will probably not be clear without also reading that section.  The result shows how, in a certain situation, to pass from $\|f \ast v\|_{L^2(\mu_G)}$ being large to a properly relativised version being large. 
\begin{lemma}\label{lem.chopup}
Suppose that $G$ is a finite group, $B_0,B_1,B_2$ are symmetric subsets of $G$ such that $\mathcal{B}_{i,j}=(B_i,B_j)$ is a $c_j$-thick, $\epsilon_j$-closed, $4$-multiplicative pair for all $j>i$, $f\in L^1(\mu_{B_2})$, $h \in L^1(\mu_{x_1B_1})$ (not identically zero) and $g \in L^2(\mu_G)$ is an eigenvector of the convolution operator $ L_{hd\mu_{x_1B_1}}^* L_{hd\mu_{x_1B_1}}$ having non-zero eigenvalue $\lambda \|h\|_{L^\infty(\mu_{x_1B_1})}^2$ with
\begin{equation*}
\|fd\mu_{B_2} \ast g\|_{L^2(\mu_G)}^2 > \eta\|f\|_{L^\infty(\mu_{B_2})}^2\|g\|_{L^2(\mu_G)}^2.
\end{equation*}
Then, if $\epsilon_1 \leq 1$ and $\epsilon_2 \leq \eta c_1^2|\lambda|^2/16$, there is some $x' \in G$ such that
\begin{equation*}
\|fd\mu_{B_2} \ast (\rho_{x'}(g)|_{B_0})\|_{L^2(\mu_{B_0})}^2 > \eta \|f\|_{L^\infty(\mu_{B_2})}^2\|\rho_{x'}(g)\|_{L^2(\mu_{B_0})}^2/4,
\end{equation*}
and
\begin{equation*}
\|\rho_{x'}(g)\|_{L^\infty(\mu_{B_0})}\leq 4\eta^{-1/2}|\lambda|^{-1}c_1^{-1}\|\rho_{x'}(g)\|_{L^2(\mu_{B_0})}.
\end{equation*}
\end{lemma}
\begin{proof}
First, we note that
\begin{eqnarray*}
\|fd\mu_{B_2} \ast g\|_{L^2(\mu_G)}^2&=&\langle \widetilde{fd\mu_{B_2}} \ast fd\mu_{B_2} \ast g,g\rangle_{L^2(\mu_G)}\\ & =&  \int{\langle \widetilde{fd\mu_{B_2}} \ast fd\mu_{B_2} \ast g, gd\mu_{B_0x}\rangle_{L^2(\mu_G)}d\mu_G(x)}
\end{eqnarray*}
by linearity, and similarly
\begin{equation*}
\|g\|_{L^2(\mu_G)}^2=\int{\langle g,gd\mu_{B_1^2B_2^2B_0x}\rangle_{L^2(\mu_G)}d\mu_G(x)},
\end{equation*}
whence, by averaging, there is an $x' \in G$ such that
\begin{equation*}
\langle \widetilde{fd\mu_{B_2}} \ast fd\mu_{B_2} \ast g, gd\mu_{B_0x'}\rangle_{L^2(\mu_G)}> \eta \|f\|_{L^\infty(\mu_{B_2})}^2\|g\|_{L^2(\mu_{B_1^2B_2^2B_0x'})}^2.
\end{equation*}
Now, $\rho_{x'}(\mu_{Ax'}) = \mu_A$ for all sets $A$, and since $\rho_{x'}$ is unitary we conclude that
\begin{equation*}
\langle \rho_{x'}(\widetilde{fd\mu_{B_2}} \ast fd\mu_{B_2} \ast g), \rho_{x'}(g)d\mu_{B_0}\rangle_{L^2(\mu_G)}
\end{equation*}
is bigger than
\begin{equation*}
\eta  \|f\|_{L^\infty(\mu_{B_2})}^2\|\rho_{x'}(g)\|_{L^2(\mu_{B_1^2B_2^2B_0})}^2.
\end{equation*}
On the other hand right translation commutes with left convolution hence
\begin{equation*}
\rho_{x'}(\widetilde{fd\mu_{B_2}} \ast fd\mu_{B_2} \ast g)=\widetilde{fd\mu_{B_2}}\ast fd\mu_{B_2} \ast \rho_{x'}(g),
\end{equation*}
and
\begin{eqnarray*}
\widetilde{hd\mu_{x_1B_1}}\ast hd\mu_{x_1B_1}  \ast \rho_{x'}(g) &=& \rho_{x'}(\widetilde{hd\mu_{x_1B_1}}\ast hd\mu_{x_1B_1}  \ast g)\\& =& \lambda \|h\|_{L^\infty(\mu_{x_1B_1})}^2\rho_{x'}(g).
\end{eqnarray*}
Thus we may assume, by translating $g$ if necessary, that $x'=1_G$.

The situation now is that
\begin{equation}\label{eqn.sitn}
\langle \widetilde{fd\mu_{B_2}} \ast fd\mu_{B_2} \ast g, g\rangle_{L^2(\mu_{B_0})}> \eta  \|f\|_{L^\infty(\mu_{B_2})}^2 \|g\|_{L^2(\mu_{B_1^2B_2^2B_0})}^2.
\end{equation}
We examine the difference $D_1$, defined to be
\begin{equation*}
|\langle \widetilde{fd\mu_{B_2}} \ast fd\mu_{B_2} \ast g, g\rangle_{L^2(\mu_{B_0})}-\langle \widetilde{fd\mu_{B_2}} \ast fd\mu_{B_2} \ast (g|_{B_0}), g\rangle_{L^2(\mu_{B_0})}|
\end{equation*}
in the first instance.  We begin by noting that $\supp \widetilde{fd\mu_{B_2}} \ast fd\mu_{B_2} \ast (g|_{B_0}) \subset B_2^2B_0 \subset B_{0,2}^+$, and
\begin{equation*}
 \widetilde{fd\mu_{B_2}} \ast fd\mu_{B_2} \ast (g|_{B_0})(x) =  \widetilde{fd\mu_{B_2}} \ast fd\mu_{B_2} \ast g(x)
\end{equation*}
for all $x \in B_{0,2}^-$.  It follows that
\begin{eqnarray*}
D_1& \leq &\|g\|_{L^\infty(\mu_{B_2^2B_0})}^2\|f\|_{L^1(\mu_{B_2})}^2\int{1_{B_{0,2}^+\setminus B_{0,2}^-}d\mu_{B_0}}\\ & \leq &  \epsilon_2\|f\|_{L^\infty(\mu_{B_2})}^2\|g\|_{L^\infty(\mu_{B_2^2B_0})}^2.
\end{eqnarray*}
Next we examine the difference $D_2$, defined to be
\begin{equation*}
|\frac{1}{\mu_G(B_0)}.\|fd\mu_{B_2} \ast (g|_{B_0})\|_{L^2(\mu_G)}^2 - \|fd\mu_{B_2} \ast (g|_{B_0})\|_{L^2(\mu_{B_0})}^2|
\end{equation*}
The integrands are the same inside $B_0$, so we have that
\begin{eqnarray*}
D_2 & \leq & \frac{1}{\mu_G(B_0)}\int{|fd\mu_{B_2} \ast (g|_{B_0})|^21_{B_2B_0\setminus B_0}d\mu_G}\\ & \leq & \epsilon_2\|f\|_{L^\infty(\mu_{B_2})}^2\|g\|_{L^\infty(B_0)}^2 
\end{eqnarray*}
by Young's inequality.

By the triangle inequality and the estimates for $D_1$ and $D_2$ applied to (\ref{eqn.sitn}) we get that
\begin{eqnarray}
\nonumber  \|fd\mu_{B_2} \ast (g|_{B_0})\|_{L^2(\mu_{B_0})}^2& >& \eta  \|f\|_{L^\infty(\mu_{B_2})}^2\|g\|_{L^2(\mu_{B_1^2B_2^2B_0})}^2\\ \label{eqn.fg}& & - 2\epsilon_2\|f\|_{L^\infty(\mu_{B_2})}^2\|g\|_{L^\infty(B_2^2B_0)}^2.
\end{eqnarray}

Now we need to bound $\|g\|_{L^\infty(B_2^2B_0)}$.  Recall that
\begin{equation*}
\widetilde{hd\mu_{x_1B_1}}\ast hd\mu_{x_1B_1}  \ast g = \lambda \|h\|_{L^\infty(\mu_{x_1B_1})}^2g,
\end{equation*}
thus
\begin{eqnarray*}
|g(x)||\lambda| \|h\|_{L^\infty(\mu_{x_1B_1})}^2 & \leq & |\widetilde{hd\mu_{x_1B_1}}\ast hd\mu_{x_1B_1}  \ast g(x)|\\ & \leq & \int{\|h\|_{L^\infty(\mu_{x_1B_1})}^2|g(y^{-1}x)|d\widetilde{\mu_{x_1B_1}}\ast \mu_{x_1B_1}(y)}\\ & =& \|h\|_{L^\infty(\mu_{x_1B_1})}^2\int{|g(y^{-1}x)|d\mu_{B_1} \ast \mu_{B_1}(y)}.
\end{eqnarray*}
Since $h$ is not identically zero, it follows that if $x \in B_2^2B_0$, then
\begin{eqnarray*}
|g(x)||\lambda| & \leq &\frac{\mu_G(B_1^2B_2^2B_0)}{\mu_G(B_1)}\int{|g(z)|d\mu_{B_1^2B_2^2B_0}(z)}\\ & \leq &  2\frac{\mu_G(B_0)}{\mu_G(B_1)}\|g\|_{L^2(\mu_{B_1^2B_2^2B_0})}
\end{eqnarray*}
by the Cauchy-Schwarz inequality and the fact that $B_2^2 \subset B_1^2$ and $(B_0,B_1)$ is $1$-closed $4$-multiplicative pair.  We have shown that
\begin{equation}\label{eqn.ods}
\|g\|_{L^\infty(\mu_{B_2^2B_0})} \leq 2|\lambda|^{-1}c_1^{-1}\|g\|_{L^2(\mu_{B_1^2B_2^2B_0})}.
\end{equation}
Inserting this and the upper bound on $\epsilon_2$ into (\ref{eqn.fg}) we get that
\begin{equation}\label{eqn.kr}
 \|fd\mu_{B_2} \ast (g|_{B_0})\|_{L^2(\mu_{B_0})}^2 >\eta \|f\|_{L^\infty(\mu_{B_2})}^2\|g\|_{L^2(\mu_{B_1^2B_2^2B_0})}^2/2.
\end{equation}
The first conclusion then follows since $B_0 \subset B_1^2B_2^2B_0$ and
\begin{equation*}
\mu_G(B_1^2B_2^2B_0) \leq 2\mu_G(B_0)
\end{equation*}
since $(B_0,B_1)$ is a $c_1$-thick $1$-closed $4$-multiplictive pair.

For the second conclusion note on combining (\ref{eqn.ods}) with (\ref{eqn.kr}), that
\begin{equation*}
\|g\|_{L^\infty(\mu_{B_2^2B_0})} \leq 4\eta^{-1/2}|\lambda|^{-1}c_1^{-1} \|f\|_{L^\infty(\mu_{B_2})}^{-1}\|(fd\mu_{B_2}) \ast (g|_{B_0})\|_{L^2(\mu_{B_0})},
\end{equation*}
as required.
\end{proof}
 
\section{Normalising a multiplicative pair}\label{sec.norm}
 
Given a group $G$ and a subgroup $H$ of bounded index, it is relatively easy to find a subgroup $K\lhd G$ such that $K \subset H$ and $K$ is also of bounded index.  The idea is to let $K$ be the kernel of the natural embedding of $G$ into the symmetry group on the cosets of $H$:
\begin{equation*}
G \mapsto \Sym(G/H); x \mapsto yH \mapsto xyH.
\end{equation*}
Normal subgroups are often much easier to work with than subgroups and we shall at times want an approximate analogue for multiplicative pairs and it is the purpose of this section to prove such a result.

Our argument is essentially the natural extension of the non-approximate situation via a covering argument.  It works in reasonable generality so we include a version not specific to multiplicative pairs for the benefit of the reader.
\begin{lemma}
Suppose that $A,B,X \subset G$, $\mu_G(BB^{-1}X^{-1}XBB^{-1})\leq K \mu_G(BB^{-1})$, and $X$ has size $M$.  Then there is a symmetric neighbourhood of the identity, $S$, with $\mu_G(S) \geq K^{1-M}\mu_G(BB^{-1})$ such that $uSu^{-1} \subset ABB^{-1}A^{-1}$ for all $u \in AX$. 
\end{lemma}
\begin{proof}
We let $S:=\bigcap_{x \in X}{xBB^{-1}x^{-1}}$, so that $S$ is certainly a symmetric neighbourhood of the identity.  Moreover, if $u \in AX$ then $u \in Ax$ for some $x \in X$, whence
\begin{equation*}
uSu^{-1} \subset AxSx^{-1}A^{-1} \subset Axx^{-1}BB^{-1}xx^{-1}A^{-1} = ABB^{-1}A^{-1}.
\end{equation*}
It remains to show that $S$ is large: enumerate $X$ as $(x_i)_i$ and define sets $(D_i)_i$ inductively such that
\begin{equation*}
D_i \subset D_{i-1}, \mu_G(D_i) = \Omega_{K}(\mu_G(D_{i-1})) \textrm{ and } D_{i}D_{i}^{-1}\subset x_{i}BB^{-1}x_{i}^{-1}.
\end{equation*}
Set $D_1:=B$ and note that it trivially satisfies the above.  Now, suppose that we have defined $D_i$.  Note that
\begin{equation*}
\supp{1_{D_i} \ast 1_{x_iBB^{-1}x_i^{-1}}} = D_ix_{i+1}BB^{-1}x_{i+1}^{-1} \subset x_1BB^{-1}X^{-1}XBB^{-1}x_{i+1}^{-1},
\end{equation*}
since $D_i \subset D_1 \subset x_1BB^{-1}x_1^{-1}$.  It follows by averaging that there is some $x$ such that
\begin{equation*}
\mu_G(D_i \cap xx_iBB^{-1}x_i^{-1})=1_{D_i} \ast 1_{x_iBB^{-1}x_i^{-1}}(x) \geq \mu_G(D_i)/K.
\end{equation*}
Let $D_{i+1}:=D_i \cap xx_iBB^{-1}x_i^{-1}$. The sequence $D_i$ clearly has the desired properties and $D_{M}D_{M}^{-1} \subset S$ from which the result follows.
\end{proof}
We shall need the following immediate corollary.
\begin{corollary}\label{cor.apnorm}
Suppose that $G$ is a finite group and $B_0,B_1,B_2$ are such that $\mathcal{B}_{0,1}=(B_0,B_1)$ is a $c_1$-thick $1$-multiplicative pair, and $\mathcal{B}_{1,2}=(B_1,B_2)$ is a $c_2$-thick, $1$-closed $1$-multiplicative pair.  Then there is a symmetric neighbourhood of the identity $B_3$ such that $\mu_G(B_3)=\Omega_{c_1,c_2}(\mu_G(B_2))$ and $xB_3x^{-1} \subset B_2^6$ for all $x \in B_1$.
\end{corollary}
It is easy to see that the bounds on $c_3^{-1}$ inherited from the earlier proof are exponential in $c_1^{-O(1)}c_2^{-O(1)}$.

We shall use the above lemma to facilitate the replacement of expressions like $g \ast \mu_B$ with their conjugates $\mu_B \ast g$, and in particular it will be done through the following lemma.
\begin{lemma}\label{lem.apnormcalc}
Suppose that $G$ is a finite group, $B_0,B_1,B_2,B_3$ are symmetric neighbourhoods of the identity such that the pair $\mathcal{B}_{1,2}:=(B_1,B_2)$ is an $\epsilon$-closed $1$-multiplicative pair, and $xB_3x^{-1} \subset B_2$ for all $x \in B_0$, $f \in L^\infty(\mu_{B_1})$, $g \in L^2(\mu_{B_0})$ and $\mu$ is a probability measure with $\supp \mu \subset B_3$. Then if
\begin{equation*}
\epsilon \leq \|(fd\mu_{B_1}) \ast (g \ast \mu)\|_{L^2(\mu_{B_0})}/2\sqrt{3}\|f\|_{L^\infty(\mu_{B_1})}\|g\|_{L^\infty(\mu_{B_0})}
\end{equation*}
we have
\begin{equation*}
\|(fd\mu_{B_1}) \ast (g \ast \mu)\|_{L^2(\mu_{B_0})}^2/2 \leq \sup_{u \in {B_0}}{\sup_{y \in B_{1,2}^-u}{|\rho_{u^{-1}}(f) \ast \mu(y)|^2}}\|g\|_{L^2(\mu_{B_0})}^2.
\end{equation*}
\end{lemma}
\begin{proof}
Begin by putting
\begin{equation*}
h_u(y):=\int{f(z)d\mu(y^{-1}z^{-1}u)} = \rho_{y^{-1}}(f) \ast \mu(u),
\end{equation*}
and note that
\begin{equation*}
\int{h_u(y)g(y)d\mu_G(y)}=f \ast (g \ast \mu)(u),
\end{equation*}
so that writing
\begin{equation*}
S:=\|(fd\mu_{B_1}) \ast (g \ast \mu)\|_{L^2(\mu_{B_0})}^2
\end{equation*}
we have
\begin{equation*}
S=\frac{1}{\mu_G(B_1)^2}\int{|\int{h_u(y)g(y)d\mu_G(y)}|^2d\mu_{B_0}(u)}.
\end{equation*}
Now, $y \in \supp h_u$ implies $u \in B_1yB_3$, and hence $\supp h_u \subset B_1uB_3\subset B_1B_2u$.  We want to estimate a quantity which is quite cumbersome to write down and so we shall have to introduce a lot of auxiliary notation.  Begin by writing
\begin{equation*}
I_1(u):=\int{h_u(y)g(y)d\mu_G(y)} ,
\end{equation*}
\begin{equation*}
I_2(u):=\int{h_u(y)g(y)1_{B_{1,2}^-u}(y)d\mu_G(y)},
\end{equation*}
and
\begin{equation*}
D_1:=\int{|I_1|^2d\mu_{B_0}} \textrm{ and }D_2:=\int{|I_2|^2d\mu_{B_0}} .
\end{equation*}
We want to estimate $D:=|D_1 -D_2|$ from above and below.  First, unpacking the notation one sees that
\begin{equation*}
D_1=\mu_G(B_1)^2\|(fd\mu_{B_1}) \ast (g \ast \mu)\|_{L^2(\mu_{B_0})}^2
\end{equation*}
from our earlier calculations.  To estimate $D$ from below, we shall estimate $D_2$ from above.  Write
\begin{equation*}
Q:=\sup_{u \in B_0}\sup_{y \in B_{1,2}^-u}{|\rho_{y^{-1}}(f) \ast \mu(u)|},
\end{equation*}
(our eventual quantity of interest) and note that
\begin{eqnarray*}
D_2 & \leq &Q^2\int{\left(\int{|g(y)|1_{B_{1,2}^-u}(y)d\mu_G(y)}\right)^2d\mu_{B_0}(u)}\\ 
& = & Q^2\int{|g(y)||g(y')|\mu_G(B_{1,2}^-y \cap B_{1,2}^-y')d\mu_G(y)d\mu_G(y')}\\& \leq & \frac{1}{2}Q^2\int{(|g(y)|^2+|g(y')|^2)\mu_G(B_{1,2}^-y \cap B_{1,2}^-y')d\mu_G(y)d\mu_G(y')}\\
& = & \mu_G(B_{1,2}^-)^2Q^2\|g\|_{L^2(\mu_{B_0})}^2 \leq \mu_G(B_1)^2Q^2\|g\|_{L^2(\mu_{B_0})}^2.
\end{eqnarray*}
Now we turn to bounding $D$ from above.  We have
\begin{equation}\label{eqn.eaty}
D =| \int{|I_1|^2 - |I_2|^2d\mu_{B_0}}| \leq \left(\int{|I_1 - I_2|^2d\mu_{B_0}}\int{(|I_1|+|I_2|)^2d\mu_{B_0}}\right)^{1/2}
\end{equation}
by the Cauchy-Schwarz inequality and the triangle inequality.  Now, since $\supp h_u \subset B_1B_2u$ we have that
\begin{equation*}
I_1(u)=\int{h_u(y)g(y)1_{B_{1,2}^+u}(y)d\mu_G(y)},
\end{equation*}
whence
\begin{equation*}
|I_1(u) - I_2(u)| \leq \|h_u\|_{L^\infty(\mu_G)}\|g\|_{L^\infty(\mu_G)}\epsilon \mu_{G}(B_1)
\end{equation*}
since $(B_1,B_2)$ is an $\epsilon$-closed $1$-multiplicative pair. By Young's inequality and the support of$f$ and $g$ we have that
\begin{equation*}
 \|h_u\|_{L^\infty(\mu_G)}\leq \|f\|_{L^\infty(\mu_{B_1})} \textrm{ and } \|g\|_{L^\infty(\mu_G)} \leq \|g\|_{L^\infty(\mu_{B_0})}.
\end{equation*}
Inserting this into (\ref{eqn.eaty}) we get that
\begin{eqnarray*}
D& \leq& \epsilon\mu_{G}(B_1)\|f\|_{L^\infty(\mu_{B_1})}\|g\|_{L^\infty(\mu_{B_0})}\left(2\int{|I_1|^2+|I_2|^2d\mu_{B_0}}\right)^{1/2}\\ & = & \epsilon \mu_G(B_1)\|f\|_{L^\infty(\mu_{B_1})}\|g\|_{L^\infty(\mu_{B_0})}(2\mu_G(B_1)^2(S + Q^2\|g\|_{L^2(\mu_{B_0})}^2))^{1/2}.
\end{eqnarray*}
On the other hand
\begin{equation*}
D \geq \mu_G(B_1)^2(S - Q^2\|g\|_{L^2(\mu_{B_0})}^2),
\end{equation*}
whence
\begin{equation*}
S - Q^2\|g\|_{L^2(\mu_{B_0})}^2 \leq \epsilon\|f\|_{L^\infty(\mu_{B_1})}\|g\|_{L^\infty(\mu_{B_0})} \sqrt{2(S+Q^2\|g\|_{L^2(\mu_{B_0})}^2)}.
\end{equation*}
Now, either we are done, or
\begin{equation*}
2Q^2\|g\|_{L^2(\mu_{B_0})}^2 \leq S,
\end{equation*}
whence
\begin{equation*}
S - Q^2\|g\|_{L^2(\mu_{B_0})}^2 \leq \epsilon\|f\|_{L^\infty(\mu_{B_1})}\|g\|_{L^\infty(\mu_{B_0})} \sqrt{3S},
\end{equation*}
and the result is proved in light of the upper bound on $\epsilon$.
\end{proof}

\section{Fourier analysis on multiplicative pairs}\label{sec.locft}

In this section we develop Fourier analysis on multiplicative pairs.  In particular, we shall try to extend as many of the results from \S\ref{sec.ft} to this approximate setting as possible.

We have previously defined a Haar measure and with this we can formulate the analogue of the transform $f \mapsto L_f$.  Suppose that $\mathcal{B}=(B,B')$ is a multiplicative pair and $f \in L^1(\mu_{B'})$.  We define the operator $L_{\mathcal{B},f}$ as follows:
\begin{equation*}
L_{\mathcal{B},f}:L^2(\mu_B) \rightarrow L^2(\mu_B); v \mapsto ((fd\mu_{B'}) \ast v)|_B.
\end{equation*}
The map is \emph{not} an algebra homomorphism, although it functions approximately as such, but it does  preserve adjoints.
\begin{lemma}
Suppose that $G$ is a finite group, $\mathcal{B}=(B,B')$ is a multiplicative pair and $f \in L^1(\mu_{B'})$. Then $L_{\mathcal{B},f}^*=L_{\mathcal{B},\tilde{f}}$.
\end{lemma}
\begin{proof}
This is simply a calculation.  Suppose that $v,w \in L^2(\mu_B)$ and note that
\begin{equation*}
\langle L_{\mathcal{B},f}v,w\rangle_{L^2(\mu_B)} = \mu_G(B)^{-1}\langle (fd\mu_{B'}) \ast v,w \rangle_{L^2(\mu_G)}.
\end{equation*}
We apply Lemma \ref{lem.triv} to see that
\begin{equation*}
\langle L_{\mathcal{B},f}v,w\rangle_{L^2(\mu_B)} = \mu_G(B)^{-1}\langle v,\widetilde{(fd\mu_{B'})} \ast w \rangle_{L^2(\mu_G)}.
\end{equation*}
However, $\widetilde{fd\mu_{B'}} = \tilde{f}d\mu_{B'}$, whence
\begin{equation*}
\langle L_{\mathcal{B},f}v,w\rangle_{L^2(\mu_B)} = \langle v,L_{\mathcal{B},\tilde{f}} w \rangle_{L^2(\mu_B)}
\end{equation*}
since $\supp v \subset B$.  Since $v$ and $w$ were arbitrary we conclude that $L_{\mathcal{B},f}^*=L_{\mathcal{B},\tilde{f}}$.
\end{proof}
We do not have a direct analogue of Parseval's theorem, however we do have an analogue of Bessel's inequality (Parseval's theorem polarised and with an inequality instead of equality) which is all we shall need for applications.

The observation that a Bessel-type inequality is often sufficient was made fact by Green and Tao in \cite{gretao::1} where they prove a Bessel inequality relative to Bohr sets in the abelian setting.
\begin{proposition}[Local Bessel inequality]\label{prop.locbes}
Suppose that $G$ is a finite group, $\mathcal{B}=(B,B')$ is a $c$-thick multiplicative pair and $f \in L^2(\mu_{B'})$.  Then
\begin{equation*}
\|L_{\mathcal{B},f}\|_{\End(L^2(\mu_B))}^2 \leq c^{-1}\|f\|_{L^2(\mu_{B'})}^2
\end{equation*}
\end{proposition}
\begin{proof}
The left hand side is just the trace of $L_{\mathcal{B},f}^*L_{\mathcal{B},f}$ which is basis invariance, thus for any orthonormal basis $e_1,\dots,e_n$ of $L^2(\mu_{B})$ we have
\begin{equation*}
\|L_{\mathcal{B},f}\|_{\End(L^2(\mu_B))}^2 = \sum_{i=1}^n{\langle L_{\mathcal{B},f}^*L_{\mathcal{B},f}e_i,e_i\rangle_{L^2(\mu_{B})}}.
\end{equation*}
Now, we choose $e_1,\dots,e_n$ judiciously as we did in the proof of Parseval's theorem.  For each $x \in B$, define the function
\begin{equation*}
e_x(y):=\begin{cases} |B|^{1/2} & \textrm{ if } y=x\\ 0 & \textrm{ otherwise.} \end{cases}
\end{equation*}
It is easy to see that $(e_x)_{x \in B}$ is an orthonormal basis for $L^2(\mu_{B})$ and, furthermore, we have
\begin{equation*}
L_{\mathcal{B},f}(e_x)(y)=f(yx^{-1})1_B(y)|B|^{1/2}|B'|^{-1} \textrm{ for all } y \in B.
\end{equation*}
It follows that
\begin{equation*}
\langle L_{\mathcal{B},f}^*L_{\mathcal{B},f}e_x,e_x\rangle_{L^2(\mu_{B})} = \|L_{\mathcal{B},f}e_x\|_{L^2(\mu_B)}^2 \leq \|f\|_{L^2(\mu_{B'})}^2|B'|^{-1}
\end{equation*}
whence, on summing, we get the result.
\end{proof}
A key property of the operators $L_f$ was that they commuted with right translation; this is only true approximately when set relative to multiplicative pairs setting.  We begin by setting some notation for right translation to help make some of our results more suggestive.

Suppose that $G$ is a finite group and $\mathcal{B}$ is an $r$-multiplicative pair with ground set $B$ and perturbation set $B'$. We write
\begin{equation*}
\rho_{\mathcal{B},y}:L^2(\mu_B) \rightarrow L^2(\mu_B);v \mapsto \rho_y(v)|_B,
\end{equation*}
for each $y \in B'^r$.  First we should remark that the maps $\rho_{\mathcal{B},y}$ are no longer unitary and recovering the situation in a useful way is a major part of our work in \S\ref{sec.anallarge}.  For now we have the following lemma which says that $\rho_{\mathcal{B},y}$ is approximately unitary.
\begin{lemma}[Approximate unitarity]\label{lem.apun}
Suppose that $G$ is a finite group and $\mathcal{B}=(B,B')$ is an $\epsilon$-closed $r$-multiplicative pair. Then
\begin{equation*}
0 \leq \|v\|_{L^2(\mu_B)}^2 - \|\rho_{\mathcal{B},y}(v)\|_{L^2(\mu_B)}^2 \leq \epsilon\|v\|_{L^\infty(\mu_B)}^2,
\end{equation*}
whenever $y \in B'^r$ and $v \in L^\infty(\mu_B)$.
\end{lemma}
\begin{proof}
We evaluate the expression on the left
\begin{eqnarray*}
\|\rho_{\mathcal{B},y}(v)\|_{L^2(\mu_B)}^2 &=& \frac{1}{\mu_G(B)}\int{|v(xy)|^21_B(x)d\mu_{G}(x)} \\ &= &\frac{1}{\mu_G(B)} \int{|v(z)|^21_B(zy^{-1})d\mu_G(z)}.
\end{eqnarray*}
The lower bound follows immediately by non-negativity of the integrand.  For the upper bound just note that
\begin{equation*}
|\frac{1}{\mu_G(B)} \int{|v(z)|^21_B(zy^{-1})d\mu_G(z)} -\frac{1}{\mu_G(B)} \int{|v(z)|^21_B(z)d\mu_G(z)}|\end{equation*}
is then at most
\begin{equation*}
 \|v\|_{L^\infty(\mu_B)}^2\|\rho_{y^{-1}}(\mu_B) - \mu_B\|
\end{equation*}
by the triangle inequality.  The result follows from Lemma \ref{lem.approxhaar} since $y^{-1} \in B'^{-r}=B'^r$.
\end{proof}
Now we turn to showing that $\rho_{\mathcal{B}',y}$ approximately commutes with $L_{\mathcal{B},f}$ if $\mathcal{B}'$ and $\mathcal{B}$ are suitably related. The following lemma encodes this fact.
\begin{lemma}[Approximate commuting]\label{lem.approxtrans}
Suppose that $G$ is a finite group, $\mathcal{B}=(B,B')$ is a $c$-thick multiplicative pair, $\mathcal{B}'=(B,B'')$ is an $\epsilon'$-closed $r$-multiplicative pair and $f \in L^\infty(\mu_{B'})$.  Then
\begin{equation*}
\|\rho_{\mathcal{B}',y}L_{\mathcal{B},f}v - L_{\mathcal{B},f}\rho_{\mathcal{B}',y}v\|_{L^2(\mu_B)}^2=O(\epsilon'c^{-2}\|v\|_{L^\infty(\mu_B)}^2\|f\|_{L^\infty(\mu_{B'})}^2)
\end{equation*}
for all $y \in B''^r$ and $v \in L^\infty(\mu_B)$.
\end{lemma}
\begin{proof}
We examine the two terms on the left individually.  With the first term we have
\begin{equation*}
\rho_{\mathcal{B}',y}L_{\mathcal{B},f}v(x)= \frac{1}{\mu_G(B')}\int{f(z)v(z^{-1}xy)d\mu_G(z)}1_B(xy)1_B(x).
\end{equation*}
By change of variables $z^{-1}xy = u$ this gives
\begin{equation}\label{eqn.uselater}
\rho_{\mathcal{B}',y}L_{\mathcal{B},f}v(x)= \frac{\mu_G(B)}{\mu_G(B')}\int{f(xyu^{-1})v(u)d\mu_B(u)}1_B(xy)1_B(x).
\end{equation}
On the other hand the second term has 
\begin{equation*}
L_{\mathcal{B},f}\rho_{\mathcal{B}',y}v(x)=\frac{1}{\mu_G(B')} \int{f(z)v(z^{-1}xy)1_{B}(z^{-1}x)d\mu_G(z)}1_B(x).
\end{equation*}
Now, make the change of variables $u=z^{-1}xy$ to get
\begin{equation*}
L_{\mathcal{B},f}\rho_{\mathcal{B}',y}v(x)=\frac{\mu_G(B)}{\mu_G(B')} \int{f(xyu^{-1})v(u)d\rho_{y^{-1}}(\mu_{B})(u)}1_B(x).
\end{equation*}
However $\|\rho_{y^{-1}}(\mu_{B}) - \mu_B\| \leq \epsilon'$ by Lemma \ref{lem.approxhaar} since $y \in B''^r$, whence
\begin{equation*}
|L_{\mathcal{B},f}\rho_{\mathcal{B}',y}v(x) - \frac{\mu_G(B)}{\mu_G(B')} \int{f(xyu^{-1})v(u)d\mu_{B}(u)}1_B(x)|
\end{equation*}
is at most
\begin{equation*}
\epsilon'c^{-1} \|v\|_{L^\infty(\mu_B)}\|f\|_{L^\infty(\mu_{B'})}1_B(x).
\end{equation*}
Combining this with (\ref{eqn.uselater}) we get that
\begin{equation*}
|\rho_{\mathcal{B}',y}L_{\mathcal{B},f}v(x) - L_{\mathcal{B},f}\rho_{\mathcal{B}',y}v(x)|
\end{equation*}
is at most
\begin{equation*}
\frac{\mu_G(B)}{\mu_G(B')}|f \ast v(xy)||1_B(xy)1_B(x)-1_B(x)| + \epsilon'c^{-1} \|v\|_{L^\infty(\mu_B)}\|f\|_{L^\infty(\mu_{B'})}1_B(x).
\end{equation*}
Integrating the square of this against $d\mu_B$ and applying the Cauchy-Schwarz inequality we see that
\begin{equation*}
\|\rho_{\mathcal{B}',y}L_{\mathcal{B},f}v - L_{\mathcal{B},f}\rho_{\mathcal{B}',y}v\|_{L^2(\mu_B)}^2
\end{equation*}
is at most
\begin{equation*}
2(c^{-2}\|f \ast v\|_{L^\infty(\mu_G)}^2\|\rho_y(1_{B})-1_B\|_{L^2(\mu_B)}^2 + \epsilon'^2c^{-2} \|v\|_{L^\infty(\mu_B)}^2\|f\|_{L^\infty(\mu_{B'})}^2).
\end{equation*}
Furthermore, we trivially have $\|f \ast v\|_{L^\infty(\mu_G)} \leq \|f\|_{L^\infty(\mu_{B'})}\|v\|_{L^\infty(\mu_{B})}$ since $\supp f \subset B'$ and $\supp v \subset B$, whence
\begin{equation*}
\|\rho_{\mathcal{B}',y}L_{\mathcal{B},f}v - L_{\mathcal{B},f}\rho_{\mathcal{B}',y}v\|_{L^2(\mu_B)}^2
\end{equation*}
is at most
\begin{equation*}
2c^{-2} \|v\|_{L^\infty(\mu_B)}^2\|f\|_{L^\infty(\mu_{B'})}^2(\|\rho_y(1_{B})-1_B\|_{L^2(\mu_B)}^2 + \epsilon'^2).
\end{equation*}
It remains to note that
\begin{eqnarray*}
\|\rho_y(1_{B})-1_B\|_{L^2(\mu_B)}^2& =& \int{|1_{B}(xy)1_B(x)-1_B(x)|^2d\mu_B(x)}\\ & =&\frac{1}{\mu_G(B)} \int{|1_{B}(xy) - 1_B(x)|1_B(x)d\mu_G(x)}\\ & \leq & \frac{1}{\mu_G(B)} \int{|1_{B}(xy) - 1_B(x)|d\mu_G(x)}\\  & = & \|\rho_y(\mu_B) - \mu_B\| \leq \epsilon'.
\end{eqnarray*}
The result is proved.
\end{proof}
As with Parseval's theorem we do not have an analogue of the inversion formula relative to multiplicative pairs.  We do, however, have an analogue of Fourier bases.  Suppose that $G$ is a finite group, $\mathcal{B}=(B,B')$ is a multiplicative pair and $f \in L^1(\mu_{B'})$. We write
\begin{equation*}
s_i(\mathcal{B},f):=s_i(L_{\mathcal{B},f})
\end{equation*}
for the singular values of $L_{\mathcal{B},f}$ and call an orthonormal basis $v_1,\dots,v_n$ of $L^2(\mu_B)$ a \emph{Fourier basis of $L^2(\mu_B)$ for $f$} if
\begin{equation*}
L_{\mathcal{B},f}^*L_{\mathcal{B},f}v_i=|s_i(\mathcal{B},f)|^2v_i \textrm{ for all } 1 \leq i \leq n.
\end{equation*}
The following is an immediate consequence of Corollary \ref{cor.speccor}.
\begin{proposition}[Local Fourier bases]\label{prop.diag}
Suppose that $G$ is a finite group, $\mathcal{B}:=(B,B')$ is a multiplicative pair and $f\in L^1(\mu_{B'})$. Then there is a Fourier basis of $L^2(\mu_B)$ for $f$.
\end{proposition}

\section{The spectrum of convolution operators on multiplicative pairs}\label{sec.locspec}

In this section we define the spectrum of a function on a multiplicative pair and develop some of the basic facts.  To analyse functions on a multiplicative pair in a Fourier theoretic spirit we shall analyse the spectrum of the function so it will be a very important structure.

We begin with the Hausdorff-Young inequality.
\begin{lemma}[Hausdorff-Young]\label{lem.hdy}
Suppose that $G$ is a finite group, $\mathcal{B}=(B,B')$ is a multiplicative pair and $f \in L^1(\mu_{B'})$.  Then
\begin{equation*}
|s_1(\mathcal{B},f)| \leq \|f\|_{L^1(\mu_{B'})}.
\end{equation*}
\end{lemma}
\begin{proof}
Let $v \in L^2(\mu_B)$ be a unit vector such that
\begin{equation*}
\|L_{\mathcal{B},f}v\|_{L^2(\mu_B)}^2=\|L_{\mathcal{B},f}\|^2=|s_1(\mathcal{B},f)|^2.
\end{equation*}
Writing out the first of these terms we get that it is equal to
\begin{equation*}
\frac{1}{\mu_G(B)\mu_G(B')^2}\int{\left|\int{f(z)v(z^{-1}y)d\mu_G(z)}\right|^21_B(y)d\mu_G(y)}.
\end{equation*}
By non-negativity of the integrand, the integral is at most
\begin{equation*}
\int{\left|\int{f(z)v(z^{-1}y)d\mu_G(z)}\right|^2d\mu_G(y)}=  \|f \ast v\|_{L^2(\mu_G)}^2.
\end{equation*}
Young's inequality tells us that
\begin{equation*}
\|f \ast v\|_{L^2(\mu_G)}^2 \leq \|f\|_{L^1(\mu_G)}^2\|v\|_{L^2(\mu_G)}^2 =  \|f\|_{L^1(\mu_G)}^2.\mu_G(B),
\end{equation*}
whence
\begin{equation*}
|s_1(\mathcal{B},f)| ^2 \leq \frac{1}{\mu_G(B)\mu_G(B')^2}\|f\|_{L^1(\mu_G)}^2.\mu_G(B) = \|f\|_{L^1(\mu_{B'})}^2.
\end{equation*}
The result is proved.
\end{proof}
In light of the preceeding we make the following definition.  Suppose that $G$ is a finite group, $\mathcal{B}=(B,B')$ is a multiplicative pair and $f \in L^1(\mu_{B'})$. The \emph{$\delta$-spectrum of $f$} is then defined to be the space
\begin{equation*}
\Spec_\delta(\mathcal{B},f):=\bigoplus_{i: |s_i(\mathcal{B},f)| \geq \delta \|f\|_{L^1(\mu_{B'})}}\{v \in L^2(\mu_B) : L_{\mathcal{B},f}^*L_{\mathcal{B},f}v=|s_i(\mathcal{B},f)|^2v\}.
\end{equation*}
Note that the spaces on the right are eigenspaces and hence vector spaces, so the definition makes sense, and is a subspace of $L^2(\mu_B)$.

Again, when $G$ is abelian and $\mathcal{B}=(G,G)$ then the characters $\gamma$ for which $|\wh{f}(\gamma)| \geq \delta \|f\|_{L^1(\mu_G)}$ form a basis for $\Spec_\delta(\mathcal{B},f)$ and we are often interested in bounding the number of such characters.  There is a relatively easy bound sometimes called the Parseval bound which follows from Bessel's inequality (or Parseval's theorem).  We now prove an analogue in our setting.
\begin{lemma}[The Parseval bound]\label{lem.pbd}
Suppose that $G$ is a finite group, $\mathcal{B}=(B,B')$ is a $c$-thick multiplicative pair, $f \in L^2(\mu_B)$ is not identically zero and $\delta \in (0,1]$ is a parameter. Then
\begin{equation*}
\dim \Spec_\delta(\mathcal{B},f) \leq c^{-1}\delta^{-2}\|f\|_{L^1(\mu_{B'})}^{-2}\|f\|_{L^2(\mu_{B'})}^2.
\end{equation*}
\end{lemma}
\begin{proof}
Let $v_1,\dots,v_n$ be a Fourier basis of $L^2(\mu_B)$ for $f$ as afforded by Proposition \ref{prop.diag}.  Writing $d$ for the dimension of $\Spec_\delta(\mathcal{B},f)$, we have
\begin{equation*}
|s_i(\mathcal{B},f)| \geq \delta \|f\|_{L^1(\mu_{B'})} \textrm{ whenever } i \leq d.
\end{equation*}
In view of this
\begin{eqnarray*}
\|f\|_{L^1(\mu_B)}^2\delta^2d&\leq &\sum_{i=1}^d{\| L_{\mathcal{B},f}v_i\|_{L^2(\mu_{B'})}^2}\\& \leq& \sum_{i=1}^n{\| L_{\mathcal{B},f}v_i\|_{L^2(\mu_{B'})}^2} = \sum_{i=1}^n{\langle L_{\mathcal{B},f}^*L_{\mathcal{B},f}v_i,v_i\rangle_{L^2(\mu_B)}}.
\end{eqnarray*}
The right hand side of this is just the trace of $L_{\mathcal{B},f}^*L_{\mathcal{B},f}$ which we bound using the Bessel-type inequality in Proposition \ref{prop.locbes}; we get
\begin{equation*}
\|f\|_{L^1(\mu_{B'})}^2\delta^2d \leq c^{-1}\|f\|_{L^2(\mu_{B'})}^2,
\end{equation*}
and the lemma is proved after some rearranging.
\end{proof}

In the abelian setting, thanks to the Fourier transform, we are able to take all the vectors in a Fourier basis to be characters of the group, and these have the nice property that they are bounded in $L^\infty$.  This need not be true in the more general setting, but we do have the following bound.
\begin{lemma}\label{lem.evalbd}
Suppose that $G$ is a finite group, $\mathcal{B}=(B,B')$ is a $c$-thick multiplicative pair, $f \in L^2(\mu_{B'})$ and $v$ is a unit eigenvector of $L_{\mathcal{B},f}^*L_{\mathcal{B},f}$ with non-zero eigenvalue $|\lambda|^2$. Then
\begin{equation*}
\|v\|_{L^\infty(\mu_{B})} \leq |\lambda|^{-2}c^{-1/2}\|f\|_{L^1(\mu_{B'})}\|f\|_{L^2(\mu_{B'})}.
\end{equation*}
\end{lemma}
\begin{proof}
Since $v$ is an eigenvector of $L_{\mathcal{B},f}^*L_{\mathcal{B},f}$ we have that
\begin{equation*}
((\tilde{f}d\mu_{B'})\ast (((fd\mu_{B'}) \ast v)|_B))|_B = |\lambda|^2 v,
\end{equation*}
so that
\begin{eqnarray*}
|\lambda|^2 \|v\|_{L^\infty(\mu_B)} & =& \mu_G(B')^{-2}\| \tilde{f}\ast((f\ast v)1_B)\|_{L^\infty(\mu_B)} \\ &\leq & \mu_G(B')^{-2}\|\tilde{f}\|_{L^1(\mu_G)}\| f\ast v\|_{L^\infty(\mu_G)}
\end{eqnarray*}
by a trivial instance of Young's inequality.  By a different instance of Young's inequality we also have
\begin{eqnarray*}
\| f\ast v\|_{L^\infty(\mu_G)}& \leq &\|f\|_{L^2(\mu_G)}\|v\|_{L^2(\mu_G)}\\& = &(\mu_G(B)\mu_G(B'))^{1/2}\|f\|_{L^2(\mu_{B'})}\|v\|_{L^2(\mu_B)}.
\end{eqnarray*}
Inserting this in the previous and noting that $\|\tilde{f}\|_{L^1(\mu_G)} = \mu_G(B')\|f\|_{L^1(\mu_{B'})}$ we get the result.
\end{proof}
A useful corollary of this is that all unit vectors in the large spectrum have well controlled $L^\infty$-norm.
\begin{corollary}
\label{cor.genl8bd}
Suppose that $G$ is a finite group, $\mathcal{B}=(B,B')$ is a $c$-thick multiplicative pair, $f \in L^2(\mu_{B'})$ is not identically zero and $\delta\in (0,1]$ is a parameter. Then
\begin{equation*}
\|v\|_{L^\infty(\mu_{B})} \leq \delta^{-3}c^{-1}\|f\|_{L^1(\mu_{B'})}^{-2}\|f\|_{L^2(\mu_{B'})}^2
\end{equation*}
for all unit vectors $v \in \Spec_\delta(\mathcal{B},f)$.
\end{corollary}
\begin{proof}
Let $v_1,\dots,v_n$ be a Fourier basis of $L^2(\mu_B)$ for $f$ as afforded by Proposition \ref{prop.diag}.  Writing $d$ for the dimension of $\Spec_{\delta}(\mathcal{B},f)$ we have that $v_1,\dots,v_d \in \Spec_{\delta}(\mathcal{B},f)$, and we may decompose $v \in \Spec_\delta(\mathcal{B},f)$ as
\begin{equation*}
v=\sum_{i=1}^d{\mu_iv_i} \textrm{ where } \sum_{i=1}^d{|\mu_i|^2}=1.
\end{equation*}
By the triangle inequality we have that
\begin{equation}\label{eqn.www}
\|v\|_{L^\infty(\mu_B)} \leq \sum_{i=1}^d{|\mu_i|}\sup_{1 \leq i \leq d}{\|v_i\|_{L^\infty(P_B)}}.
\end{equation}
We estimate the left hand term on the right by the Cauchy-Schwarz inequality which shows that it is at most $\sqrt{d}$.  This may, in turn, be bounded by the Parseval bound from Lemma \ref{lem.pbd} to get that
\begin{equation*}
d \leq c^{-1}\delta^{-2}\|f\|_{L^1(\mu_{B'})}^{-2}\|f\|_{L^2(\mu_{B'})}^2.
\end{equation*}
We now estimate the right most term in (\ref{eqn.www}) by Lemma \ref{lem.evalbd}:
\begin{equation*}
\sup_{1 \leq i \leq d}{\|v_i\|_{L^\infty(P_B)}} \leq \delta^{-2}c^{-1/2} \|f\|_{L^1(\mu_{B'})}^{-1}\|f\|_{L^2(\mu_{B'})}.
\end{equation*}
Combining all this gives the result.
\end{proof}

\section{Analysis of the large spectrum}\label{sec.anallarge}

The large spectrum determines the average behaviour of a function. Indeed, given a multiplicative pair $\mathcal{B}=(B,B')$ and a function $f \in L^1(\mu_{B'})$ we should like to find a large set $B''$ such that $\mathcal{B}'=(B,B'')$ is a multiplicative pair and
\begin{equation*}
\|\rho_{\mathcal{B}',y}(v) - v\|_{L^2(\mu_B)} \leq \epsilon \|v\|_{L^2(\mu_B)} \textrm{ for all }v \in \Spec_\delta(\mathcal{B},f),
\end{equation*}
and then decompose $f$ as
\begin{equation*}
f = f \ast \mu_{B''} + (f- f \ast \mu_{B''} ).
\end{equation*}
In the abelian setting the set $B''$ is just the Bohr set corresponding to the characters in $\Spec_\delta(\mathcal{B},f)$, which is large since the spectrum has bounded dimension by the Parseval bound.  In this section we show that there is a large set $B''$ such that
\begin{equation*}
\Spec_\delta(\mathcal{B},f)\rightarrow \Spec_\delta(\mathcal{B},f); v \mapsto \rho_{\mathcal{B}',y}(v)|_{\Spec_\delta(\mathcal{B},f)}
\end{equation*}
is close to a representation in a certain sense, and then this can be combined with the theory of non-abelian Bohr sets from the next section to get our desired decomposition.

The above is just a sketch and we now turn to the business of realising a version of it.  To state our results we shall find it useful to have one extra piece of notation: given a set $B$ and a function $f \in L^1(\mu_B)$, it will be useful to define the \emph{width} of $f$ to be
\begin{equation*}
w(f):=\|f\|_{L^1(\mu_B)}\|f\|_{L^\infty(\mu_B)}^{-1}.
\end{equation*}
This is just a notational convenience, but to help intuition, think of the case when $f$ is an indicator function. Then $w(f)$ is just the density of its support.

We begin by showing that the whole large spectrum is almost closed under translation by elements in a sufficiently small ball. 
\begin{lemma}\label{lem.trans}
Suppose that $G$ is a finite group, $\mathcal{B}=(B,B')$ is a $c$-thick multiplicative pair, $\mathcal{B}'=(B,B'')$ is an $\epsilon'$-closed $r$-multiplicative pair, $f \in L^1(\mu_{B'})$ is not identically zero, and $\delta,\eta \in (0,1]$ are parameters. Then
\begin{equation*}
d(\rho_{\mathcal{B}',y}v,\Spec_{\delta-\eta}(\mathcal{B},f))^2 =O( \epsilon'\eta^{-O(1)}c^{-O(1)}\delta^{-O(1)}w(f)^{-O(1)}\|v\|_{L^2(\mu_B)}^2)
\end{equation*}
for all $v \in \Spec_\delta(\mathcal{B},f)$ and $y \in B''^r$.
\end{lemma}
We remind the reader that if $H$ is a finite dimensional Hilbert space, $V \leq H$ and $v \in H$ then
\begin{equation*}
d(v,V) = \inf{\{\|v-v'\| : v' \in V\}}.
\end{equation*}
Before proving this lemma we establish the result for eigenvectors.
\begin{lemma}
Suppose that $G$ is a finite group, $\mathcal{B}=(B,B')$ is a $c$-thick multiplicative pair, $\mathcal{B}'=(B,B'')$ is an $\epsilon'$-closed $r$-multiplicative pair, $f \in L^1(\mu_{B'})$ is not identically zero, $v \in \Spec_\delta(\mathcal{B},f)$ is a unit eigenvector of $L_{\mathcal{B},f}^*L_{\mathcal{B},f}$ and $0<\eta \leq \delta \leq 1$ are parameters.  Then
\begin{equation*}
d(\rho_{\mathcal{B}',y}v,\Spec_{\delta-\eta}(\mathcal{B},f))^2 =O( \epsilon'\eta^{-O(1)}c^{-O(1)}\delta^{-O(1)}w(f)^{-O(1)})
\end{equation*}
for all $y \in B''^r$.
\end{lemma}
\begin{proof}
Since $L_{\mathcal{B},f}^*=L_{\mathcal{B},\tilde{f}}$ we may apply Lemma \ref{lem.approxtrans} to get that
\begin{equation*}
\|\rho_{\mathcal{B}',y}L_{\mathcal{B},f}^*L_{\mathcal{B},f}v - L_{\mathcal{B},f}^*\rho_{\mathcal{B}',y}L_{\mathcal{B},f}v\|_{L^2(\mu_B)}^2=O(\epsilon'c^{-2}\|L_{\mathcal{B},f}v\|_{L^\infty(\mu_B)}^2\|\tilde{f}\|_{L^\infty(\mu_{B'})}^2).
\end{equation*}
Of course $\|\tilde{f}\|_{L^\infty(\mu_{B'})} = \|f\|_{L^\infty(\mu_{B'})}$ and
\begin{equation*}
\|L_{\mathcal{B},f}v\|_{L^\infty(\mu_B)} = \|((fd\mu_{B'}) \ast v)1_B\|_{L^\infty(\mu_G)} \leq \|f\|_{L^1(\mu_{B'})}\|v\|_{L^\infty(\mu_B)}
\end{equation*}
by Young's inequality, whence
\begin{equation}\label{eqn.uyy}
\|\rho_{\mathcal{B}',y}L_{\mathcal{B},f}^*L_{\mathcal{B},f}v - L_{\mathcal{B},f}^*\rho_{\mathcal{B}',y}L_{\mathcal{B},f}v\|_{L^2(\mu_B)}^2
\end{equation}
is at most
\begin{equation*}
O(\epsilon'c^{-2}\|v\|_{L^\infty(\mu_B)}^2\|f\|_{L^1(\mu_{B'})}^2\|f\|_{L^\infty(\mu_{B'})}^2).
\end{equation*}
On the other hand, by the Hausdorff-Young bound (Lemma \ref{lem.hdy}) we have that
\begin{equation*}
\|L_{\mathcal{B},f}^*\rho_{\mathcal{B},y}L_{\mathcal{B},f}v - L_{\mathcal{B},f}^*L_{\mathcal{B},f}\rho_{\mathcal{B},y}v\|_{L^2(\mu_B)}^2
\end{equation*}
is at most
\begin{equation*}
\|f\|_{L^1(\mu_{B'})}^2.\|\rho_{\mathcal{B},y}L_{\mathcal{B},f}v - L_{\mathcal{B},f}\rho_{\mathcal{B},y}v\|_{L^2(\mu_B)}^2.
\end{equation*}
We estimate the second of these terms by Lemma \ref{lem.approxtrans} to get that
\begin{equation*}
\|L_{\mathcal{B},f}^*\rho_{\mathcal{B},y}L_{\mathcal{B},f}v - L_{\mathcal{B},f}^*L_{\mathcal{B},f}\rho_{\mathcal{B},y}v\|_{L^2(\mu_B)}^2
\end{equation*}
is at most
\begin{equation*}
O(\epsilon'c^{-2}\|v\|_{L^\infty(\mu_B)}^2\|f\|_{L^1(\mu_{B'})}^2\|f\|_{L^\infty(\mu_{B'})}^2).
\end{equation*}
Combining this with the upper bound for (\ref{eqn.uyy}) by the triangle inequality it follows that
\begin{equation*}
\|\rho_{\mathcal{B}',y}L_{\mathcal{B},f}^*L_{\mathcal{B},f}v - L_{\mathcal{B},f}^*L_{\mathcal{B},f}\rho_{\mathcal{B}',y}v\|_{L^2(\mu_B)}^2
\end{equation*}
is at most
\begin{equation*}
O(\epsilon'c^{-2}\|v\|_{L^\infty(\mu_B)}^2\|f\|_{L^1(\mu_{B'})}^2\|f\|_{L^\infty(\mu_{B'})}^2).
\end{equation*}

Now, $L_{\mathcal{B},f}^*L_{\mathcal{B},f}v=|\lambda|^2 v$ for some $\lambda$ with $|\lambda| \geq \delta \|f\|_{L^1(\mu_{B'})}>0$, whence
\begin{equation*}
\||\lambda|^2 \rho_{\mathcal{B}',y}v - L_{\mathcal{B},f}^*L_{\mathcal{B},f}\rho_{\mathcal{B}',y}v\|_{L^2(\mu_B)}^2  =O(\epsilon'c^{-2} \|v\|_{L^\infty(\mu_B)}^2\|f\|_{L^1(\mu_G)}^2\|f\|_{L^\infty(\mu_{B'})}^2).
\end{equation*}
Apply Lemma \ref{lem.evalbd} to bound $\|v\|_{L^\infty(\mu_{B})}$ from above so that
\begin{equation*}
\||\lambda|^2 \rho_{\mathcal{B}',y}v - L_{\mathcal{B},f}^*L_{\mathcal{B},f}\rho_{\mathcal{B}',y}v\|_{L^2(\mu_B)}^2
\end{equation*}
is at most
\begin{equation*}
O( \epsilon'c^{-3}|\lambda|^{-4}\|f\|_{L^2(\mu_{B'})}^{2}\|f\|_{L^1(\mu_{B'})}^{4}\|f\|_{L^\infty(\mu_{B'})}^{2}).
\end{equation*}
Since
\begin{equation*}
\|f\|_{L^2(\mu_{B'})}^2 \leq \|f\|_{L^1(\mu_{B'})}\|f\|_{L^\infty(\mu_{B'})} \textrm{ and } |\lambda| \geq \delta\|f\|_{L^1(\mu_{B'})}
\end{equation*}
we can simplify this to give
\begin{equation*}
\||\lambda|^2 \rho_{\mathcal{B}',y}v - L_{\mathcal{B},f}^*L_{\mathcal{B},f}\rho_{\mathcal{B}',y}v\|_{L^2(\mu_B)}^2  =O( \epsilon'c^{-3}\delta^{-4}w(f)\|f\|_{L^\infty(\mu_{B'})}^4).
\end{equation*}
Now, let $v_1,\dots,v_n$ be a Fourier basis of $L^2(\mu_B)$ for $f$  -- such a basis is afforded by Proposition \ref{prop.diag}. Decompose
\begin{equation*}
\rho_{\mathcal{B}',y}v=\sum_{i=1}^n{\mu_iv_i},
\end{equation*}
and insert this into the previous bound to get that
\begin{equation*}
\sum_{i=1}^n{|\mu_i|^2(|\lambda|^2-|s_i(\mathcal{B},f)|^2)^2} =O( \epsilon'c^{-3}\delta^{-4}w(f)\|f\|_{L^\infty(\mu_{B'})}^4).
\end{equation*}
Thus, if $|s_i(\mathcal{B},f)| \leq (\delta-\eta)\|f\|_{L^1(\mu_{B'})}$ then
\begin{equation*}
(|\lambda|^2 - |\lambda_i|^2)^2 \geq \delta^2\eta^2 \|f\|_{L^1(\mu_{B'})}^4.
\end{equation*}
It follows that
\begin{equation*}
\sum_{i: |\lambda_i| \leq (\delta-\eta)\|f\|_{L^1(\mu_{B'})}}{|\mu_i|^2} =O( \epsilon'\eta^{-2}c^{-3}\delta^{-6}w(f)^{-3}).
\end{equation*}
On the other hand $v_i \in \Spec_{\delta-\eta}(\mathcal{B},f)$ for all $i$ with $|s_i(\mathcal{B},f)| \geq (\delta - \eta)\|f\|_{L^1(\mu_{B'})}$ whence
\begin{equation*}
\sum_{i: |\lambda_i| > (\delta-\eta)\|f\|_{L^1(\mu_{B'})}}{\mu_iv_i} \in \Spec_{\delta-\eta}(\mathcal{B},f),
\end{equation*}
and so
\begin{equation*}
d(\rho_{\mathcal{B}',y}v, \Spec_{\delta-\eta}(\mathcal{B},f))^2 =O( \epsilon'\eta^{-2}c^{-3}\delta^{-6}w(f)^{-3})
\end{equation*}
and we have the result.
\end{proof}
The proof of our desired lemma is now a straightforward corollary.
\begin{proof}[Proof of Lemma \ref{lem.trans}]
Let $v_1,\dots,v_n$ be a Fourier basis of $L^2(\mu_B)$ for $f$.  We may rescale $v$ so that it is a unit vector and writing $d$ for the dimension of $\Spec_{\delta}(\mathcal{B},f)$ we there are complex numbers $\mu_1,\dots,\mu_d$ such that
\begin{equation*}
v=\sum_{i=1}^d{\mu_iv_i} \textrm{ and } \sum_{i=1}^d{|\mu_i|^2}=1.
\end{equation*}
By the triangle inequality we have that
\begin{equation*}
d(\rho_{\mathcal{B}',y}v,\Spec_{\delta-\eta}(\mathcal{B},f))^2 \leq (\sum_{i=1}^d{|\mu_i|})^2\sup_{1 \leq i \leq d}{d(\rho_{\mathcal{B}',y}v_i,\Spec_{\delta-\eta}(\mathcal{B},f))^2}.
\end{equation*}
On the other hand the first term on the right is at most $d$ by the Cauchy-Schwarz inequality and we can bound this by the Parseval bound in Lemma \ref{lem.pbd}:
\begin{equation*}
d \leq c^{-1}\delta^{-2}\|f\|_{L^1(\mu_{B'})}^{-2}\|f\|_{L^2(\mu_{B'})}^2 \leq c^{-1}\delta^{-2}w(f)^{-1}.
\end{equation*}
The second term is bounded by the preceeding lemma and we get the result.
\end{proof}
We know from Lemma \ref{lem.apun} that the maps $\rho_{\mathcal{B}',y}$ have $\|\rho_{\mathcal{B},y}v\|_{L^2(\mu_B)}$ close to $\|v\|_{L^2(\mu_B)}$, however they are not necessarily endomorphisms.  The preceeding lemma shows that they are close to almost endomorphisms of $\Spec_\delta(\mathcal{B},f)$, and we can now apply the pigeonhole principle to see that there is a point when they are close to actual endomorphisms.
\begin{lemma}\label{lem.pcore}
Suppose that $G$ is a finite group, $\mathcal{B}=(B,B')$ is a $c$-thick multiplicative pair, $f \in L^1(\mu_{B'})$ is not identically zero, and $\delta \in (0,1]$ is a parameter.  Then there is some $\delta' \in (\delta/2,\delta]$ and an $\eta \in (0,1]$ with 
\begin{equation*}
\eta = \Omega(\delta^3c w(f))
\end{equation*}
such that
\begin{equation*}
\Spec_{\delta'}(\mathcal{B},f) = \Spec_{\delta'-\eta}(\mathcal{B},f).
\end{equation*}
\end{lemma}
\begin{proof}
By the Parseval bound in Lemma \ref{lem.pbd} we have that
\begin{equation*}
\dim \Spec_{\delta/2}(f) \leq k:=\lfloor 4\delta^{-2}c^{-1}\|f\|_{L^1(\mu_{B'})}^{-2}\|f\|_{L^2(\mu_{B'})}^2\rfloor.
\end{equation*}
Consider the sequence of $k+2$ spaces
\begin{eqnarray*}
\Spec_{\delta}(\mathcal{B},f) & \leq & \Spec_{\delta - \delta/(2k+2)}(\mathcal{B},f) \\ & \leq& \Spec_{\delta - 2\delta/(2k+2)}(\mathcal{B},f)\\&  \leq&  \dots \\ & \leq & \Spec_{\delta - (k+1)\delta/(2k+2)}(\mathcal{B},f).
\end{eqnarray*}
Since $ \Spec_{\delta-(k+1)\delta/(2k+2)}(\mathcal{B},f)=\Spec_{\delta/2}(\mathcal{B},f)$ has dimension at most $k$ it follows that some two of the spaces in the sequence must have the same dimension and hence be equal.

The fact that $\|f\|_{L^2(\mu_{B'})}^2 \leq \|f\|_{L^1(\mu_{B'})}\|f\|_{L^\infty(\mu_{B'})}$ now completes the lemma.
\end{proof}
The next lemma combines our work so far to give us a ball of almost unitary endomorphisms.
\begin{lemma}\label{lem.nr}
Suppose that $G$ is a finite group, $\mathcal{B}=(B,B')$ is a $c$-thick multiplicative pair, $f \in L^1(\mu_{B'})$ is not identically zero, and $\delta \in (0,1]$ is a parameter. Then there is some $\delta' \in (\delta/2,\delta]$ such that if  $\mathcal{B}'=(B,B'')$ is an $\epsilon'$-closed $r$-multiplicative pair, then
\begin{equation*}
d(\rho_{\mathcal{B}',y}v,\Spec_{\delta'}(\mathcal{B},f))^2 = O(\epsilon'\delta'^{-O(1)}c^{-O(1)}w(f)^{-O(1)}\|v\|_{L^2(\mu_B)}^2)
\end{equation*}
for all $v \in \Spec_{\delta'}(\mathcal{B},f)$ and $y \in B''^r$.
\end{lemma}
\begin{proof}
Apply the previous lemma to get $\delta'$ and an $\eta$, and then Lemma \ref{lem.trans} to complete.
\end{proof}
We make two new definitions which will be convenient.  Let $C_{\mathcal{S}\mathcal{R}}>0$ be some absolute constant such that one has the bound
\begin{equation*}
d(\rho_{\mathcal{B}',y}v,\Spec_{\delta'}(\mathcal{B},f))^2 \leq\epsilon'(2\delta^{-1}c^{-1}w(f)^{-1})^{C_{\mathcal{S}\mathcal{R}}}\|v\|_{L^2(\mu_B)}^2
\end{equation*}
in Lemma \ref{lem.nr}.

Now, suppose that $G$ is a finite group, $\mathcal{B}=(B,B')$ is a $c$-thick multiplicative pair and $f \in L^1(\mu_{B'})$ is not identically zero.  We say that $\delta$ is \emph{regular} for $(\mathcal{B},f)$ if for every $\epsilon'$-closed $r$-multiplicative pair $\mathcal{B}'=(B,B'')$ we have
\begin{equation*}
d(\rho_{\mathcal{B}',y}v,\Spec_{\delta'}(\mathcal{B},f))^2 \leq \epsilon'(2\delta^{-1}c^{-1}w(f)^{-1})^{C_{\mathcal{S}\mathcal{R}}}\|v\|_{L^2(\mu_B)}^2
\end{equation*}
for all $v \in \Spec_{\delta}(\mathcal{B},f)$ and $y \in B''^r$.  Lemma \ref{lem.nr} guarantees a plentiful supply of regular values.

Furthermore, for each $y \in G$ we define the map
\begin{equation*}
T_{\mathcal{B},f,\delta,y}:\Spec_\delta(\mathcal{B},f) \rightarrow \Spec_\delta(\mathcal{B},f); v \mapsto \pi_{\Spec_{\delta}(\mathcal{B},f)}((\rho_yv)|_B),
\end{equation*}
where $\pi_{\Spec_{\delta}(\mathcal{B},f)}$ denotes the usual orthogonal projection of $L^2(\mu_B)$ onto the subspace $\Spec_{\delta}(\mathcal{B},f)$.

If $\delta$ is regular for $(\mathcal{B},f)$ then it turns out that the map $y \mapsto T_{\mathcal{B},f,\delta,y}$ is approximately a homomorphism -- the next lemma makes this precise. 
\begin{lemma}\label{lem.apxhom}
Suppose that $G$ is a finite group, $\mathcal{B}=(B,B')$ is a $c$-thick multiplicative pair, $f \in L^1(\mu_{B'})$ is not identically zero, and $\delta \in (0,1]$ is regular for $(\mathcal{B},f)$. Then for every $\epsilon'$-closed $r$-multiplicative pair $\mathcal{B}''=(B,B'')$ we have
\begin{equation*}
\|T_{\mathcal{B},f,\delta,yz}- T_{\mathcal{B},f,\delta,y}T_{\mathcal{B},f,\delta,z}\|^2 = O(\epsilon'\delta^{-O(1)}c^{-O(1)}w(f)^{-O(1)})
\end{equation*}
for all $z,y,yz \in B''^r$, and $T_{\mathcal{B},f,\delta,1_G}= I_{\Spec_\delta(\mathcal{B},f)}$, the identity on $\Spec_\delta(\mathcal{B},f)$.
\end{lemma}
\begin{proof}
Note that if $v \in \Spec_{\delta}(\mathcal{B},f)$ is a unit vector, and $z \in B''^r$ then
\begin{eqnarray*}
 \|\rho_zv - \rho_{\mathcal{B}',z}v\|_{L^2(\mu_G)}^2& =& \int{|v(xz) - v(xz)1_B(x)|^2d\mu_G(x)}\\  & = & \int{|v(xz)|^2|1_{B}(xz) - 1_B(x)|d\mu_G(x)}\\ & \leq & \|v\|_{L^\infty(\mu_B)}^2\epsilon'\mu_G(B)
\end{eqnarray*}
by Lemma \ref{lem.approxhaar} since $\mathcal{B}'=(B,B'')$ is an $\epsilon'$-closed $r$-multiplicative pair. We estimate $\|v\|_{L^\infty(\mu_B)}$ via Corollary \ref{cor.genl8bd} and the fact that $\|f\|_{L^2(\mu_{B'})}^2 \leq \|f\|_{L^1(\mu_{B'})}\|f\|_{L^\infty(\mu_{B'})}$ to get that
\begin{equation*}
\|v\|_{L^\infty(\mu_B)} =O( \delta^{-O(1)}c^{-O(1)}w(f)^{-O(1)}).
\end{equation*}
Inserting this in the previous we get that
\begin{equation*}
\|\rho_zv - \rho_{\mathcal{B}',z}v\|_{L^2(\mu_G)}^2=O(\epsilon'\delta^{-O(1)}c^{-O(1)}w(f)^{-O(1)}\mu_G(B)).
\end{equation*}
Moreover, since $\delta$ is regular
\begin{eqnarray*}
\|\rho_{\mathcal{B}',z}v - T_{\mathcal{B},f,\delta,z}v\|_{L^2(\mu_G)}^2 & = & \|\rho_{\mathcal{B}',z}v - T_{\mathcal{B},f,\delta,z}v\|_{L^2(\mu_B)}^2\mu_G(B)\\ & = & d(\rho_{\mathcal{B}',z}v,\Spec_{\delta}(\mathcal{B},f))^2\mu_G(B)\\ & = & O(\epsilon'\delta^{-O(1)}c^{-O(1)}w(f)^{-O(1)}\mu_G(B)).
\end{eqnarray*}
It follows from the triangle inequality that
\begin{equation}\label{eqn.apx}
\|\rho_zv - T_{\mathcal{B},f,\delta,z}v\|_{L^2(\mu_G)}^2 =O(\epsilon'\delta^{-O(1)}c^{-O(1)}w(f)^{-O(1)}\mu_G(B)),
\end{equation}
whenever $z \in B''^r$ and $v \in \Spec_\delta(\mathcal{B},f)$ is a unit vector.  In particular we also have
\begin{equation}\label{eqn.apx2}
\|\rho_yv - T_{\mathcal{B},f,\delta,y}v\|_{L^2(\mu_G)}^2 =O(\epsilon'\delta^{-O(1)}c^{-O(1)}w(f)^{-O(1)}\mu_G(B))
\end{equation}
whenever $y \in B''^r$ and $v \in \Spec_\delta(\mathcal{B},f)$ is a unit vector, and
\begin{equation}\label{eqn.apx3}
\|\rho_{yz}v - T_{\mathcal{B},f,\delta,yz}v\|_{L^2(\mu_G)}^2 =O(\epsilon'\delta^{-O(1)}c^{-O(1)}w(f)^{-O(1)}\mu_G(B)),
\end{equation}
whenever $yz \in B''^r$ and $v \in \Spec_\delta(\mathcal{B},f)$ is a unit vector.

Now, the map $\rho_y$ is unitary and $\rho_{yz}=\rho_y \rho_z$, so (\ref{eqn.apx}) gives that
\begin{eqnarray*}
\|\rho_{yz}v - \rho_yT_{\mathcal{B},f,\delta,z}v\|_{L^2(\mu_G)}^2 & = & \|\rho_y(\rho_z - T_{\mathcal{B},f,\delta,z})v\|_{L^2(\mu_G)}^2\\ & = &O(\epsilon'\delta^{-O(1)}c^{-O(1)}w(f)^{-O(1)}\mu_G(B)).
\end{eqnarray*}
Thus by the triangle inequality and (\ref{eqn.apx3}) we have that
\begin{equation*}
\|T_{\mathcal{B},f,\delta,yz}v - \rho_yT_{\mathcal{B},f,\delta,z}v\|_{L^2(\mu_G)}^2 = O(\epsilon'\delta^{-O(1)}c^{-O(1)}w(f)^{-O(1)}\mu_G(B)).
\end{equation*}
On the other hand applying (\ref{eqn.apx2}) to the vector $T_{\mathcal{B},f,\delta,z}v$ appropriately rescaled we get that
\begin{equation*}
\|\rho_yT_{\mathcal{B},f,\delta,z}v - T_{\mathcal{B},f,\delta,y}T_{\mathcal{B},f,\delta,z}v\|_{L^2(\mu_G)}^2 = O(\epsilon'\delta^{-O(1)}c^{-O(1)}w(f)^{-O(1)}\mu_G(B)),
\end{equation*}
since $\|T_{\mathcal{B},f,\delta,z}v\|_{L^2(\mu_B)} \leq \|v\|_{L^2(\mu_B)}$.  Finally the first conclusion of the lemma follows from the triangle inequality.  The second conclusion is immediate.
\end{proof}
The maps $T_{\mathcal{B},f,\delta,y}$ can also be well approximated by unitary maps, as the following lemma confirms.  To prove this we use a general operator theoretic result which says that if a map $M \in \End(H)$ has $\|Mv\|_H \approx 1$ for all unit vectors $v \in H$, then it is close to a unitary map.
\begin{lemma}\label{lem.approxunitary}
Suppose that $H$ is a $d$-dimensional complex Hilbert space, $M:H \rightarrow H$ and $\epsilon \in (0,1)$ is a parameter such that
\begin{equation*}
|\|Mv\|_H-1| \leq \epsilon \textrm{ for all unit } v \in H.
\end{equation*}
Then there is a unitary matrix $U$ such that $\|M-U\| \leq \epsilon$.
\end{lemma}
\begin{proof}
Let $v_1,\dots,v_d$ be an orthonormal basis of the type afforded by Corollary \ref{cor.speccor}, and note that
\begin{equation*}
\|Mv_i\|_H^2=\langle Mv_i,Mv_i\rangle_H = \langle M^*Mv_i,v_i\rangle_H = |s_i(M)|^2.
\end{equation*}
It follows from the hypothesis that $|s_i(M)-1| \leq \epsilon<1$ for all $i$.  In particular $s_i(M)>0$ for all $i$, whence we can define $U$ on the basis by $Uv_i:=Mv_i/s_i(M)$, extending to $H$ by linearity. 

First we check that
\begin{equation*}
\|Mv_i-Uv_i\|_H^2=\langle Mv_i,Mv_i\rangle_H + \langle Uv_i,Uv_i\rangle_H - 2\Re \langle Mv_i,Uv_i\rangle_H.
\end{equation*}
By construction of $U$ it follows that
\begin{equation*}
\|Mv_i-Uv_i\|_H^2=(|s_i(M)|^2+1-2|s_i(M)|)\leq \epsilon^2.
\end{equation*}
Since $v_1,\dots,v_d$ is an orthonormal basis for $H$ it follows that $\|M-U\| \leq \epsilon$ as claimed.

To complete the lemma we check that $U$ is unitary.  Again it suffices to check this on the basis:
\begin{eqnarray*}
\langle Uv_i,Uv_j\rangle_H& =& \frac{1}{|s_i(M)||s_j(M)|}\langle Mv_i,Mv_j\rangle_H\\ & =& \frac{1}{|s_i(M)||s_j(M)|}\langle M^*Mv_i,v_j\rangle_H =\frac{|s_i(M)|}{|s_j(M)|}\langle v_i,v_j\rangle_H.
\end{eqnarray*}
This quantity is $1$ is $i=j$ and $0$ otherwise.  It follows that $U$ is unitary and we are done.
\end{proof}
\begin{lemma}\label{lem.un}
Suppose that $G$ is a finite group, $\mathcal{B}=(B,B')$ is a $c$-thick multiplicative pair, $f \in L^1(\mu_{B'})$ is not identically zero, and $\delta$ is regular for $(\mathcal{B},f)$. Then for every $\epsilon'$-closed $r$-multiplicative pair $\mathcal{B}'=(B,B'')$, and every $y \in B''^r$ there is a unitary map $U_y \in U(\Spec_\delta(\mathcal{B},f))$ such that
\begin{equation*}
\|T_{\mathcal{B},f,\delta,y} - U_y\|^2 = O(\epsilon'\delta^{-O(1)}c^{-O(1)}w(f)^{-O(1)}).
\end{equation*}
\end{lemma}
\begin{proof}
Suppose that $v \in \Spec_\delta(\mathcal{B},f)$ is a unit vector.  By Lemma \ref{lem.apun} we have that
\begin{equation*}
|\|\rho_{\mathcal{B},y}v\|_{L^2(\mu_B)}^2 - 1| \leq \epsilon' \|v\|_{L^\infty(\mu_B)}^2
\end{equation*}
whenever $y \in B''^r$.  Now, by Corollary \ref{cor.genl8bd} we have that
\begin{equation*}
|\|\rho_{\mathcal{B},y}v\|_{L^2(\mu_B)}^2 - 1| =O(\epsilon' \delta^{-O(1)}c^{-O(1)}w(f)^{-O(1)}).
\end{equation*}
On the other hand, by regularity of $\delta$ we have
\begin{equation*}
\|\rho_{\mathcal{B},y}v - T_{\mathcal{B},f,\delta,y}v\|_{L^2(\mu_B)}^2 = O(\epsilon' \delta^{-O(1)}c^{-O(1)}w(f)^{-O(1)}),
\end{equation*}
whence
\begin{equation*}
|\|T_{\mathcal{B},f,\delta,y}v\|_{L^2(\mu_B)}^2 -1 | = O(\epsilon' \delta^{-O(1)}c^{-O(1)}w(f)^{-O(1)}).
\end{equation*}
The conclusion follows from Lemma \ref{lem.approxunitary}.
\end{proof}

\section{Bohr sets and balls in unitary groups}\label{sec.bohrandunball}

In this section we develop some basic size estimates for non-abelian Bohr sets.  We begin by recalling the traditional abelian definition of a Bohr set:  suppose that $G$ is a finite abelian group, $\Gamma=\{\gamma_1,\dots,\gamma_d\}$ is a set of homomorphisms $G \rightarrow S^1$, and $\delta \in (0,2]$. Then the \emph{Bohr set} with \emph{frequency set} $\Gamma$ and \emph{width} $\delta$ is
\begin{equation*}
\Bohr(\Gamma,\delta):=\{x \in G: |\gamma_i(x) - 1| \leq \delta \textrm{ for all } 1 \leq i \leq d\}.
\end{equation*}
There is, of course, an ever so slightly different (and more common) definition where we ask that $|\arg \gamma_i(x)| \leq \delta$ instead of $|\gamma_i(x)-1| \leq \delta$, but since the $\gamma_i$s are locally linear this difference plays no material r\^{o}le.

Now we shall present an equivalent definition which generalises to the non-abelian setting more easily.
\begin{lemma}[Alternative definition of Bohr sets]
Suppose that $G$ is a finite abelian group, $H$ is a $d$-dimensional Hilbert space and $\delta \in (0,2]$.  Then
\begin{enumerate}
\item given a set $\Gamma:=\{\gamma_1,\dots,\gamma_d\}$ of homomorphisms $G \rightarrow S^1$, there is a homomorphism $\gamma:G \rightarrow U(H)$ such that
\begin{equation*}
\Bohr(\Gamma,\delta)=\{x \in G: \|\gamma(x) -I\|\leq \delta\};
\end{equation*}
\item and conversely given a homomorphism $\gamma:G \rightarrow U(H)$ we get a set $\Gamma=\{\gamma_1,\dots,\gamma_d\}$ of homomorphisms $G \rightarrow S^1$ such that
\begin{equation*}
\Bohr(\Gamma,\delta)=\{x \in G:  \|\gamma(x) -I\| \leq \delta\}.
\end{equation*}
\end{enumerate}
\end{lemma}
\begin{proof}
Both parts are easy, but the first perhaps slightly more so.  Begin by letting $v_1,\dots,v_d$ be an orthonormal basis of $H$ and define a map $\gamma:G \rightarrow U(H)$ by
\begin{equation*}
x \mapsto \gamma(x):H \rightarrow H; \sum_{i=1}^d{\mu_iv_i} \mapsto \sum_{i=1}^d{\mu_i\gamma_i(x)v_i}.
\end{equation*}
It is easy to check that this is a well-defined homomorphism, and we also see that
\begin{equation*}
\|\gamma(x)-I\| \leq \delta \textrm{ if and only if } |\gamma_i(x) - 1| \leq \delta \textrm{ for all }1 \leq i \leq d.
\end{equation*}
The first part then follows immediately.

On the other hand given a homomorphism $\gamma:G \rightarrow U(H)$ we construct a frequency set as follows.  Since $G$ is abelian, $\gamma(G)$ is abelian, and $x \in G$ iff $x^{-1} \in G$ whence $\gamma(G)$ is an adjoint closed commuting set of operators.  It follows from the spectral theorem that there is an orthonormal basis$v_1,\dots,v_d$ simultaneously diagonalizing all of $\gamma(G)$.  Let
\begin{equation*}
\gamma_i:G \rightarrow S^1; x \mapsto \langle \gamma(x)v_i,v_i\rangle.
\end{equation*}
It is easy to see that all the $\gamma_i$ are well-defined homomorphisms and that
\begin{equation*}
\|\gamma(x)-I\|\leq \delta \textrm{ if and only if } |\gamma_i(x) - 1| \leq \delta \textrm{ for all }1 \leq i \leq d,
\end{equation*}
from which the result follows immediately on setting $\Gamma:=\{\gamma_1,\dots,\gamma_d\}$.
\end{proof}

In light of the above lemma we make the following definitions.  Suppose that $H$ is a $d$-dimensional Hilbert space and $\delta \in (0,2]$.  Then we write
\begin{equation*}
B(U(H),\delta):=\{M \in U(H): \|M - I \| \leq \delta\}
\end{equation*}
which is the usual $\delta$-ball around the identity.  Now, suppose that $G$ is a finite (not necessarily abelian) group and $\gamma:G \rightarrow U(H)$ is a homomorphism, then we write
\begin{equation*}
\Bohr(\gamma,\delta):=\{x \in G: \gamma(x) \in B(U(H),\delta)\}.
\end{equation*}

In the abelian setting there is a very useful pigeonhole argument which gives an estimate for the size of a Bohr set by pulling back an estimate for the size of balls in $(S^1)^d$.
\begin{lemma}[Size of abelian Bohr sets, {\cite[Lemma 4.19]{taovu::}}]
Suppose that $G$ is a finite abelian group, $\Gamma=\{\gamma_1,\dots,\gamma_d\}$ is a set of homomorphisms $G \rightarrow S^1$, and $\delta \in (0,2]$. Then
\begin{equation*}
\mu_G(\Bohr(\Gamma,\delta)) \geq \Omega(\delta)^d.
\end{equation*}
\end{lemma}
We should remark that technically the lemma in \cite{taovu::} is for the more common definition of Bohr set but it is easy to pass between the two by replacing $\delta$ with some quantity of size $\Omega(\delta)$.

The is an analogue of the previous lemma in the non-abelian setting.
\begin{lemma}[Size of non-abelian Bohr sets]\label{lem.bohrs}
Suppose that $G$ is a finite group, $H$ is a $d$-dimensional Hilbert space, $\gamma:G \rightarrow U(H)$ is a homomorphism and $\delta \in (0,2]$. Then
\begin{equation*}
\mu_G(\Bohr(\gamma,\delta)) \geq \Omega(\delta)^{d^2}.
\end{equation*}
\end{lemma}
This can be easily proved using the usual volume argument for unitary balls see, for example, the proof of \cite[Theorem 4.7]{gow::4}.  Write $\mu_{U(H)}$ for the unique left and right invariant probability measure on $U(H)$ -- consult \cite{wey::,hal::} or \cite{meh::} for a proof that such exists.
\begin{lemma}[Size of unitary balls]\label{lem.unitaryballsize}
Suppose that $H$ is a $d$-dimensional Hilbert space and $\delta \in (0,2]$. Then
\begin{equation*}
 \mu_{U(H)}(B(U(H),\delta)) \geq \Omega(\delta)^{d^2}.
\end{equation*}
\end{lemma}
We could prove Lemma \ref{lem.bohrs} directly now, but in fact we shall need the following more robust version which immediately yields the lemma as a corollary.  The proof method is the same as for \cite[Lemma 4.19]{taovu::}, namely a covering argument.
\begin{lemma}\label{lem.bohraverage}
Suppose that $G$ is a finite group, $H$ is a $d$-dimensional Hilbert space, $B \subset G$, $\phi:B \rightarrow U(H)$ is a map and $\delta \in (0,2]$ is a parameter. Then there is a subset $B' \subset B$ with $\mu_B(B') \geq \Omega(\delta)^{d^2}$ such that
\begin{equation*}
\|\phi(x)^{-1}\phi(x') - I\| \leq \delta \textrm{ for all } x,x' \in B'.
\end{equation*}
\end{lemma}
\begin{proof}
Consider the following average
\begin{equation*}
\int{\sum_{x \in B}{1_{B(U(H),\delta/2)N^{-1}}(\phi(x)^{-1})}d\mu_{U(H)}(N)}
\end{equation*}
which is equal to
\begin{equation*}
 \sum_{ x \in B}{\int{1_{\phi(x)B(U(H),\delta/2)}(N)d\mu_{U(H)}(N)}}
\end{equation*}
since integration is linear.  However the measure $\mu_{U(H)}$ is left invariant so 
\begin{eqnarray*}
\int{1_{\phi(x)B(U(H),\delta/2)}(N)d\mu_{U(H)}(N)}&=&\mu_{U(H)}(\phi(x)B(U(H),\delta/2)) \\&= &\mu_{U(H)}(B(U(H),\delta/2)) \geq \Omega(\delta)^{d^2}
\end{eqnarray*}
by Lemma \ref{lem.unitaryballsize}, whence
\begin{equation*}
\int{\sum_{x \in B}{1_{B(U(H),\delta/2).N^{-1}}(\phi(x)^{-1})}d\mu_{U(H)}(N)}\geq \Omega(\delta)^{d^2}|B|.
\end{equation*}
It follows by averaging that there is some $N \in U(H)$ such that
\begin{equation*}
\sum_{x \in B}{1_{B(U(H),\delta/2).N^{-1}}(\phi(x)^{-1})}\geq \Omega(\delta)^{d^2}|B|.
\end{equation*}
However, $B(U(H),\delta/2)^{-1}=B(U(H),\delta/2)$ whence
\begin{equation*}
\sum_{x \in B}{1_{NB(U(H),\delta/2)}(\phi(x))}\geq \Omega(\delta)^{d^2}|B|.
\end{equation*}
Now, letting $B':=\{x \in G: \phi(x) \in NB(U(H),\delta/2)\}$ and we see that
\begin{equation*}
|B'| \geq \Omega(\delta)^{d^2}|B|.
\end{equation*}
Finally if $x,x' \in B'$ then there are operators $M,M' \in B(U(H),\delta/2)$ such that $\phi(x)=NM^{-1}$ and $\phi(x') = NM'$.  (The asymmetry is possible since $B(U(H),\delta/2)$ is symmetric.)  Then
\begin{eqnarray*}
\|\phi(x)^{-1}\phi(x') - I\|& = &\|MM' - I\| \\ & \leq & \|(M-I)M'\| + \|M'-I\|\\ & = & \|M-I\| + \|M'-I\|\leq \delta
\end{eqnarray*}
by the triangle inequality and unitarity of $M'$.  The result follows.
\end{proof}

\section{From large multiplicative energy to correlation with a multiplicative pair}\label{sec.bog}

In this section we shall prove a result which lets us pass from large multiplicative energy to correlation with a multiplicative pair.  This can be seen as a sort of weak asymmetric non-abelian Bogolio{\`u}boff theorem relative to multiplicative pairs (\emph{c.f.} \cite{bog::}).  
\begin{proposition}\label{prop.mpls}
Suppose that $G$ is a finite group, $B_0,B_1,B_2,B_3$ are sets such that $\mathcal{B}_{i,j}=(B_i,B_j)$ is a $c_j$-thick, $\epsilon_j$-closed $r_j$-multiplicative pair for each $i<j$, $f \in L^1(\mu_{B_2})$ and $g \in L^2(\mu_{B_1})$, not identically zero, are such that
\begin{equation*}
\|(fd\mu_{B_2})\ast g\|_{L^2(\mu_{B_1})}^2 \geq \nu \|g\|_{L^2(\mu_{B_1})}^2\|f\|_{L^\infty(\mu_{B_2})}^2,
\end{equation*}
and $\eta \in (0,1]$ is a parameter. Then there is an absolute constant $C_{\rm{Bog}}>0$ such that if 
\begin{equation*}
r_3 \geq 32, \epsilon_3 \leq \left(\frac{c_2\nu}{2}\right)^{C_{\rm{Bog}}} \textrm{ and } \epsilon_3 \leq \frac{\|g\|_{L^2(\mu_{B_1})}^4\nu^2}{256\|g\|_{L^\infty(\mu_{B_1})}^4},
\end{equation*}
then there is a positive real $c=\Omega_{\eta,\nu,c_1,c_2,c_3}(\mu_G(B_1))$ and some $c$-thick, $\eta$-closed and $4$-multiplicative pairs $\mathcal{B}_{9,10}=(B_9,B_{10})$ and $\mathcal{B}_{10,11}=(B_{10},B_{11})$ such that $B_9^2 \subset B_3^3$ and 
\begin{equation*}
\sup_{x \in B_{2,3}^-}{|f \ast \widetilde{\mu_{B_9}}\ast \mu_{B_9}(x)|} =\Omega(\sqrt{\nu}\|f\|_{L^\infty(\mu_{B_2})}).
\end{equation*}
\end{proposition}
It may be useful at a first reading to think of the special case of $B_0=B_1=B_2=B_3=G$, when we see that $f$ (which may be signed) correlates with a product set with small doubling. 

Our argument is inspired by Bogolio{\`u}boff's result (popularised by Ruzsa \cite{ruz::9}) although the details are rather different.  The energy hypothesis implies that the large spectrum supports a large chunk of the mass of $f$, and then combine the work in \S\ref{sec.anallarge}\verb!&!\ref{sec.bohrandunball} to find a suitable multiplicative pair to project onto, leading to the correlation.

The first result draws draws together the work of \S\ref{sec.anallarge} -- it may be worth recalling the definition of regular for $(\mathcal{B},f)$ from that section -- and \S\ref{sec.bohrandunball}.
\begin{lemma}\label{lem.level}
Suppose that $G$ is a finite group, $\mathcal{B}=(B,B')$ is a $c$-thick multiplicative pair, $f \in L^1(\mu_{B'})$ is not identically zero, $\delta$ is regular for $(\mathcal{B},f)$, and $\eta\in (0,1]$ is a parameter. Then there is an absolute constant $C>0$ such that if 
\begin{equation*}
\epsilon' \leq (\eta\delta c w(f)/2)^{C}
\end{equation*}
and $\mathcal{B}'=(B,B'')$ is a $c'$-thick, $\epsilon'$-closed $32$-multiplicative pair, then there is a symmetric neighbourhood of the identity $B'''$ with
\begin{equation*}
B'''^{16} \subset B''^4 \textrm{ and } \mu_{G}(B''') = \Omega_{\eta,\delta,c,c',w(f)}(\mu_G(B)),
\end{equation*}
and such that for any probability measure $\mu$ with $\supp \mu \subset B'''^{16}$ we have
\begin{equation*}
\|(v \ast \mu)|_B -v\|_{L^2(\mu_B)}^2 \leq \eta^2\|v\|_{L^2(\mu_B)}^2
 \end{equation*}
 for all $v \in \Spec_\delta(\mathcal{B},f)$.
\end{lemma}
\begin{proof}
Let $C_{\mathcal{A}\mathcal{H}}$ be an absolute constant such that
\begin{equation*}
\|T_{\mathcal{B},f,\delta,yz}- T_{\mathcal{B},f,\delta,y}T_{\mathcal{B},f,\delta,z}\|^2 \leq \epsilon'(2\delta^{-1}c^{-1}w(f)^{-1})^{C_{\mathcal{A}\mathcal{H}}}
\end{equation*}
holds in the conclusion of Lemma \ref{lem.apxhom}, and similarly $C'$ be a constant such that
\begin{equation*}
\|T_{\mathcal{B},f,\delta,y} - U_y\|^2 \leq \epsilon'(2\delta^{-1}c^{-1}w(f)^{-1})^{C'}
\end{equation*}
holds in the conclusion of Lemma \ref{lem.un}.  Put
\begin{equation*}
C:=8+\max\{C_{\mathcal{S}\mathcal{R}},C_{\mathcal{A}\mathcal{H}},C'\}.
\end{equation*}

Write $d$ for $\dim \Spec_\delta(\mathcal{B},f)$ and recall from the Parseval bound that
\begin{equation*}
d \leq c^{-1}\delta^{-2}w(f)^{-1}
\end{equation*}
since $\|f\|_{L^2(\mu_{B'})}^2 \leq\|f\|_{L^1(\mu_{B'})}\|f\|_{L^\infty(\mu_{B'})}$.

Consider the map $B'' \rightarrow U(\Spec_\delta(\mathcal{B},f))$ such that $y \mapsto U_y$, given by Lemma \ref{lem.un}.  By Lemma \ref{lem.bohraverage} there is a set $B_1 \subset B''$ with
\begin{equation*}
\mu_{B''}(B_1) = \Omega(\eta)^{d^2}=\Omega_{\eta,\delta,c,w(f)}(1),
\end{equation*}
such that
\begin{equation}\label{eqn.early}
\|U_{y}^{-1}U_z - I\| \leq \eta/256 \textrm{ for all } y,z \in B_1.
\end{equation}
Given the size of $\epsilon'$, Lemma \ref{lem.apxhom} tells us that if $y,z,yz \in B''^{32}$ then
\begin{equation*}
\|T_{\mathcal{B},f,\delta,yz} -T_{\mathcal{B},f,\delta,y}T_{\mathcal{B},f,\delta,z}\| \leq \eta/256.
\end{equation*}
Moreover,
\begin{equation*}
\|T_{\mathcal{B},f,\delta,z}\| \leq \| \pi_{\Spec_\delta(\mathcal{B},f)}\|\|\rho_{\mathcal{B}',y}\| \leq 1
\end{equation*}
whenever $z \in B''^{32}$ by Lemma \ref{lem.apun} and the fact that projections have operator norm at most $1$. Now, if $u \in (B_1^{-1}B_1)^{16}$ then there are elements $y_1,\dots,y_{16}, z_1,\dots,z_{16} \in B_1$ such that
\begin{equation*}
u=y_1^{-1}z_1\dots y_{16}^{-1}z_{16}.
\end{equation*}
Combining the preceding bounds on the operator norm using the triangle inequality (by the telescoping sum method) we get that
\begin{equation*}
\|T_{\mathcal{B},f,\delta,u}-T_{\mathcal{B},f,\delta,y_1^{-1}}T_{\mathcal{B},f,\delta,z_1}\dots T_{\mathcal{B},f,\delta,y_{16}^{-1}}T_{\mathcal{B},f,\delta,z_{16}}\| \leq \eta/8.
\end{equation*}
On the other hand by Lemma \ref{lem.un} we have
\begin{equation}\label{eqn.if}
\|T_{\mathcal{B},f,\delta,z} -U_{z}\| \leq \eta/256
\end{equation}
for all $z \in B_1$.  Again by the triangle inequality we get that
\begin{equation*}
\|T_{\mathcal{B},f,\delta,u}-U_{y_1^{-1}}U_{z_1}\dots U_{y_{16}^{-1}}U_{z_{16}}\|\leq \eta/4
\end{equation*}
Now, by Lemma \ref{lem.apxhom} we have
\begin{equation*}
\|T_{\mathcal{B},f,\delta,y^{-1}} T_{\mathcal{B},f,\delta,y} -I\| \leq \eta/256
\end{equation*}
for all $y \in B_1$. Combining this with (\ref{eqn.if}) (and using the fact that the operators $U_y$ are unitary) we get that
\begin{equation*}
\|U_y^{-1} -U_{y^{-1}}\| \leq 3\eta/256
\end{equation*}
for all $y \in B_1$.  Hence the triangle inequality again gives
\begin{equation*}
\|T_{\mathcal{B},f,\delta,u}-U_{y_1}^{-1}U_{z_1}\dots U_{y_{16}}^{-1}U_{z_{16}}\|\leq 7\eta/16
\end{equation*}
On the other hand we may now use (\ref{eqn.early}) coupled with the triangle inequality to get that
\begin{equation*}
\|T_{\mathcal{B},f,\delta,u}-I\| \leq \eta/2.
\end{equation*}
Now, since $\delta$ is regular and $\epsilon'$ is small by design, we have
\begin{equation*}
\|\rho_{\mathcal{B}',u}v-T_{\mathcal{B},f,\delta,u}v\|_{L^2(\mu_B)} \leq \eta/2,
\end{equation*}
for all unit vectors $v \in \Spec_\delta(\mathcal{B},f)$, whence
\begin{equation*}
\|\rho_{\mathcal{B}',u}v-v\|_{L^2(\mu_B)}\leq \eta
\end{equation*}
for all unit vectors $v \in \Spec_\delta(\mathcal{B},f)$.

It remains to put $B''':=B_1^{-1}B_1$ and note that if $\mu$ is a probability measure with $\supp \mu \subset B'''^{16}$ then
\begin{eqnarray*}
\|(v \ast \mu_{B''''})|_B-v\|_{L^2(\mu_B)}^2& =& \|\int{\rho_{u^{-1}}(v)d\mu_{B'''}(u)}-v\|_{L^2(\mu_B)}^2\\ & \leq & \int{\|\rho_{u^{-1}}(v) -v\|_{L^2(\mu_B)}^2d\mu(u)}\leq \eta^2
\end{eqnarray*}
since $u \in B'''^{16}$ if and only if $u \in B'''^{-16}$.  The result is proved.
\end{proof}
Now we are in a position to prove the main result of this section.
\begin{proof}[Proof of Proposition \ref{prop.mpls}]
First we note that
\begin{eqnarray*}
\|(fd\mu_{B_2}) \ast g \|_{L^2(\mu_{B_1})}^2&\leq& \frac{\mu_G({B_1})}{\mu_G(B_2)^2}\|f \ast g\|_{L^2(\mu_G)}^2 \\ &\leq &\frac{\mu_G({B_1})}{\mu_G(B_2)^2}\|f\|_{L^1(\mu_G)}^2\|g\|_{L^2(\mu_G)}^2,
\end{eqnarray*}
by positivity and Young's inequality.  It follows that
\begin{equation*}
\|(fd\mu_{B_2}) \ast g \|_{L^2(\mu_{B_1})}^2 \leq \|f\|_{L^1(\mu_{B_2})}^2\|g\|_{L^2(\mu_{B_1})}^2,
\end{equation*}
and so given the lower bound of
\begin{equation*}
\|(fd\mu_{B_2}) \ast g \|_{L^2(\mu_{B_1})}^2 \geq \nu \|f\|_{L^\infty(\mu_{B_2})}^2\|g\|_{L^2(\mu_{B_1})}^2,
\end{equation*}
we conclude that $w(f)\geq \sqrt{\nu}$ since $g$ is not identically zero.

Let $v_1,\dots,v_n$ be a Fourier basis of $L^2(\mu_{B_1})$ for $f$ as provided by Proposition \ref{prop.diag}.  We have that
\begin{equation*}
\|(fd\mu_{B_2}) \ast g \|_{L^2(\mu_{B_1})}^2 = \sum_{i=1}^n{|s_i(\mathcal{B}_{1,2},f)|^2|\langle g,v_i\rangle_{L^2(\mu_{B_1})}|^2}.
\end{equation*}
However, the left hand side is at least $\nu \|f\|_{L^\infty(\mu_{B_2})}^2\|g\|_{L^2(\mu_{B_1})}^2$ and, of course,
\begin{equation*}
\|g\|_{L^2(\mu_{B_1})}^2 = \sum_{i=1}^n{|\langle g,v_i\rangle_{L^2(\mu_{B_1})}|^2}.
\end{equation*} 
It follows from this and the triangle inequality that 
\begin{equation}\label{eqn.slp}
\sum_{i: v_i \in \Spec_\delta(\mathcal{B}_{1,2},f)}{|s_i(\mathcal{B}_{1,2},f)|^2|\langle g,v_i\rangle_{L^2(\mu_{B_1})}|^2} \geq \|(fd\mu_{B_2}) \ast g \|_{L^2(\mu_{B_1})}^2/2
\end{equation}
for any $\delta \leq \sqrt{\nu}/2$. Pick a $\delta\in (\sqrt{\nu}/4,\sqrt{\nu}/2]$ regular for $(\mathcal{B}_{1,2},f)$ (possible by Lemma \ref{lem.nr}). Note by the Parseval bound (Lemma \ref{lem.pbd}) that
\begin{equation}\label{eqn.bdp}
\dim \Spec_\delta(\mathcal{B}_{1,2},f)\leq c_2^{-1}\delta^{-2}\|f\|_{L^1(\mu_{B_2})}^{-2}\|f\|_{L^2(\mu_{B_2})}^2 \leq 16c_2^{-1}\nu^{-3/2}
\end{equation}
since
\begin{equation*}
\|f\|_{L^1(\mu_{B_2})}^{-2}\|f\|_{L^2(\mu_{B_2})}^2 \leq w(f)^{-1} \leq \nu^{-1/2}.
\end{equation*}
Apply Lemma \ref{lem.level} to $f$ and $\mathcal{B}_{1,2}$, with parameter $\eta':=\nu^{5/2}c_2/64$ (this determines the necessary value of $C_{\rm{Bog}}$ and entails the requirement that $r_3 \geq 32$) to get $B_4$, a symmetric neighbourhood of the identity with
\begin{equation*}
B_4^{16} \subset B_3^4\textrm{ and }\mu_G(B_4) = \Omega_{\nu,c_2,c_3}(\mu_G(B_1))
\end{equation*}
and such that for any probability measure $\mu$ with $\supp \mu \subset B_4^{16}$ we have
\begin{equation}\label{eqn.ooop}
\|(v \ast \mu)|_{B_1} - v \|_{L^2(\mu_{B_1})}^2 \leq \eta' \|v\|_{L^2(\mu_G)}^2
\end{equation}
for all $v \in \Spec_\delta(\mathcal{B},f)$.  We should like to apply Corollary \ref{cor.apnorm} to the sets $B_0,B_1,B_4$; we can on noting that the pair $(B_1,B_4)$ is certainly an $\Omega_{\nu,c_2,c_3}(1)$-thick $1$-closed $1$-multiplicative pair since $r_3 \geq 4$ and $\epsilon_3 \leq 1$. It follows there is a symmetric neighbourhood of the identity $B_5$ such that
\begin{equation*}
\mu_G(B_5) = \Omega_{\nu,c_1,c_2,c_3}(\mu_G(B_4)) =\Omega_{\nu,c_1,c_2,c_3}(\mu_G(B_1))
\end{equation*}
and
\begin{equation*}
xB_5x^{-1} \subset B_4^6 \textrm{ for all } x \in B_1.
\end{equation*}
Specifically $B_5 \subset B_4^6 \subset B_3^{24} \subset B_1$ whence (from the lower bound on the size of $B_5$) it has doubling $O_{\nu,c_1,c_2,c_3}(1)$.  We now apply Proposition \ref{prop.regapp} to get a positive real $c=O_{\eta,\nu,c_1,c_2,c_3}(1)$ and sets $B_6,B_7,B_8$ such that $B_6 \subset B_5^4$ and $\mathcal{B}_{6,7}$ and $\mathcal{B}_{7,8}$ are $c$-thick, $\eta$-closed $4$-multiplicative pairs and
\begin{equation*}
B_6 \subset B_5^4 \textrm{ and } \mu_G(B_6) = \Omega_{\nu,c_1,c_2,c_3}(\mu_G(B_1)).
\end{equation*}
In view of this
\begin{equation*}
xB_6^2x^{-1} \subset xB_5^8x^{-1} \subset B_4^{48} \subset B_3^{3}\subset B_2 \textrm{ for all } x\in B_1
\end{equation*}
since $r_3\geq 2$.  Since $B_6 \subset B_4^{8}$ we may leverage (\ref{eqn.ooop}) as follows:
\begin{equation*}
|\langle g,v_i \ast \widetilde{\mu_{B_6}} \ast \mu_{B_6}\rangle_{L^2(\mu_{B_1})}-\langle g,v_i\rangle_{L^2(\mu_{B_1})}|\leq \|g\|_{L^2(\mu_{B_1})}\eta',\end{equation*}
and hence
\begin{equation*}
||\langle g,v_i \ast \widetilde{\mu_{B_6}} \ast \mu_{B_6}\rangle_{L^2(\mu_{B_1})}|^2-|\langle g,v_i\rangle_{L^2(\mu_{B_1})}|^2|\leq 2\|g\|_{L^2(\mu_{B_1})}^2\eta'
\end{equation*}
by the triangle inequality.  Inserting this in (\ref{eqn.slp}) and using the bound (\ref{eqn.bdp}) and the definition of $\eta'$, we get that
\begin{equation*}
\sum_{i: v_i \in \Spec_\delta(\mathcal{B}_{1,2},f)}{|s_i(\mathcal{B}_{1,2},f)|^2|\langle g,v_i\ast\widetilde{\mu_{B_6}} \ast \mu_{B_6}\rangle_{L^2(\mu_{B_1})}|^2} \geq \|(fd\mu_{B_2}) \ast g \|_{L^2(\mu_{B_1})}^2/4.
\end{equation*}
This rearranges to give
\begin{equation*}
\sum_{i: v_i \in \Spec_\delta(\mathcal{B}_{1,2},f)}{|s_i(\mathcal{B}_{1,2},f)|^2|\langle g\ast \widetilde{\mu_{B_6}} \ast \mu_{B_6},v_i\rangle_{L^2(\mu_{B_1})}|^2} \geq \|(fd\mu_{B_2}) \ast g \|_{L^2(\mu_{B_1})}^2/4.
\end{equation*}
By positivity and the definition of the basis $(v_i)_{i=1}^n$ we conclude that
\begin{equation*}
\|(fd\mu_{B_2}) \ast ((g \ast \widetilde{\mu_{B_6}} \ast \mu_{B_6})|_{B_1})\|_{L^2(\mu_{B_1})}^2\geq \|(fd\mu_{B_2}) \ast g \|_{L^2(\mu_{B_1})}^2/4.
\end{equation*}
Now $\supp g \ast \widetilde{\mu_{B_6}} \ast \mu_{B_6}\subset B_1B_3^4 \subset B_1B_2$ since $r_3 \geq 4$ whence by Lemma \ref{lem.bogcalc} we have that
\begin{equation*}
|\|(fd\mu_{B_2}) \ast ((g \ast \widetilde{\mu_{B_6}} \ast \mu_{B_6})|_{B_1})\|_{L^2(\mu_{B_1})}^2-\|(fd\mu_{B_2}) \ast (g \ast \widetilde{\mu_{B_6}} \ast \mu_{B_6})\|_{L^2(\mu_{B_1})}^2|
\end{equation*}
is at most
\begin{equation*}
\|(fd\mu_{B_2}) \ast g \|_{L^2(\mu_{B_1})}^2/8
\end{equation*}
in view of the second upper bound on $\epsilon_2$.  We conclude that
\begin{equation*}
\|(fd\mu_{B_2}) \ast (g \ast \widetilde{\mu_{B_6}} \ast \mu_{B_6})\|_{L^2(\mu_{G})}^2\geq \|(fd\mu_{B_2}) \ast g \|_{L^2(\mu_{B_1})}^2/8
\end{equation*}
and it remains to apply Lemma \ref{lem.apnormcalc} with the sets $B_1,B_2,B_3,B_6^2$ which can be done since $B_6^2\subset B_3^{3}$ so $(B_2,B_6^2)$ is a $1$-closed $\epsilon_3$-multiplicative pair and
\begin{eqnarray*}
\epsilon_3 & \leq & \sqrt{\nu}\|g\|_{L^2(\mu_{B_1})}/2^5\|g\|_{L^\infty(\mu_{B_1})}\\ & \leq & \|(fd\mu_{B_2}) \ast (g \ast \widetilde{\mu_{B_6}} \ast \mu_{B_6})\|_{L^2(\mu_{G})}/2\sqrt{3}\|f\|_{L^\infty(\mu_{B_2})}\|g\|_{L^\infty(\mu_{B_1})}.
\end{eqnarray*}
Doing this tells us that
\begin{equation*}
\sup_{y \in {B_1}}{\sup_{x \in B_{2,3}^-y}{|\rho_{y^{-1}}(f) \ast \widetilde{\mu_{B_6}} \ast \mu_{B_6}(x)|^2}} =\Omega(\nu \|f\|_{L^\infty(\mu_G)}^2).
\end{equation*}
It remains to pick $y \in B_1$ and $x \in B_{2,3}^-$ such that the supremum is attained and set $B_9:=yB_6y^{-1}$, $B_{10}:=yB_7y^{-1}$ and $B_{11}:=yB_8y^{-1}$ and we have our multiplicative pairs.  Now, note that
\begin{equation*}
\rho_{y^{-1}}(f) \ast \widetilde{\mu_{B_6}} \ast \mu_{B_6}(x) = f \ast \widetilde{\mu_{B_9}} \ast \mu_{B_9}(xy^{-1}),
\end{equation*}
and the result is proved.
\end{proof}
It is fairly easy to see that the bound on $c^{-1}$ is a bounded tower of exponentials in $\eta^{-1},\nu^{-1},c_1^{-1},c_2^{-1}$ and $c_3^{-1}$.  With more effort it can be pinned down more precisely.

\section{The spectrum of multiplicative pairs}\label{sec.specmp}

As well as having good behaviour in physical space, we should also like multiplicative pairs to have good spectral behaviour.  There are various results of this flavour in the abelian setting (\emph{e.g.} \cite[Lemma 3.6]{grekon::}) which characterise the characters at which $\wh{\mu_B}$ is large.  In the non-abelian setting we are given a basis to work with respect to, and this does not necessarily diagonalise the operator $L_{\mu_B}^*L_{\mu_B}$.

We have the following lemma which is fit for purpose.  It shows how correlation with a multiplicative pair corresponds to spectral mass in the dual object, in this case the set of basis vectors which are large under convolution with the ground set of the multiplicative pair.  The result can be used without loss in place of the usual abelian arguments for collecting spectral mass.
\begin{lemma}\label{lem.spectralcollection}
Suppose that $G$ is a finite group, $f \in A(G)$ has $\|f\|_{A(G)} \leq M$, and $\mathcal{B}=(B,B')$ is an $\epsilon$-closed $1$-multiplicative pair, $v_1,\dots,v_N$ is a Fourier basis of $L^2(\mu_G)$ for $f$ and 
\begin{equation*}
\|f \ast(\widetilde{\mu_B}\ast \mu_B) - f \ast (\widetilde{\mu_{B'}}\ast \mu_{B'}) \|_{L^\infty(\mu_G)} \geq \nu.
\end{equation*}
Then
\begin{equation*}
\sum_{i=1}^N{|s_i(f)|\|\mu_{B'} \ast v_i\|_{L^2(\mu_G)}^2} \geq \sum_{i=1}^N{|s_i(f)|\|\mu_{B} \ast v_i\|_{L^2(\mu_G)}^2} +\nu^2/M-4\epsilon M.
\end{equation*}
\end{lemma}
\begin{proof}
Let $x' \in G$ be such that the $L^\infty(\mu_G)$-norm is attained, \emph{i.e.} such that
\begin{equation*}
|(f \ast(\widetilde{\mu_B}\ast \mu_B) - f \ast (\widetilde{\mu_{B'}}\ast \mu_{B'}))(x')|
\end{equation*}
is maximal, and note that the term inside the mod signs is equal to
\begin{equation}\label{eqn.rui}
g(x'):=L_f(\widetilde{\mu_B}\ast \mu_B - \widetilde{\mu_{B'}}\ast \mu_{B'})(x').
\end{equation}
We now recall the proof of Lemma \ref{lem.linfag}.  As usual since $v_1,\dots,v_N$ is a Fourier basis of $L^2(\mu_G)$ for $f$, so is $\rho_y(v_1),\dots,\rho_y(v_N)$ for all $y \in G$.  Thus we may write
\begin{equation*}
g(x') = \sum_{i=1}^N{\langle \widetilde{\mu_B}\ast \mu_B - \widetilde{\mu_{B'}}\ast \mu_{B'},\rho_yv_i\rangle_{L^2(\mu_G)}L_f\rho_yv_i(x')}.
\end{equation*}
On the other hand $L_f\rho_yv_i(x') = L_fv_i(x'y)$ since left convolution commutes with right translation, whence
\begin{equation*}
g(x')= \sum_{i=1}^N{\langle \widetilde{\mu_B}\ast \mu_B - \widetilde{\mu_{B'}}\ast \mu_{B'},\rho_yv_i\rangle_{L^2(\mu_G)}L_fv_i(x'y)}.
\end{equation*}
However,
\begin{equation*}
\langle \widetilde{\mu_B}\ast \mu_B - \widetilde{\mu_{B'}}\ast \mu_{B'},\rho_yv_i\rangle_{L^2(\mu_G)}= \widetilde{( \widetilde{\mu_B}\ast \mu_B - \widetilde{\mu_{B'}}\ast \mu_{B'})} \ast v_i(y)
\end{equation*}
Of course the first term is self-adjoint, whence
\begin{equation*}
|g(x')| \leq \sum_{i=1}^N{| (\widetilde{\mu_B}\ast \mu_B - \widetilde{\mu_{B'}}\ast \mu_{B'}) \ast v_i(y)||L_fv_i(x'y)|}.
\end{equation*}
Now integrate $y$ against $\mu_G$ and apply the Cauchy-Schwarz inequality term-wise so that
\begin{equation*}
\nu \leq |g(x')| \leq  \sum_{i=1}^N{|s_i(f)|\|(\widetilde{\mu_B}\ast \mu_B - \widetilde{\mu_{B'}}\ast \mu_{B'}) \ast v_i\|_{L^2(\mu_G)}},
\end{equation*}
since $\|L_fv_i\|_{L^2(\mu_G)}=|s_i(f)|$ for all $i \in \{1,\dots,N\}$.  Finally we apply Cauchy-Schwarz to this to get that
\begin{equation*}
\nu^2 \leq \left(\sum_{i=1}^N{|s_i(f)|\|(\widetilde{\mu_B}\ast \mu_B - \widetilde{\mu_{B'}}\ast \mu_{B'}) \ast v_i\|_{L^2(\mu_G)}^2}\right)\left(\sum_{i=1}^N{|s_i(f)|}\right),
\end{equation*}
which rearranges by the explicit formula for $A(G)$ to give
\begin{equation}\label{eqn.kuh}
\nu^2/M \leq\sum_{i=1}^N{|s_i(f)|\|(\widetilde{\mu_B}\ast \mu_B - \widetilde{\mu_{B'}}\ast \mu_{B'}) \ast v_i\|_{L^2(\mu_G)}^2}.
\end{equation}
To estimate the summands on the right we expand them:
\begin{equation}\label{eqn.weight}
\|(\widetilde{\mu_B}\ast \mu_B - \widetilde{\mu_{B'}}\ast \mu_{B'}) \ast v_i\|_{L^2(\mu_G)}^2
\end{equation}
is equal to
\begin{eqnarray*}
\|\widetilde{\mu_B}\ast \mu_B\ast v_i\|_{L^2(\mu_G)}^2& +& \|\widetilde{\mu_{B'}}\ast \mu_{B'} \ast v_i\|_{L^2(\mu_G)}^2\\ &-&2\Re \langle \widetilde{\mu_B}\ast \mu_B\ast v_i,\widetilde{\mu_{B'}}\ast \mu_{B'} \ast v_i\rangle_{L^2(\mu_G)}.
\end{eqnarray*}
The first two terms can be simplified by Young's inequality so that
\begin{equation*}
\|\widetilde{\mu_B}\ast \mu_B\ast v_i\|_{L^2(\mu_G)}^2\leq\|\mu_B\ast v_i\|_{L^2(\mu_G)}^2
\end{equation*}
and
\begin{equation*}
\|\widetilde{\mu_{B'}}\ast \mu_{B'}\ast v_i\|_{L^2(\mu_G)}^2\leq\|\mu_{B'}\ast v_i\|_{L^2(\mu_G)}^2.
\end{equation*}
The inner product is dealt with slightly differently: recall that $\mathcal{B}$ is an $\epsilon$-closed and $1$-multiplicative pair so
\begin{equation*}
| \langle \widetilde{\mu_B}\ast \mu_B\ast v_i,\widetilde{\mu_{B'}}\ast \mu_{B'} \ast v_i\rangle_{L^2(\mu_G)}- \langle \widetilde{\mu_B}\ast \mu_B\ast v_i,v_i\rangle_{L^2(\mu_G)}|
\end{equation*}
is at most
\begin{equation*}
|\langle\widetilde{\mu_{B'}}\ast \mu_{B'} \ast \widetilde{\mu_B}\ast \mu_B- \widetilde{\mu_B}\ast \mu_B,v_i \ast \widetilde{v_i}\rangle_{L^2(\mu_G)}| \leq 2\epsilon
\end{equation*}
by Lemma \ref{lem.approxhaar} since $\|v_i \ast \widetilde{v_i}\|_{L^\infty(\mu_G)} \leq \|v_i\|_{L^2(\mu_G)}^2$ by Young's inequality.  It follows that (\ref{eqn.weight}) is at most
\begin{equation*}
\|\mu_B\ast v_i\|_{L^2(\mu_G)}^2 + \|\mu_{B'} \ast v_i\|_{L^2(\mu_G)}^2 -2\Re \langle \widetilde{\mu_B}\ast \mu_B\ast v_i,v_i\rangle_{L^2(\mu_G)}+4\epsilon,
\end{equation*}
which in turn is equal to
\begin{equation*}
\|\mu_{B'} \ast v_i\|_{L^2(\mu_G)}^2-\|\mu_B\ast v_i\|_{L^2(\mu_G)}^2  +4\epsilon.
\end{equation*}
Inserting this into (\ref{eqn.kuh}) we conclude that
\begin{equation*}
\nu^2/M \leq \sum_{i=1}^N{|s_i(f)|(\|\mu_{B'} \ast v_i\|_{L^2(\mu_G)}^2-\|\mu_B\ast v_i\|_{L^2(\mu_G)}^2)} + 4\epsilon M
\end{equation*}
by the explicit formula for $A(G)$. The result follows.
\end{proof}

\section{Quantitative continuity of functions in $A(G)$}\label{sec.qc}

We showed in \S\ref{sec.algnormprop} that the $L^\infty(\mu_G)$ norm is dominated by the $A(G)$-norm and, indeed, it is relatively easy to show in the infinitary setting that if $f\in A(G)$ then $f=g$ almost everywhere for some continuous function $g$.  In this section we make this notion quantitative.  

Our main result is the following.
\begin{proposition}[{Quantitative continuity in $A(G)$}]\label{prop.quantcon}
Suppose that $G$ is a finite group, $f \in L^1(\mu_G)$ has $\|f\|_{A(G)} \leq M$, $A$ is symmetric and $\mu_G(A^4) \leq K\mu_G(A)$ and $\nu \in (0,1]$ is a parameter. Then there are symmetric neighbourhoods of the identity $B' \subset B \subset A^4$ such that $\mu_G(B')=\Omega_{K,\nu,M}(\mu_G(A))$,
\begin{equation*}
\sup_{x \in G}{\|f\ast\widetilde{\mu_B}\ast  \mu_B-f\ast\widetilde{\mu_B}\ast  \mu_B(x)\|_{L^\infty(\mu_{xB'})}}\leq \nu
\end{equation*}
and
\begin{equation*}
\sup_{x \in G}{\|f - f\ast\widetilde{\mu_B}\ast  \mu_B\|_{L^2(\mu_{xB'})}}\leq \nu.
\end{equation*}
\end{proposition}
There is an analogous result in \cite[Proposition 5.1]{gresan::0}, and the inspiration for that proof came, in turn, from the idea of relativizing the main argument in \cite{grekon::}.  In this paper the argument is rather different because we have a weaker structure to which we need to relativize and the non-abelian Fourier transform is not equal to the task.  

The proposition will be proved by iterating the next result which is a dichotomy between good average behaviour and correlation with a structured sub-object -- a type of dichotomy frequently found in additive combinatorics.
\begin{proposition}[{Proposition \ref{prop.disc2}}]\label{prop.disclarge}
Suppose that $G$ is a finite group, $f \in A(G)$ has $\|f\|_{A(G)} \leq M$, $A$ is symmetric and $\mu_G(A^4) \leq K\mu_G(A)$ and $\nu,\eta \in (0,1]$ are parameters.  Then either
\begin{enumerate}
\item there are symmetric neighbourhoods of the identity $B' \subset B \subset A^4$ such that $\mu_G(B')=\Omega_{K,\nu,\eta,M}(\mu_G(A))$,
\begin{equation*}
\sup_{x \in G}{\|f\ast\widetilde{\mu_B}\ast  \mu_B-f\ast\widetilde{\mu_B}\ast  \mu_B(x)\|_{L^\infty(\mu_{xB'})}}\leq \nu
\end{equation*}
and
\begin{equation*}
\sup_{x \in G}{\|f - f\ast\widetilde{\mu_B}\ast  \mu_B\|_{L^2(\mu_{xB'})}}\leq \nu.
\end{equation*}
\item or there are symmetric neighbourhoods of the identity $B''\subset B' \subset B \subset A^4$ such that  $\mu_G(B'')=\Omega_{K,\nu,\eta,M}(\mu_G(A))$, $\mathcal{B}':=(B,B')$ and $\mathcal{B}'':=(B',B'')$ are $\eta$-closed $4$-multiplicative pairs and
\begin{equation*}
\|f \ast \widetilde{\mu_B}\ast  \mu_B - f \ast \widetilde{\mu_{B'}}\ast  \mu_{B'}\|_{L^\infty(\mu_G)}= \Omega_{\nu,M}(1).
\end{equation*}
\begin{equation*}
\end{equation*}
\end{enumerate}
\end{proposition}
The main meat of the argument is the proof of the above proposition.  We include the proof of the reduction to this statement now, setting the stage for the work of the next section where we establish the above.

The idea of the proof is to use Proposition \ref{prop.disclarge} to repeatedly provide an increase in spectral mass.  This process must terminate since the algebra norm is bounded, so we are not always in the second case of the proposition; in the first case we have our desired continuity conclusion.
\begin{proof}[Proof of Proposition \ref{prop.quantcon}]  
Let $(v_i)_{i=1}^N$ be a Fourier basis of $L^2(\mu_G)$ for $f$.  Write $F(\nu,M)$ for the function of $\nu$ and $M$ implicit in the second conclusion of Proposition \ref{prop.disclarge} -- it is apparent from the nature of the proposition that this may be taken to be monotonely decreasing in $M$ and monotonely increasing in $\nu$.

We construct three sequences of sets $(B_i)_{i\geq 1}$, $(B_i')_{i\geq 1}$ and $(B_i'')_{i\geq 1}$ iteratively and write
\begin{equation*}
\mu_i:=\sum_{j=1}^N{|s_j(f)|\|\mu_{B_i'}\ast v_j\|_{L^2(\mu_G)}^2}.
\end{equation*}
At this point we remind the reader of Lemma \ref{lem.spectralcollection} to give an idea of how $\mu_i$ is to be controlled.

We shall arrange the sets such that the following properties hold for $\eta:=F(\nu,M)^2/32M^2$.
\begin{enumerate}
\item $B_i \subset B_{i-1}''^4$ and $B_i \subset A^4$;
\item $(B_i,B_{i}')$ is an $\eta$-closed $4$-multiplicative pair;
\item $(B_i',B_{i}'')$ is an $\eta$-closed $4$-multiplicative pair;
\item $\mu_i \geq \mu_{i-1}+F(\nu,M)^2/2M$;
\item $\mu_G(B_i),\mu_G(B_i'),\mu_G(B_i'') = \Omega_{\nu,M,K,i}(\mu_G(A))$.
\end{enumerate}
To begin the construction apply Proposition \ref{prop.regapp} to get a constant $c_0=\Omega_{K,\eta}(1)$ and sets $B_0'',B_0',B_0$ such that $(B_0,B_0')$ and $(B_0',B_0'')$ is are $c_0$-thick $\eta$-closed $4$-multiplicative pairs,
\begin{equation*}
B_0 \subset A^4 \textrm{ and } \mu_G(B_0)= \Omega_{K,\eta}(\mu_G(A)).
\end{equation*}
The iteration is thus initialised in light of the definition of $\eta$.  Now, suppose that we have defined $B_i,B_i'$ and $B_i''$ as per the above. By the lower bounds on the density of $B_i''$ and the fact that $B_i''^4 \subset B_i \subset A^4$ we have that
\begin{equation*}
\mu_G(B_i''^4) =O_{\nu,M,K,i}(\mu_G(B_i'')).
\end{equation*}
We apply Proposition \ref{prop.disclarge} to the set $B_i$ and the function $f$.  If we are in the first case of the proposition terminate; otherwise we get $\eta$-closed $4$-multiplicative pairs $(B_{i+1},B_{i+1}')$ and $(B_{i+1}',B_{i+1}'')$ with $B_{i+1}\subset B_{i}''^4\subset A^4$ such that
\begin{equation*}
\|f \ast \widetilde{\mu_{B_{i+1}}}\ast  \mu_{B_{i+1}} - f \ast \widetilde{\mu_{B_{i+1}'}}\ast  \mu_{B_{i+1}'}\|_{L^\infty(\mu_G)}\geq F(\nu,M)
\end{equation*}
and
\begin{equation*}
\mu_G(B_{i+1}''),\mu_G(B_{i+1}'),\mu_G(B_{i+1}) = \Omega_{\nu,K,M,\eta,i}(\mu_G(B_i'')) =  \Omega_{\nu,K,M,i}(\mu_G(A)).
\end{equation*}
It follows that all the hypotheses are satisfied except the bound on $\mu_{i+1}$.  By Lemma \ref{lem.spectralcollection} we have that
\begin{equation*}
\mu_{i+1} \geq \sum_{j=1}^N{|s_j(f)|\|\mu_{B_{i+1}} \ast v_j\|_{L^2(\mu_G)}^2} + F(\nu,M)^2/M - 8\eta M
\end{equation*}
Now, since $(B_i',B_{i}'')$ is an $\eta$-closed $4$-multiplicative pair and $B_{i+1} \subset B_i''^4$ we get that $(B_i',B_{i+1})$ is an $\eta$-closed $1$-multiplicative pair, whence by Lemma \ref{lem.spectralcollection} again we have that
\begin{equation*}
\sum_{j=1}^N{|s_j(f)|\|\mu_{B_{i+1}} \ast v_j\|_{L^2(\mu_G)}^2} \geq \mu_i - 8\eta M.
\end{equation*}
We thus conclude from the choice of $\eta$ that
\begin{equation*}
\mu_{i+1} \geq \mu_i + F(\nu,M)^2/2M.
\end{equation*}
Now, it remains to note that $\mu_i \leq M$, whence the iteration terminates with some $i = O_{\nu,M}(1)$; for it to terminate we must have been in the first case of Proposition \ref{prop.disclarge} and we are done since $B_i \subset A^4$.
\end{proof}
It turns out that $F$ is polynomial in its variables, whence the lower bound in the above proposition is a tower of towers of height $O(\nu^{-O(1)}M^{O(1)})$.

\section{Discontinuity in $A(G)$ implies correlation with a multiplicative pair}\label{sec.dis}
 
 In this section we shall show that if a function $f \in A(G)$ is not continuous in the sense of the conclusion of Proposition \ref{prop.quantcon}, then we have correlation with a multiplicative pair.  Specfically we show the following. 
 \begin{proposition}\label{prop.disc2}
Suppose that $G$ is a finite group, $f \in A(G)$ has $\|f\|_{A(G)} \leq M$, $A$ is symmetric and $\mu_G(A^4) \leq K\mu_G(A)$ and $\nu,\eta \in (0,1]$ are parameters.  Then either
\begin{enumerate}
\item there are symmetric neighbourhoods of the identity $B' \subset B \subset A^4$ such that $\mu_G(B')=\Omega_{K,\nu,\eta,M}(\mu_G(A))$,
\begin{equation*}
\sup_{x \in G}{\|f\ast\widetilde{\mu_B}\ast  \mu_B-f\ast\widetilde{\mu_B}\ast  \mu_B(x)\|_{L^\infty(\mu_{xB'})}}\leq \nu
\end{equation*}
and
\begin{equation*}
\sup_{x \in G}{\|f - f\ast\widetilde{\mu_B}\ast  \mu_B\|_{L^2(\mu_{xB'})}}\leq \nu.
\end{equation*}
\item or there are symmetric neighbourhoods of the identity $B''\subset B' \subset B \subset A^4$ such that  $\mu_G(B'')=\Omega_{K,\nu,\eta,M}(\mu_G(A))$, $\mathcal{B}':=(B,B')$ and $\mathcal{B}'':=(B',B'')$ are $\eta$-closed $4$-multiplicative pairs and
\begin{equation*}
\|f \ast \widetilde{\mu_B}\ast  \mu_B - f \ast \widetilde{\mu_{B'}}\ast  \mu_{B'}\|_{L^\infty(\mu_G)}= \Omega_{\nu,M}(1).
\end{equation*}
\begin{equation*}
\end{equation*}
\end{enumerate}
\end{proposition}
In the abelian setting the basic idea is that if the first conclusion is not satisfied then by Parseval's theorem there is a character at which $\wh{f}$ is large.  This is then converted into a density increment on a Bohr set by standard arguments.  In the non-abelian setting things are not so simple.

We shall take an singular value decomposition of $L_f$ and find that since $\|f\|_{A(G)}$ is bounded, we have a vector $v$ such that $f$ and $v$ have large cross energy.  We then use an averaging argument (Lemma \ref{lem.chopup}) to pass to a situation of having large energy and apply the work of \S\ref{sec.bog} to get correlation with a multiplicative pair. Unfortunately the averaging argument does not work directly and we are forced to introduce an additional regularity argument to ensure that $v$ is well enough behaved that it does.

We are now ready to proceed with the proof.
\begin{proof}[Proof of Proposition \ref{prop.disc2}]
Let $c$ be the absolute constant in the correlation lower bound in Proposition \ref{prop.mpls}.  We apply Proposition \ref{prop.regapp} to the set $A$ to get sets $(B_j)_{j=0}^J$ with $J=\lceil 100 M^2\nu^{-2} \rceil$, and positive reals $c_j=\Omega_{\eta,\nu,M,K}(1)$ such that $(B_i,B_j)$ is a $c_j$-thick, $\epsilon_j$-closed, $r_j$-multiplicative pair with
\begin{equation*}
r_j = 32 \textrm{ and } \epsilon_j \leq (c\eta \nu /2M)^{1000}c_{j-1}^4
\end{equation*}
for all $j \in \{0,\dots,J\}$ (with the obvious convection that $c_{-1}=1$), and
\begin{equation*}
B_0 \subset A^4 \textrm{ and } \mu_G(B_0) = \Omega_{\eta,\nu,M,K}(\mu_G(A)).
\end{equation*}
From Lemma \ref{lem.a-ap} and Lemma \ref{lem.bohrfourier} we have that
\begin{equation*}
\|f -f \ast \widetilde{\mu_{B_3}} \ast \mu_{B_3}\|_{A(G)} \leq M,
\end{equation*}
and hence, by Lemma \ref{lem.linfag},
\begin{equation*}
\|f-f \ast \widetilde{\mu_{B_3}} \ast \mu_{B_3}\|_{L^\infty(\mu_G)}\leq M.
\end{equation*}
We shall use these bounds in the sequel without comment.  Now, if
\begin{equation*}
\sup_{x \in G}{\|f -f \ast \widetilde{\mu_{B_3}} \ast \mu_{B_3}\|_{L^2(\mu_{xB_4})}^2} \leq \nu^2,
\end{equation*}
then we are done in the first case of the proposition with $B:= B_3$ and $B':=B_4$, since it is easy to check by Lemma \ref{lem.appcon} that since $\epsilon_4 \leq \nu/M$ and $r_4 \geq 1$ we have
\begin{equation*}
\sup_{x \in G}{\|f\ast \widetilde{\mu_{B_3}} \ast \mu_{B_3}-f \ast \widetilde{\mu_{B_3}} \ast \mu_{B_3}(x)\|_{L^\infty(\mu_{xB_4})}}\leq \nu.
\end{equation*}
Thus we may assume that there is some $x_0 \in G$ such that
\begin{equation}\label{eqn.juk}
\|f -f \ast \widetilde{\mu_{B_3}} \ast \mu_{B_3}\|_{L^2(\mu_{x_0B_4})}^2 >\nu^2.
\end{equation}
We put
\begin{equation*}
h:=(f -f \ast \widetilde{\mu_{B_3}} \ast \mu_{B_3})1_{x_0B_3} \ast \mu_{B_3},
\end{equation*}
and make the following claims about $h$.
\begin{claim*}
\begin{equation*}
\|h\|_{A(G)} \leq M \textrm{ and } \|h\|_{L^\infty(\mu_G)} \leq M
\end{equation*}
and
\begin{equation}\label{eqn.big}
h^2 \ast \mu_{B_4}(x_0)=\|h\|_{L^2(\mu_{x_0B_4})}^2 > 9\nu^2/16.
\end{equation}
\end{claim*}
\begin{proof}
To see the algebra norm bound note by the calculation in Lemma \ref{lem.convag} couples with the translation invariance of Lemma \ref{lem.invariance} that
\begin{equation*}
\|\widetilde{1_{xB_3}} \ast  \mu_{B_3}\|_{A(G)} \leq 1,
\end{equation*}
whence by the product property of the $A(G)$-norm we have the desired bound.  The second bound is just the usual domination of the algebra norm by $L^\infty(\mu_G)$ from Lemma \ref{lem.linfag}.

Now for the lower size estimate we begin by noting that
\begin{equation*}
\sup_{b \in B_4}|1_{x_0B_3} \ast \mu_{B_3}(x_0b) - 1_{x_0B_3} \ast \mu_{B_3}(x_0) | \leq \nu/4M
\end{equation*}
by Lemma \ref{lem.appcon} since $\epsilon_4 \leq \nu^2/16M^2$ and $r_4 \geq 1$.  On the other hand $1_{x_0B_3} \ast \mu_{B_3}(x_0)=1$, and so 
\begin{equation*}
\sup_{b \in B_4}|1_{x_0B_3} \ast \mu_{B_3}(x_0b) - 1| \leq \nu/4M.
\end{equation*}
Now, by a trivial instance of H{\"o}lder's inequality we then have
\begin{equation*}
\|f -f \ast \widetilde{\mu_{B_3}} \ast \mu_{B_3}-h\|_{L^2(\mu_{x_0B_4})}^2\leq M^2. (\nu/4M)^2 = \nu^2/16.
\end{equation*}
The final bound now follows from the triangle inequality and (\ref{eqn.juk}).
\end{proof}
Now, let $v_1,\dots,v_N$ be a Fourier basis of $L^2(\mu_G)$ for $h$ and recall that the singular values of $h$ are just
\begin{equation*}
|s_i(h)| = \|L_hv_i\|_{L^2(\mu_G)} \textrm{ for all } i \in \{1,\dots,N\}.
\end{equation*}
Decompose the vectors into sets according to the size of the corresponding singular value as follows:
\begin{equation*}
\mathcal{L}_j:=\{i:  \nu^6\mu_G(B_{j})/2^{20}M^5\leq |s_i(h)| < \nu^6\mu_G(B_{j-1})/2^{20}M^5\}.
\end{equation*}
By the explicit formula for the algebra norm we have
\begin{equation*}
\sum_{i=1}^N{\|L_hv_i\|_{L^2(\mu_G)}} = \|h\|_{A(G)} \leq M,
\end{equation*}
whence by the pigeonhole principle there is a natural $j$ with $ 5 \leq j \leq 8M^2\nu^{-2}+6$ such that
\begin{equation*}
 |\sum_{i \in \mathcal{L}_j}{\| L_hv_i\|_{L^2(\mu_G)}}| \leq \nu^2/8M.
\end{equation*}
We should like to examine $h$ on the set $B_j$ and the next claim asserts that it is large on some translate.
\begin{claim*}
There is some $x_j \in B_4$ such that
\begin{equation*}
h^2\ast \mu_{B_j}(x_0x_j)> \nu^2/2.
\end{equation*}
\end{claim*}
\begin{proof}
By Young's inequality and Lemma \ref{lem.approxhaar} we have that
\begin{equation*}
|h^2 \ast \mu_{B_j} \ast \mu_{B_4}(x_0)-h^2 \ast \mu_{B_4}(x_0)| \leq M^2\epsilon_j \leq \nu^2/16,
\end{equation*}
since
\begin{equation*}
r_j \geq 1 \textrm{ and } \epsilon_j \leq \nu^2/16M^2.
\end{equation*}
Now, recall from (\ref{eqn.big}) that $h^2 \ast \mu_{B_4}(x_0) > 9\nu^2/16$, whence, by the triangle inequality we have
\begin{equation*}
\int{h^2 \ast \mu_{B_j}(z) d\mu_{B_4}(z^{-1}x_0)}=h^2   \ast \mu_{B_j} \ast \mu_{B_4}(x_0)> \nu^2/2.
\end{equation*}
Thus by averaging there is some $x_j \in B_4$ such that
\begin{equation*}
h^2\ast \mu_{B_j}(x_0x_j)> \nu^2/2,
\end{equation*}
and the result follows.
\end{proof}
Write $g:=h.d\mu_{x_0x_jB_j}$ and
\begin{equation*}
\mathcal{L}:=\{i: \|L_gv_i\|_{L^2(\mu_G)}\geq \nu^2/8M\}.
\end{equation*}
We shall estimate the size of $\mathcal{L}$ using the Parseval bound in the usual way: by Parseval's theorem
\begin{eqnarray*}
\sum_{i \in \mathcal{L}}{\|L_gv_i\|_{L^2(\mu_G)}^2} & \leq & \sum_{i=1 }^N{\|L_gv_i\|_{L^2(\mu_G)}^2} \\ & = & \langle g,g\rangle_{L^2(\mu_G)} = h^2\ast \mu_{B_j}(x_0x_j)/\mu_G(x_0x_jB_j).
\end{eqnarray*}
Thus, by the definition of $\mathcal{L}$ and the upper bound on $\|h\|_{L^\infty(\mu_G)}$ we have
\begin{equation*}
|\mathcal{L}| \leq M^2\mu_G(x_0x_jB_j)^{-1}.(8M/\nu^2)^2 = 2^{6}M^4\nu^{-4}\mu_G(B_j)^{-1}.
\end{equation*}
Now, by the Cauchy-Schwarz inequality we have
\begin{equation}\label{eqn.ce1}
 |\sum_{i\not \in \mathcal{L}}{\langle L_hv_i,L_gv_i\rangle_{L^2(\mu_G)}}|  \leq  \sum_{i=1}^N{\|L_hv_i\|_{L^2(\mu_G)}}.\sup_{i \not \in \mathcal{L}}\|L_gv_i\|_{L^2(\mu_G)}\leq  \nu^2/8,
\end{equation}
by the explicit formula for $\|h\|_{A(G)}$ and the fact that its upper bound is $M$.  Furthermore, writing 
\begin{equation*}
\mathcal{S}:=\{i:  |s_i(h)| \leq \nu^6\mu_G(B_{j})/2^{20}M^5\}
\end{equation*}
we have (again by Cauchy-Schwarz)
\begin{equation*}
 |\sum_{i \in \mathcal{L}\cap \mathcal{S}}{\langle L_hv_i,L_gv_i\rangle_{L^2(\mu_G)}}| \leq  |\mathcal{L}|.\sup_{i \in \mathcal{S}}{\|L_gv_i\|_{L^2(\mu_G)}\|L_hv_i\|_{L^2(\mu_G)}}.
\end{equation*}
Now by the Hausdorff-Young bound (Lemma \ref{lem.hdy}) we have that
\begin{equation*}
\|L_gv_i\|_{L^2(\mu_G)} \leq \|g\| \leq \|h\|_{L^\infty(\mu_G)} \leq M,
\end{equation*}
whence
\begin{equation}\label{eqn.ce2}
|\sum_{i \in \mathcal{L}\cap \mathcal{S}}{\langle L_hv_i,L_gv_i\rangle_{L^2(\mu_G)}}| \leq  \nu^2/8.
\end{equation}
Finally, by the choice of $j$, we have that
\begin{equation}\label{eqn.ce3}
 |\sum_{i \in \mathcal{L}\cap \mathcal{L}_j}{\langle L_hv_i,L_gv_i\rangle_{L^2(\mu_G)}}|\leq \nu^2/8.
\end{equation}
Combining (\ref{eqn.ce1}), (\ref{eqn.ce2}) and (\ref{eqn.ce3}) by the triangle inequality we get that
\begin{equation*}
|\sum_{i \not \in \mathcal{L}\cup \mathcal{S} \cup \mathcal{L}_j}{\langle L_hv_i,L_gv_i\rangle_{L^2(\mu_G)}}| \leq 3\nu^2/8.
\end{equation*}
On the other hand, by Parseval's theorem
\begin{equation*}
\nu^2/2 < h^2 \ast \mu_{B_j}(x_0x_j) = \langle g,h \rangle_{L^2(\mu_G)} = \sum_{i=1}^N{ \langle L_hv_i,L_gv_i\rangle_{L^2(\mu_G)}},
\end{equation*}
whence by the triangle inequality
\begin{equation*}
|\sum_{i \in \mathcal{L}\cap \mathcal{S} \cap \mathcal{L}_j}{\langle L_hv_i,L_gv_i\rangle_{L^2(\mu_G)}}| > \nu^2/8.
\end{equation*}
In particular, there is some $i$ such that
\begin{equation*}
\|g \ast v_i\|_{L^2(\mu_G)} \geq \nu^2/8M, \textrm{ and } |s_i(h)| \geq \nu^6\mu_G(B_{j-1})/2^{20}M^5.
\end{equation*}
Now we should like to apply Lemma \ref{lem.chopup} to the sets $(B_1,B_2,B_{j})$.  First, writing $k$ for the function $x \mapsto h(x_0x_jx)$ restricted to $B_j$ we find that $k \in L^1(\mu_{B_j})$ and
\begin{equation*}
(kd\mu_{B_j}) \ast v_i(x) = g \ast v_i(x_0x_jx) \textrm{ for all }x \in G.
\end{equation*}
Thus
\begin{equation*}
\|(kd\mu_{B_j}) \ast v_i\|_{L^2(\mu_G)}^2= \|g \ast v_i\|_{L^2(\mu_G)}^2> \nu^2\|k\|_{L^\infty(\mu_{B_j})}^2\|v_i\|_{L^2(\mu_G)}^2/8M^3
\end{equation*}
by change of variables and the fact that $\|v_i\|_{L^2(\mu_G)}=1$.  Secondly, the function $h$ is supported on $x_0B_3^2 \subset x_0B_2$ since $r_3 \geq 1$ and
\begin{equation*}
\widetilde{(hd\mu_{x_0B_2})} \ast (hd\mu_{x_0B_2}) \ast v_i = \lambda \|h\|_{L^\infty(\mu_{x_0B_2})}^2v_i
\end{equation*}
for some $\lambda$ with
\begin{equation*}
|\lambda| \geq (\nu^6 \mu_G(B_{j-1})/2^{20}M^5)^2/\mu_G(B_2)^2\|h\|_{L^\infty(\mu_{x_0B_2})}^2 \geq \nu^{12}c_{j-1}^2/2^{40}M^{12}.
\end{equation*}
Thus since $r_2,r_j \geq 4$, $\epsilon_2 \leq 1$ and
\begin{equation*}
\epsilon_j \leq c_{j-1}^4\nu^{26}/2^{90}M^{27}\leq (\nu^2/8M^3).c_2^2(\nu^{12}c_{j-1}^2/2^{40}M^{12})^2/16
\end{equation*}
our application of Lemma \ref{lem.chopup} gives us some $x' \in G$ such that
\begin{equation*}
\|(kd\mu_{B_j}) \ast (\rho_{x'}(v_i)|_{B_1})\|_{L^2(\mu_{B_1})}^2 > \nu^2\|k\|_{L^\infty(\mu_{B_j})}^2\|\rho_{x'}(v_i)\|_{L^2(\mu_{B_1})}^2/2^5M^3,
\end{equation*}
and
\begin{equation*}
\|\rho_{x'}(v_i)\|_{L^\infty(\mu_{B_1})}\leq 16\nu^{-1}M|\lambda|^{-1}c_2^{-1}\|\rho_{x'}(v_i)\|_{L^2(\mu_{B_1})}.
\end{equation*}
In particular we have
\begin{equation*}
\frac{\|\rho_{x'}(v_i)\|_{L^2(\mu_{B_1})}}{\|\rho_{x'}(v_i)\|_{L^\infty(\mu_{B_1})} }\geq \nu^{13}c_{j-1}^3/M^{12}2^{44}.
\end{equation*}
Now we may apply Proposition \ref{prop.mpls} to the sets $B_0,B_1,B_j,B_{j+1}$ since $r_{j+1} \geq 32$,
\begin{equation*}
\epsilon_{j+1} \leq (c_j\nu^2/2^9M^3)^{C_{\rm{Bog}}}
\end{equation*}
and
\begin{equation*}
\epsilon_{j+1} \leq \nu^{54}c_{j-1}^{12}/M^{48}2^{200} \leq (\nu^2/2^8) .\left(\frac{\|\rho_{x'}(v_i)\|_{L^2(\mu_{B_1})}}{\|\rho_{x'}(v_i)\|_{L^\infty(\mu_{B_1})} }\right)^4.
\end{equation*}
The proposition gives us a positive real $c = \Omega_{\nu,\eta,M,K}(1)$ and $c$-thick $\eta$-closed $4$-multiplicative pairs $(B,B')$ and $(B',B'')$ with $B^2 \subset B_{j+1}^3$ and some $x' \in B_{j,j+1}^-$ such that
\begin{equation*}
|k \ast \widetilde{\mu_{B}} \ast \mu_B(x')|^2 \geq c\nu \|k\|_{L^\infty(\mu_G)}/2^4M^2,
\end{equation*}
where $c$ is the implied constant in the lower bound in Proposition \ref{prop.mpls}. Now, define 
\begin{equation*}
k'(x):=(f - f \ast \widetilde{\mu_{B_3}} \ast \mu_{B_3})(x_0x_jx)1_{B_j}(x).
\end{equation*}
We make the following claim.
\begin{claim*}
\begin{equation*}
\|k - k'\|_{L^\infty(\mu_G)} \leq c\nu \|k\|_{L^\infty(\mu_G)}/2^5M^2.
\end{equation*}
\end{claim*}
\begin{proof}
We have seen this calculation before.  We begin by noting that
\begin{equation*}
\sup_{b \in B_j}|1_{x_0B_3} \ast \mu_{B_3}(x_0x_jb) - 1_{x_0B_3} \ast \mu_{B_3}(x_0) |
\end{equation*}
is at most
\begin{equation*}
c\nu \|k\|_{L^\infty(\mu_G)}/2^5M^3
\end{equation*}
by Lemma \ref{lem.appcon} since $x_jb \in B_4^2$, 
\begin{equation*}
\epsilon_4 \leq c\nu/2^6M^4 \textrm{ and } r_4 \geq 2.
\end{equation*}
On the other hand $1_{x_0B_3} \ast \mu_{B_3}(x_0)=1$, and so 
\begin{equation*}
\sup_{b \in B_4}|1_{x_0B_3} \ast \mu_{B_3}(x_0b) - 1| \leq c\nu \|k\|_{L^\infty(\mu_G)}/2^5M^3.
\end{equation*}
Now, by a trivial instance of H{\"o}lder's inequality we are done.
\end{proof}
It follows from the claim and the triangle inequality that
\begin{equation*}
|k' \ast \widetilde{\mu_{B}} \ast \mu_B(x')|^2 \geq c\nu \|k\|_{L^\infty(\mu_G)}/2^5M^2.
\end{equation*}
Of course, since $x' \in B_{j,j+1}^+$ and $B^2 \subset B_{j+1}^3$ we have that $xB^2 \subset B_j$ since $r_{j+1} \geq 3$, whence
\begin{equation*}
k' \ast \widetilde{\mu_{B}} \ast \mu_B(x)= f \ast\widetilde{\mu_{B}} \ast \mu_B(x_0x_jx) - f \ast \widetilde{\mu_{B_3}} \ast \mu_{B_3} \ast \widetilde{\mu_{B}} \ast \mu_B(x_0x_jx).
\end{equation*}
It follows that
\begin{equation*}
\| f \ast\widetilde{\mu_{B}} \ast \mu_B - f \ast \widetilde{\mu_{B_3}} \ast \mu_{B_3} \ast \widetilde{\mu_{B}} \ast \mu_B\|_{L^\infty(\mu_G)} \geq c\nu \|k\|_{L^\infty(\mu_G)}/2^5M^2.
\end{equation*}
Finally, $B^2 \subset B_{j+1}^3$ whence $(B_3,B)$ is an $\epsilon_{j+1}$-closed $2$-multiplicative pair (since $r_{j+1} \geq 6$) , and so by Lemma \ref{lem.approxhaar} we have that
\begin{equation*}
\| f \ast\widetilde{\mu_{B_3}} \ast \mu_{B_3} - f \ast \widetilde{\mu_{B_3}} \ast \mu_{B_3} \ast \widetilde{\mu_{B}} \ast \mu_B\|_{L^\infty(\mu_G)}
\end{equation*}
is at most
\begin{equation*}
c\nu \|k\|_{L^\infty(\mu_G)}/2^6M^2,
\end{equation*}
since
\begin{equation*}
\epsilon_j \leq c\nu \|k\|_{L^\infty(\mu_G)}/2^6M^3.
\end{equation*}
The result follows by the triangle inequality.
\end{proof}
While the lower bound on the correlation is $\Omega(\nu^{O(1)}M^{-O(1)})$, the lower bound on the size of the balls is a tower of $2\eta^{-1}$s of height $O(\nu^{-O(1)}M^{O(1)})$.  The regularity argument in the above proof gives rise to one of the `tower contributions' in the final bound.

\section{From small algebra norm to correlation with a multiplicative pair}\label{sec.ank}

In this section we show that if $f$ is integer-valued and has small algebra norm then its square correlates with a large multiplicative pair. Specifically we shall prove the following proposition.
\begin{proposition}\label{prop.correl}
Suppose that $G$ is a finite group, $f:G \rightarrow \Z$, not identically zero, has $\|f\|_{A(G)} \leq M$ and $\epsilon \in (0,1]$ is a parameter.  Then there is positive real $c=\Omega_{M,\epsilon}(1)$ and a $c$-thick $\epsilon$-closed $4$-multiplicative pair $\mathcal{B}=(B,B')$ such that
\begin{equation*}
\|f^2\ast \mu_B\|_{L^\infty(\mu_G)} =\Omega_M(1) \textrm{ and } \mu_G(B)=\Omega_{M}(\|f\|_{L^2(\mu_G)}^2).
\end{equation*}
\end{proposition}
The important part about this result is that the lower bound on $\|f^2\ast\mu_B\|_{L^\infty(\mu_G)}$ depends only on $M$.

Our strategy is fairly straightforward: we note from the algebra norm property (and this is the only place in the paper that we use that property in generality) implies that $f^2$ also has small algebra norm.  Then, essentially by H{\"o}lder's inequality, we conclude that the support has large multiplicative energy, and apply the Balog-Szemer{\'e}di theorem and our weak Fre{\u\i}man-type theorem from \S\ref{sec.frei}.

We need to recall the non-abelian version of the Balog-Szemer{\'e}di-Gowers theorem -- the proof does not change in the passage to the non-abelian world as is remarked in \cite{taovu::}.  The argument seems to have been first officially recorded by Tao in \cite{tao::6}.
\begin{theorem}[{Balog-Szemer{\'e}di-Gowers theorem, \cite[Theorem 5.4]{tao::6}}]
Suppose that $G$ is a finite group, $A \subset G$ has $\|1_{A^{-1}}\ast 1_{A}\|_{L^2(\mu_G)}^2\geq c\mu_G(A)^3$.  Then there is a subset $A' \subset A$ and $x \in G$ such that
\begin{equation*}
\mu_G(A') \geq c^{O(1)}\mu_G(A) \textrm{ and } \mu_G((xA')^2) \leq c^{-O(1)}\mu_G(A').
\end{equation*}
\end{theorem}

\begin{proof}[Proof of Proposition \ref{prop.correl}]
Since $\|f\|_{A(G)} \leq M$, we have $\|f^2\|_{A(G)} \leq M^2$ since $A(G)$-norm is an algebra norm. Put $g:=f^2$, and let $v_1,\dots,v_N$ be a Fourier basis of $L^2(\mu_G)$ for $g$.  It follows from Parseval's theorem and the fact $h \mapsto L_h$ is a homomorphsim that
\begin{equation*}
\|\tilde{g} \ast g\|_{L^2(\mu_G)}^2 =\sum_{i=1}^N{\langle L_g^*L_gv_i,L_g^*L_gv_i\rangle_{L^2(\mu_G)}} =  \sum_{i=1}^N{|s_i(g)|^4}.
\end{equation*}
Now, by H{\"o}lder's inequality
\begin{equation*}
\left(\sum_{i=1}^N{|s_i(g)|^2}\right)^3 \leq \left(\sum_{i=1}^N{|s_i(g)|^4}\right)\left(\sum_{i=1}^N{|s_i(g)|}\right)^2.
\end{equation*}
On the other hand by the explicit formula for the $A(G)$-norm and Parseval's theorem we have
\begin{equation*}
\|g\|_{A(G)} = \sum_{i=1}^N{|s_i(g)|} \textrm{ and } \|g\|_{L^2(\mu_G)}^2=\sum_{i=1}^N{|s_i(g)|^2},
\end{equation*}
whence
\begin{equation*}
\|\tilde{g} \ast g\|_{L^2(\mu_G)}^2\geq \|g\|_{L^2(\mu_G)}^6/M^4.
\end{equation*}
Put $A:=\supp g$ and note that $\|g\|_{L^\infty(\mu_G)} \leq M^2$, whence $1_A \leq g \leq M^2.1_A$.  We conclude that
\begin{equation*}
\|1_{A^{-1}}\ast 1_A \|_{L^2(\mu_G)}^2 \geq \|1_A\|_{L^2(\mu_G)}^6/M^{12}.
\end{equation*}
We apply the Balog-Szemer{\'e}di-Gowers theorem to get a set $A' \subset A$ and $x \in G$ such that
\begin{equation*}
\mu_G(A')\geq M^{-O(1)}\|f\|_{L^2(\mu_G)}^2 \textrm{ and } \mu_G((xA')^2) \leq M^{O(1)}\mu_G(A').
\end{equation*}
Now apply Proposition \ref{prop.wkfr} to get a positive real $c=\Omega_{M,\epsilon}(1)$ and a $c$-thick $\epsilon$-closed $4$-multiplicative pair $(B,B')$ such that
\begin{equation*}
\mu_G(B) = \Omega_{M}(\mu_G(A)) \textrm{ and } \|1_{A'}\ast \mu_B\|_{L^\infty(\mu_G)} = \|1_{xA'}\ast \mu_B\|_{L^\infty(\mu_G)} =\Omega_M(1).
\end{equation*}
Since $A' \subset A$ it follows that $\|1_A \ast \mu_B\|_{L^\infty(\mu_G)}=\Omega_M(1)$, and we have the result.
\end{proof}
The bounds in the above are inherited from Proposition \ref{prop.wkfr}: $c^{-1}$ is quadruply exponential in $O(M^{O(1)}\epsilon^{-O(1)})$, while the correlation bounds are exponential in $O(M^{O(1)})$.

\section{The proof of the main theorem}\label{sec.pmt}

In this section we bring together the main results of \S\S\ref{sec.qc}\verb!&!\ref{sec.ank} in our proof of Theorem \ref{thm.main}.  We shall also make use of some of the more elementary results about the algebra norm and multiplicative pairs.

Before we start the proof proper we shall need two lemmas and some notation.  Suppose that $G$ is a finite group and $f:G \rightarrow \R$.  We say that $f$ is \emph{$\epsilon$-almost integer-valued} if $f(G) \subset \Z+(-\epsilon,\epsilon)$ and we write $f_{\Z}$ for the (unique if $\epsilon<1/2$) integer-valued function which most closely approximates $f$.

Although we shall not be working with the whole class of $\epsilon$-almost integer-valued functions we shall be dipping in and out.  Our first lemma lets us take an almost integer-valued function with small algebra norm and bound the algebra norm of its integral approximant in certain cases.
\begin{lemma}\label{lem.apxalg}
Suppose that $G$ is a finite group and $\epsilon \in [0,1/6)$ is a parameter such that $f \in L^1(\mu_G)$ is $\epsilon$-almost integer valued, and $H \leq G$ is a subgroup such that $f_\Z \ast \mu_H$ is $\epsilon$-almost integer valued and
\begin{equation*}
\|f_\Z\|_{A(G)} \leq M \textrm{and } \mu_G(H) \geq \eta\|f_{\Z}\|_{L^1(\mu_G)}.
\end{equation*}
Then $f-f \ast \mu_H$ is $3\epsilon$-almost integer-valued,
\begin{equation*}
\|(f-f\ast \mu_H)_{\Z}\|_{A(G)} =O(\eta^{-1}M)
\end{equation*}
and there is a natural $k=O(\eta^{-1})$ and integers $z_1,\dots,z_k$ with absolute values at most  $M+O(1)$ such that
\begin{equation*}
(f \ast \mu_H)_{\Z}=\sum_{i=1}^k{z_i.1_{x_i.H}}.
\end{equation*}
\end{lemma}
\begin{proof}
Since $f$ is $\epsilon$-almost integer-valued we have that $\|f-f_\Z\|_{L^\infty(\mu_G)} \leq \epsilon$.  It follows from Young's inequality that
\begin{equation}\label{eqn.wb}
\|f \ast \mu_H - f_\Z \ast \mu_H\|_{L^\infty(\mu_G)} \leq \epsilon.
\end{equation}
However, $f_\Z\ast \mu_H$ is $\epsilon$-almost integer valued, whence
\begin{equation*}
\|f \ast \mu_H - (f_\Z \ast \mu_H)_\Z\|_{L^\infty(\mu_G)} \leq 2\epsilon
\end{equation*}
by the triangle inequality.  It follows that $f \ast \mu_H$ is $2\epsilon$-almost integer-valued and, again by the triangle inequality, that $f-f \ast \mu_H$ is $3\epsilon$-almost integer valued as required. Furthermore, since $3\epsilon<1/2$ we have that $(f-f \ast \mu_H)_\Z$ is well-defined.

Now we examine how often we can have $|(f \ast \mu_H)_{\Z}(x)|>0$.  Since $f \ast \mu_H$ is $2\epsilon$-almost integer valued and $\epsilon<1/4$ we have that $|(f \ast \mu_H)_{\Z}(x)|>0$ if and only if $|f \ast \mu_H(x)|>1/2$.  However, $f \ast \mu_H$ is constant on cosets of $H$ whence $|(f \ast \mu_H)_{\Z}(x)|>0$ if and only if $|f \ast \mu_H(x')|>1/2$ for all $x' \in xH$. It follows that there are cosets $x_1.H,\dots,x_k.H$ such that
\begin{equation}\label{eqn.1}
(f \ast \mu_H)_{\Z}=\sum_{i=1}^k{(f \ast \mu_H)_\Z(x_i).1_{x_i.H}}.
\end{equation}
However since $f \ast \mu_H$ is $2\epsilon$-almost integer-valued we have that
\begin{equation*}
|(f \ast \mu_H)_{\Z}(x)| \leq (1-2\epsilon)^{-1}|f \ast \mu_H(x)|1_{|f \ast \mu_H(x)|>1/2}.
\end{equation*}
Now, from (\ref{eqn.wb}) we have
\begin{equation*}
|f_\Z\ast \mu_H(x)| \geq |f \ast \mu_H(x)| - \epsilon \geq |f \ast \mu_H(x)|1_{|f \ast \mu_H(x)|>1/2}(1-2\epsilon).
\end{equation*}
It follows that
\begin{equation*}
(f \ast \mu_H)_{\Z}(x) \leq (1-2\epsilon)^{-2}|f_\Z\ast \mu_H(x)| \leq 4|f_\Z\ast \mu_H(x)| .
\end{equation*}
Using this upper bound in (\ref{eqn.1}) and the fact that the cosets are disjoint we get that
\begin{equation*}
k\mu_G(H)/2 \leq \|(f \ast \mu_H)_{\Z}\|_{L^1(\mu_G)} \leq 4\|f_\Z\ast \mu_H\|_{L^1(\mu_G)} \leq 4\|f_\Z\|_{L^1(\mu_G)}.
\end{equation*}
The bound on $k$ now follows from the lower bound on the size of $H$.  To bound the integers $(f \ast \mu_H)_\Z(x_i)$ we just note that
\begin{eqnarray*}
\|(f \ast \mu_H)_{\Z}\|_{L^\infty(\mu_G)}& \leq& \|f\ast \mu_H\|_{L^\infty(\mu_G)} + 1\\& \leq & \|f\|_{L^\infty(\mu_G)}+1 \leq \|f_Z\|_{L^\infty(\mu_G)} + 2 \leq M+2
\end{eqnarray*}
by Young's inequality again and Lemma \ref{lem.linfag}.  Finally by Corollary \ref{cor.indsmall} and the triangle inequality we have
\begin{equation*}
\|(f \ast \mu_H)_{\Z}\|_{A(G)} \leq \sum_{i=1}^k{|(f \ast \mu_H)_\Z(x_i)|} \leq 8\eta^{-1}(M+2),
\end{equation*}
and the result follows by the triangle inequality given that $(f-f\ast \mu_H)_\Z = f_\Z - (f\ast\mu_H)_\Z$ and $\|f_\Z\|_{A(G)} \leq M$.
\end{proof}
The next lemma lets us take a quantitative notion of continuity such as that developed in \S\ref{sec.qc} and show that if at the same time the function is almost integer valued then in fact it is approximately constant on the subgroup generated by the ball of continuity.
\begin{lemma}\label{lem.spread}
Suppose that $G$ is a finite group, $f:G \rightarrow \Z$, $g \in L^\infty(\mu_G)$ and $B \subset G$ is a symmetric non-empty set such that for some parameter $\epsilon \in [0,1/10)$ we have
\begin{equation*}
\sup_{x \in G}{\|g-g(x)\|_{L^\infty(\mu_{xB})}} \leq \epsilon \textrm{ and } \sup_{x \in G}{\|f-g\|_{L^2(\mu_{xB})}} \leq \epsilon.
\end{equation*}
Then, writing $H:=\langle B \rangle$ for the group generated by $B$ we have that $f \ast \mu_H$ is $5\epsilon$-almost integer-valued and
\begin{equation*}
\|f \ast \mu_H\|_{L^\infty(\mu_G)} > \|g\|_{L^\infty(\mu_G)} - 3\epsilon.
\end{equation*}
\end{lemma}
\begin{proof}
It follows from the triangle inequality and nesting of norms that
\begin{equation*}
\sup_{x \in G}{\|f-g(x)\|_{L^2(\mu_{xB})}} \leq 2\epsilon.
\end{equation*}
Let $z$ be the smallest integer with $z \geq g(x)$ and put $S:=\{y \in xB: f(y) \geq z\}$.  Since $f$ is integer-valued it follows that
\begin{equation*}
|z-g(x)|.\sqrt{\mu_{xB}(S)} + |z-1-g(x)|.\sqrt{1-\mu_{xB}(S)} \leq 2\epsilon.
\end{equation*}
It follows that
\begin{equation*}
\min\{|z-g(x)|,|z-1-g(x)|\}.(\sqrt{\mu_{xB}(S)} + \sqrt{1-\mu_{xB}(S)}) \leq 2\epsilon.
\end{equation*}
However,
\begin{equation*}
\sqrt{\mu_{xB}(S)} + \sqrt{1-\mu_{xB}(S)}\geq 1,
\end{equation*}
whence $g(x)$ is within $2\epsilon$ of an integer.

 On the other hand $|g(y)-g(x)|<\epsilon$ whenever $y\in xB$ whence
 \begin{equation*}
 |g_\Z(x)-g_\Z(y)| \leq |g_\Z(x) - g(x)| + |g(x) - g(y)| + |g(y) - g_\Z(y)| \leq 5\epsilon.
 \end{equation*}
 It follows that $g_\Z(x)=g_\Z(y)$ since $\epsilon<1/10$, and hence $g_\Z$ is constant on cosets of $H$.
 
Now, note that for all $x \in G$ we have
\begin{eqnarray*}
|f \ast \mu_B(x) - g(x)|^2 &\leq& (|f - g(x)| \ast \mu_B(x))^2\\ & \leq & (f-g(x))^2\ast \mu_B(x) \leq \epsilon^2
\end{eqnarray*}
by the Cauchy-Schwarz inequality.  Now, integrating over cosets of $H$ we get that
\begin{equation*}
|f \ast \mu_B\ast \mu_H(x) - g\ast \mu_H(x)| = |\int{f \ast \mu_B(x')d\mu_{xH}(x')} -\int{g(x')d\mu_{xH}(x')}|\leq \epsilon.
\end{equation*}
However, $\mu_B \ast \mu_H=\mu_H$ whence
\begin{equation*}
|f \ast \mu_H(x) - g \ast \mu_H(x)| \leq \epsilon.
\end{equation*}
Now $g$ is $\epsilon$-almost integer-valued and $g_\Z$ is constant on cosets of $H$, whence
\begin{eqnarray*}
|g \ast \mu_H(x) - g(x)| &\leq& |g \ast \mu_H(x) - g_\Z(x)| +\epsilon\\ & = & |g \ast \mu_H(x) - g_\Z \ast \mu_H(x)| +\epsilon \leq 2\epsilon.
\end{eqnarray*}
The lemma follows from the triangle inequality.
\end{proof}
We are now in a position to prove the main theorem.
\begin{proof}[Proof of Theorem \ref{thm.main}]
We define a sequence of functions $(f_i)_{i\geq 1}$ and subgroups $(H_i)_{i\geq 1}$ with the following properties.
\begin{enumerate}
\item \label{pty.00} $f_i$ is $3^i\epsilon $-almost integer-valued;
\item \label{pty.01} $M_i:=\|(f_i)_\Z\|_{A(G)}$ has $M_i=O_{M,i}(1)$;
\item \label{pty.02} $\|f_i \ast \mu_{H_i}\|_{L^\infty(\mu_G)} >1/2$;
\item \label{pty.03} $f_{i+1} =f_i - f_i \ast \mu_{H_i}$;
\item \label{pty.04} there are integers $k,z_1,\dots,z_k = O_{M,i}(1)$ and elements $x_1,\dots,x_k \in G$ such that\begin{equation*}
(f_{i} \ast \mu_{H_i})_\Z = \sum_{j=1}^k{z_j.1_{x_j.H_i}}.
\end{equation*}
\end{enumerate}

We put $f_0:=f$ which trivially satisfies the first two conditions above since $f$ is integer valued.  Suppose that we are at stage $i$ of the iteration, having defined $f_i$ satisfying conditions (\ref{pty.00}) and (\ref{pty.01}).  Write $F$ for the function hiding behind the $\Omega_M(1)$ in Proposition \ref{prop.correl}, $\epsilon_i=3^i\epsilon$ and
\begin{equation*}
\nu_i:=\min\{\epsilon_i/5,F(M_i)^2/12,1/20\}.
\end{equation*}
If $(f_i)_\Z \equiv 0$ then terminate the iteration.  Apply Proposition \ref{prop.correl} to $(f_i)_{\Z}$ to get a $\nu_i$-closed $4$-multiplicative pair $(B_i,B_i')$ such that
\begin{equation}\label{eqn.wp}
\|(f_i)_\Z^2\ast \mu_{B_i}\|_{L^\infty(\mu_G)}\geq F(M_i),
\end{equation}
and
\begin{equation*}
\mu_G(B_i),\mu_G(B_i')=\Omega_{\epsilon,i,M}(\|(f_i)_\Z\|_{L^2(\mu_G)}^2).
\end{equation*}
Since $B_i'^4 \subset B_i$ we get (from the bounds on $\mu_G(B_i)$ and $\mu_G(B_i')$) that
\begin{equation*}
\mu_G(B_i'^4) =O_{\epsilon,i,M}(\mu_G(B_i')).
\end{equation*}
We may thus apply Proposition \ref{prop.quantcon} to $B_i'$ and $(f_i)_\Z$ with parameter $\nu_i$ to get balls $B_i''' \subset B_i'' \subset B_i^4$ with  $\mu_G(B_i''')=\Omega_{\epsilon,i,M}(\mu_G(B_i))$,
\begin{equation*}
\sup_{x \in G}{\|(f_i)_\Z\ast\widetilde{\mu_{B_i''}}\ast  \mu_{B_i''}-(f_i)_\Z\ast\widetilde{\mu_{B_i''}}\ast  \mu_{B_i''}(x)\|_{L^\infty(\mu_{xB_i'''})}}\leq \nu_i
\end{equation*}
and
\begin{equation}\label{eqn.jx}
\sup_{x \in G}{\|(f_i)_\Z- (f_i)_\Z\ast\widetilde{\mu_{B_i''}}\ast  \mu_{B_i''}\|_{L^2(\mu_{xB_i'''})}}\leq \nu_i.
\end{equation}
Writing $H_i:=\langle B_i'''\rangle$ we have, by Lemma \ref{lem.spread}, that $(f_i)_\Z \ast \mu_{H_i}$ is $5\nu_i$-almost integer-valued and
\begin{equation}\label{eqn.maass}
\|(f_i)_\Z \ast \mu_{H_i}\|_{L^\infty(\mu_G)} > \|(f_i)_\Z\ast\widetilde{\mu_{B_i''}}\ast  \mu_{B_i''}\|_{L^\infty(\mu_G)} - 3\nu_i.
\end{equation}
On the other hand, by (\ref{eqn.jx}) we have that
\begin{eqnarray*}
\|(f_i)_\Z \ast \widetilde{\mu_{B_i''}} \ast \mu_{B_i''}\|_{L^\infty(\mu_G)}& \geq & \sup_{x \in G}{\|(f_i)_\Z\|_{L^2(xB_i''')}} -\nu_i\\ & = & \|(f_i)_\Z^2 \ast \mu_{B_i'''}\|_{L^\infty(\mu_G)}^{1/2} - \nu_i. 
\end{eqnarray*}
Now, since $B_i''' \subset B_i'^4$ and $(B_i,B_i')$ is a $\nu_i$-closed $4$-multiplicative pair we see that
\begin{equation*}
\| \mu_{B_i'''}  \ast \mu_{B_i}- \mu_{B_i}\| \leq \nu_i 
\end{equation*}
by Lemma \ref{lem.approxhaar}.  It follows by the triangle inequality and Young's inequality that
\begin{equation*}
\|(f_i)_{\Z}^2 \ast \mu_{B_i'''}\|_{L^\infty(\mu_G)}\geq \|(f_i)_{\Z}^2 \ast \mu_{B_i}\|_{L^\infty(\mu_G)} - \nu_i,
\end{equation*}
and hence
\begin{equation*}
\|(f_i)_\Z \ast \widetilde{\mu_{B_i''}} \ast \mu_{B_i''}\|_{L^\infty(\mu_G)}\geq  \|(f_i)_{\Z}^2 \ast \mu_{B_i}\|_{L^\infty(\mu_G)}^{1/2} - 2\nu_i.
\end{equation*}
Combining this with (\ref{eqn.maass}) tells us that
\begin{equation*}
\|(f_i)_\Z \ast \mu_{H_i}\|_{L^\infty(\mu_G)} > \|(f_i)_\Z^2 \ast \mu_{B_i}\|_{L^\infty(\mu_G)}^{1/2} -6\nu_i>0
\end{equation*}
by choice of $\nu_i$ and the lower bound from (\ref{eqn.wp}).  Since $(f_i)_\Z \ast \mu_{H_i}$ is $5\nu_i$-almost integer-valued and $f_i$ is $\epsilon_i$-almost integer valued this bootstraps to
\begin{equation*}
\|f_i\ast \mu_{H_i}\|_{L^\infty(\mu_G)} >1/2
\end{equation*}
provided $\epsilon_i<1/4$, and so (\ref{pty.03}) is satisfied by $f_i$. Now, by Lemma \ref{lem.apxalg} applied to $f_i$ which is $\epsilon_i$-almost integer-valued and has $\|(f_i)_\Z\|_{A(G)} \leq M_i$, and the subgroup $H_i$ which is such that $(f_i)_\Z\ast \mu_{H_i}$ is $\epsilon_i$-almost integer-valued by choice of $\nu_i$ and
\begin{equation*}
\mu_G(H_i) \geq \mu_G(B_i''') = \Omega_{\epsilon,i,M}(\|(f_i)_\Z\|_{L^2(\mu_G)}^2)= \Omega_{\epsilon,i,M}(\|(f_i)_\Z\|_{L^1(\mu_G)}).
\end{equation*}
Property (\ref{pty.04}) now follows immediately from this lemma, as do properties (\ref{pty.00}) and (\ref{pty.01}) for $f_{i+1}$.  This closes the induction.

Now, by properties (\ref{pty.02}) and (\ref{pty.03}) of the sequence constructed above and Lemma \ref{lem.decompmass} we get that
\begin{equation*}
\|f_{i+1}\|_{A(G)}^2 \leq \|f_i\|_{A(G)}^2 - 1/4.
\end{equation*}
Since $\|f_0\|_{A(G)} \leq M$ it follows that the iteration cannot proceed for more than $O(M^2)$ steps which lets us choose $\epsilon \geq \exp(-O(M^2))$ so that $3^i\epsilon < 1/10i$. It remains to unravel the situation when the construction terminates.  Suppose it does so at some stage $i_0$ in which case we have $(f_{i_0})_\Z \equiv 0$.  Since $f_i$ is always $1/4$-almost integer-valued we have that
\begin{equation*}
(f_{i+1})_\Z =( f_i)_\Z-(f_i \ast \mu_{H_i})_\Z,
\end{equation*}
by the definition of the $f_i$s. It follows by induction that
\begin{equation*}
f=f_0=(f_0)_\Z=\sum_{i=0}^{i_0-1}{(f_i \ast \mu_{H_i})_\Z}.
\end{equation*}
On the other hand each of the summands has a structure described by (\ref{pty.04}), and so combining all these gives the result.
 \end{proof}
The bound on $L$ of a tower of tower of towers in $O(M^2)$ can be easily read out of this argument: essentially we iterate $O(M^2)$ times and each time we do it we replace $M_i$ by a tower of towers in $M_i$, whence the bound.

\section{Concluding remarks}\label{sec.con}

The bounds in Theorem \ref{thm.main} appear rather weak and, indeed, we have no better example of an integer-valued function with small algebra norm than we have in the abelian setting, namely an arithmetic progression. Specifically if $G=\Z/p\Z$ for some large prime $p$ and $A$ is an arithmetic progression then it is easy enough to see that $\|1_A\|_{A(G)} = \Omega(\log |A|)$.  Of course if $|A|$ is sufficiently small then any $\pm$-decomposition of $1_A$ into indicator functions of cosets must involve $|A|$ terms.  It follows that $L(\|1_A\|_{A(G)}) = \Omega(|A|)$ and hence we must have $L(M)=\exp(\Omega(M))$ in Theorem \ref{thm.main}.

\section*{Acknowledgements}

The author should like to thank Ben Green for useful conversations and much encouragement, and the anonymous referee for suggesting many improvements to the paper.

\bibliographystyle{halpha}

\bibliography{references}

\begin{thebibliography}{Shk08b}
\expandafter\ifx\csname url\endcsname\relax
  \def\url#1{\texttt{#1}}\fi
\expandafter\ifx\csname doi\endcsname\relax
  \def\doi#1{\burlalt{doi:#1}{http://dx.doi.org/#1}}\fi
\expandafter\ifx\csname urlprefix\endcsname\relax\def\urlprefix{URL }\fi
\expandafter\ifx\csname href\endcsname\relax
  \def\href#1#2{#2}\fi
\expandafter\ifx\csname burlalt\endcsname\relax
  \def\burlalt#1#2{\href{#2}{#1}}\fi

\bibitem[BG10]{bregre::0}
E.~Breuillard and B.~J. Green.
\newblock Approximate groups, {II}: the solvable linear case.
\newblock {\em Q. J. Math.}, 2010.
\newblock To appear.

\bibitem[BG11]{bregre::}
E.~Breuillard and B.~J. Green.
\newblock Approximate groups. {I}: the torsion-free nilpotent case.
\newblock {\em J. Inst. Math. Jussieu}, 10(1):37--57, 2011.
\newblock \doi{10.1017/S1474748010000150}.

\bibitem[Bog39]{bog::}
N.~Bogolio{\`u}boff.
\newblock Sur quelques propri\'et\'es arithm\'etiques des presque-p\'eriodes.
\newblock {\em Ann. Chaire Phys. Math. Kiev}, 4:185--205, 1939.

\bibitem[Bou99]{bou::5}
J.~Bourgain.
\newblock On triples in arithmetic progression.
\newblock {\em Geom. Funct. Anal.}, 9(5):968--984, 1999.
\newblock \doi{10.1007/s000390050105}.

\bibitem[Coh60]{coh::}
P.~J. Cohen.
\newblock On a conjecture of {L}ittlewood and idempotent measures.
\newblock {\em Amer. J. Math.}, 82:191--212, 1960.
\newblock \doi{10.2307/2372731}.

\bibitem[CS10]{crosis::}
E.~S. Croot and O.~Sisask.
\newblock A probabilistic technique for finding almost-periods of convolutions.
\newblock {\em Geom. Funct. Anal.}, 20(6):1367--1396, 2010,
  \burlalt{arXiv:1003.2978}{http://arxiv.org/abs/arXiv:1003.2978}.
\newblock \doi{10.1007/s00039-010-0101-8}.

\bibitem[DF03]{desfre::}
J.-M. Deshouillers and G.~A. Freiman.
\newblock A step beyond {K}neser's theorem for abelian finite groups.
\newblock {\em Proc. London Math. Soc. (3)}, 86(1):1--28, 2003.
\newblock \doi{10.1112/S0024611502013709}.

\bibitem[DSV03]{davsarval::}
G.~Davidoff, P.~Sarnak, and A.~Valette.
\newblock {\em Elementary number theory, group theory, and {R}amanujan graphs},
  volume~55 of {\em London Mathematical Society Student Texts}.
\newblock Cambridge University Press, Cambridge, 2003.
\newblock \doi{10.1017/CBO9780511615825}.

\bibitem[Eym64]{eym::}
P.~Eymard.
\newblock L'alg\`ebre de {F}ourier d'un groupe localement compact.
\newblock {\em Bull. Soc. Math. France}, 92:181--236, 1964.
\newblock \doi{doi.org/10.24033/bsmf.1607}.

\bibitem[FKP10]{fiskatpen::}
D.~Fisher, N.~H. Katz, and I.~Peng.
\newblock Approximate multiplicative groups in nilpotent {L}ie groups.
\newblock {\em Proc. Amer. Math. Soc.}, 138(5):1575--1580, 2010.
\newblock \doi{10.1090/S0002-9939-10-10078-1}.

\bibitem[Fou77]{fou::}
J.~J.~F. Fournier.
\newblock Sharpness in {Y}oung's inequality for convolution.
\newblock {\em Pacific J. Math.}, 72(2):383--397, 1977.

\bibitem[Fre66]{fre::}
G.~A. Freiman.
\newblock {\em Nachala strukturnoi teorii slozheniya mnozhestv}.
\newblock Kazan. Gosudarstv. Ped. Inst, 1966.

\bibitem[Fre73]{fre::0}
G.~A. Freiman.
\newblock {\em Foundations of a structural theory of set addition}.
\newblock American Mathematical Society, Providence, R. I., 1973.
\newblock \doi{10.1090/mmono/037}.
\newblock Translated from the Russian, Translations of Mathematical Monographs,
  Vol 37.

\bibitem[GK09]{grekon::}
B.~J. Green and S.~V. Konyagin.
\newblock On the {L}ittlewood problem modulo a prime.
\newblock {\em Canad. J. Math.}, 61(1):141--164, 2009,
  \burlalt{arXiv:math/0601565}{http://arxiv.org/abs/arXiv:math/0601565}.
\newblock \doi{10.4153/CJM-2009-007-4}.

\bibitem[Gow98]{gow::4}
W.~T. Gowers.
\newblock A new proof of {S}zemer\'edi's theorem for arithmetic progressions of
  length four.
\newblock {\em Geom. Funct. Anal.}, 8(3):529--551, 1998.
\newblock \doi{10.1007/s000390050065}.

\bibitem[Gow08]{gow::2}
W.~T. Gowers.
\newblock Quasirandom groups.
\newblock {\em Comb. Probab. Comput.}, 17(3):363--387, 2008,
  \burlalt{arXiv:0710.3877}{http://arxiv.org/abs/arXiv:0710.3877}.
\newblock \doi{10.1017/S0963548307008826}.

\bibitem[GR07]{greruz::0}
B.~J. Green and I.~Z. Ruzsa.
\newblock Freiman's theorem in an arbitrary {A}belian group.
\newblock {\em J. Lond. Math. Soc. (2)}, 75(1):163--175, 2007,
  \burlalt{arXiv:math/0505198}{http://arxiv.org/abs/arXiv:math/0505198}.
\newblock \doi{10.1112/jlms/jdl021}.

\bibitem[Gre09]{gre::05}
B.~J. Green.
\newblock Approximate groups and their applications: work of {B}ourgain,
  {G}amburd, {H}elfgott and {S}arnak.
\newblock 2009,
  \burlalt{arXiv:0911.3354}{http://arxiv.org/abs/arXiv:0911.3354}.

\bibitem[GS08]{gresan::0}
B.~J. Green and T.~Sanders.
\newblock A quantitative version of the idempotent theorem in harmonic
  analysis.
\newblock {\em Ann. of Math. (2)}, 168(3):1025--1054, 2008,
  \burlalt{arXiv:math/0611286}{http://arxiv.org/abs/arXiv:math/0611286}.
\newblock \doi{10.4007/annals.2008.168.1025}.

\bibitem[GT08]{gretao::1}
B.~J. Green and T.~C. Tao.
\newblock An inverse theorem for the {G}owers {$U\sp 3(G)$} norm.
\newblock {\em Proc. Edinb. Math. Soc. (2)}, 51(1):73--153, 2008,
  \burlalt{arXiv:math/0503014}{http://arxiv.org/abs/arXiv:math/0503014}.
\newblock \doi{10.1017/S0013091505000325}.

\bibitem[GT09]{gretao::3}
B.~J. Green and T.~C. Tao.
\newblock A note on the {F}reiman and {B}alog-{S}zemer\'edi-{G}owers theorems
  in finite fields.
\newblock {\em J. Aust. Math. Soc.}, 86(1):61--74, 2009.
\newblock \doi{10.1017/S1446788708000359}.

\bibitem[Hal50]{hal::}
P.~R. Halmos.
\newblock {\em Measure {T}heory}.
\newblock D. Van Nostrand Company, Inc., New York, N. Y., 1950.

\bibitem[Hos86]{hos::}
B.~Host.
\newblock Le th\'eor\`eme des idempotents dans {$B(G)$}.
\newblock {\em Bull. Soc. Math. France}, 114(2):215--223, 1986.
\newblock \doi{10.24033/bsmf.2055}.

\bibitem[HP02]{hampla::}
Y.~O. Hamidoune and A.~Plagne.
\newblock A generalization of {F}reiman's {$3k-3$} theorem.
\newblock {\em Acta Arith.}, 103(2):147--156, 2002.
\newblock \doi{10.4064/aa103-2-4}.

\bibitem[Hru12]{hru::}
E.~Hrushovski.
\newblock Stable group theory and approximate subgroups.
\newblock {\em J. Amer. Math. Soc.}, 25(1):189--243, 2012,
  \burlalt{arXiv:0909.2190}{http://arxiv.org/abs/arXiv:0909.2190}.
\newblock \doi{10.1090/S0894-0347-2011-00708-X}.

\bibitem[IS05]{ilispr::}
M.~Ilie and N.~Spronk.
\newblock Completely bounded homomorphisms of the {F}ourier algebras.
\newblock {\em J. Funct. Anal.}, 225(2):480--499, 2005.
\newblock \doi{10.1016/j.jfa.2004.11.011}.

\bibitem[KSV09]{kraserven::}
D.~Kr{\'a}l, O.~Serra, and L.~Vena.
\newblock A combinatorial proof of the removal lemma for groups.
\newblock {\em J. Combin. Theory Ser. A}, 116(4):971--978, 2009,
  \burlalt{arXiv:0804.4847}{http://arxiv.org/abs/arXiv:0804.4847}.
\newblock \doi{10.1016/j.jcta.2008.12.003}.

\bibitem[{\L}ab01]{aba::}
I.~{\L}aba.
\newblock Fuglede's conjecture for a union of two intervals.
\newblock {\em Proc. Amer. Math. Soc.}, 129(10):2965--2972 (electronic), 2001.
\newblock \doi{10.1090/S0002-9939-01-06035-X}.

\bibitem[Lef72]{lef::}
M.~Lefranc.
\newblock Sur certaines alg\`ebres de fonctions sur un groupe.
\newblock {\em C. R. Acad. Sci. Paris S\'er. A-B}, 274:A1882--A1883, 1972.

\bibitem[LPS88]{lubphisar::}
A.~Lubotzky, R.~Phillips, and P.~Sarnak.
\newblock Ramanujan graphs.
\newblock {\em Combinatorica}, 8(3):261--277, 1988.

\bibitem[Meh04]{meh::}
M.~L. Mehta.
\newblock {\em Random matrices}, volume 142 of {\em Pure and Applied
  Mathematics (Amsterdam)}.
\newblock Elsevier/Academic Press, Amsterdam, third edition, 2004.

\bibitem[Pl{\"u}69]{plu::}
H.~Pl{\"u}nnecke.
\newblock {\em Eigenschaften und {A}bsch\"atzungen von {W}irkungsfunktionen}.
\newblock BMwF-GMD-22. Gesellschaft f\"ur Mathematik und Datenverarbeitung,
  Bonn, 1969.

\bibitem[Rot53]{rot::0}
K.~F. Roth.
\newblock On certain sets of integers.
\newblock {\em J. London Math. Soc.}, 28:104--109, 1953.
\newblock \doi{10.1112/jlms/s1-28.1.104}.

\bibitem[Ruz94]{ruz::9}
I.~Z. Ruzsa.
\newblock Generalized arithmetical progressions and sumsets.
\newblock {\em Acta Math. Hungar.}, 65(4):379--388, 1994.
\newblock \doi{10.1007/BF01876039}.

\bibitem[Ruz99]{ruz::01}
I.~Z. Ruzsa.
\newblock An analog of {F}reiman's theorem in groups.
\newblock {\em Ast\'erisque}, (258):xv, 323--326, 1999.
\newblock
  \urlprefix\url{http://archive.dimacs.rutgers.edu/TechnicalReports/TechReports/1993/93-77.ps}.
\newblock Structure theory of set addition.

\bibitem[Sae68a]{sae::}
S.~Saeki.
\newblock On norms of idempotent measures.
\newblock {\em Proc. Amer. Math. Soc.}, 19:600--602, 1968.

\bibitem[Sae68b]{sae::0}
S.~Saeki.
\newblock On norms of idempotent measures. {II}.
\newblock {\em Proc. Amer. Math. Soc.}, 19:367--371, 1968.

\bibitem[San10]{san::03}
T.~Sanders.
\newblock On a non-abelian {B}alog-{S}zemer\'edi-type lemma.
\newblock {\em J. Aust. Math. Soc.}, 89(1):127--132, 2010,
  \burlalt{arXiv:0912.0306}{http://arxiv.org/abs/arXiv:0912.0306}.
\newblock \doi{10.1017/S1446788710000236}.

\bibitem[San18]{san::24}
T.~Sanders.
\newblock Coset decision trees and the {F}ourier algebra.
\newblock {\em J. Anal. Math.}, 2018,
  \burlalt{arXiv:1805.02168}{http://arxiv.org/abs/arXiv:1805.02168}.
\newblock To appear.

\bibitem[Shk08a]{shk::4}
I.~D. Shkredov.
\newblock On sets of large trigonometric sums.
\newblock {\em Izv. Ross. Akad. Nauk Ser. Mat.}, 72(1):161--182, 2008,
  \burlalt{arXiv:math/0605689}{http://arxiv.org/abs/arXiv:math/0605689}.
\newblock \doi{10.1070/IM2008v072n01ABEH002396}.

\bibitem[Shk08b]{shk::3}
I.~D. Shkredov.
\newblock On sets with small doubling.
\newblock {\em Mat. Zametki}, 84(6):927--947, 2008.
\newblock \doi{10.1134/S000143460811028X}.

\bibitem[Tao05]{tao::2}
T.~C. Tao.
\newblock Fourier analysis on finite non-{A}belian groups.
\newblock 2005.
\newblock \urlprefix\url{http://www.math.ucla.edu/~tao/247b.1.07w/notes10.pdf}.

\bibitem[Tao06]{tao::4}
T.~C. Tao.
\newblock Szemer\'edi's regularity lemma revisited.
\newblock {\em Contrib. Discrete Math.}, 1(1):8--28 (electronic), 2006.

\bibitem[Tao08]{tao::6}
T.~C. Tao.
\newblock Product set estimates for non-commutative groups.
\newblock {\em Combinatorica}, 28(5):547--594, 2008.
\newblock \doi{10.1007/s00493-008-2271-7}.

\bibitem[Tao10]{tao::9}
T.~C. Tao.
\newblock Freiman's theorem for solvable groups.
\newblock {\em Contrib. Disc. Math.}, 5(2):137--184, 2010,
  \burlalt{arXiv:0906.3535}{http://arxiv.org/abs/arXiv:0906.3535}.
\newblock \doi{10.11575/cdm.v5i2.62020}.

\bibitem[TV06]{taovu::}
T.~C. Tao and V.~H. Vu.
\newblock {\em Additive combinatorics}, volume 105 of {\em Cambridge Studies in
  Advanced Mathematics}.
\newblock Cambridge University Press, Cambridge, 2006.
\newblock \doi{10.1017/CBO9780511755149}.

\bibitem[{\"U}lg03]{ulg::}
A.~{\"U}lger.
\newblock A characterization of the closed unital ideals of the
  {F}ourier-{S}tieltjes algebra {$B(G)$} of a locally compact amenable group
  {$G$}.
\newblock {\em J. Funct. Anal.}, 205(1):90--106, 2003.
\newblock \doi{10.1016/S0022-1236(03)00143-5}.

\bibitem[Wey39]{wey::}
H.~Weyl.
\newblock {\em The {C}lassical {G}roups. {T}heir {I}nvariants and
  {R}epresentations}.
\newblock Princeton University Press, Princeton, N.J., 1939.

\end{thebibliography}

\end{document}